\documentclass[12pt]{amsart}

\usepackage[margin=1in,marginparwidth=0.8in, marginparsep=0.1in]{geometry}

\pdfoutput=1

\usepackage{amsthm,amsfonts,amsmath,amscd,amssymb,color,xcolor,graphicx}
\usepackage{epsfig}
\usepackage[all]{xy}
\usepackage{enumerate}
\usepackage{pinlabel}
\usepackage{hyperref}

\theoremstyle{plain}
\newtheorem{theorem}{Theorem}[section]

\newtheorem{proposition}[theorem]{Proposition}
\newtheorem{lemma}[theorem]{Lemma}
\theoremstyle{remark}
\newtheorem{remark}[theorem]{Remark}

\theoremstyle{definition}
\newtheorem{definition}[theorem]{Definition}

\def\R{\mathbb{R}}
\def\A{\mathcal{A}}
\def\M{\mathcal{M}}
\def\Z{\mathbb{Z}}
\def\d{\partial}
\def\LCH{\operatorname{LCH}}
\def\ix{\operatorname{ind}}
\def\F{\mathcal{F}}
\def\coker{\operatorname{coker}}
\def\Cord{\operatorname{Cord}}

\def\B{F}
\def\C{C}
\definecolor{orange}{HTML}{FFA500}
\title{A complete knot invariant from contact homology}
\author{Tobias Ekholm}
\author{Lenhard Ng}
\author{Vivek Shende}

\begin{document}

\begin{abstract}
We construct an enhanced version of knot contact homology,
and show that we can deduce from it the group ring of the knot group together with
the peripheral subgroup.   In particular, it completely
determines a knot up to smooth isotopy.  The enhancement consists of
the (fully noncommutative) Legendrian contact homology associated to the union of the conormal torus of the knot and a disjoint cotangent fiber sphere, along with a product on a filtered part of this homology.
As a corollary, we obtain a
new, holomorphic-curve proof of a result of the third author that the
Legendrian isotopy class of the conormal torus is a complete knot invariant. 
\end{abstract}

\maketitle

\section{Introduction}
\label{sec:intro}

\subsection{Conormal tori and knot contact homology}

A significant thread in recent research in symplectic and contact topology has concerned the study of smooth manifolds through the symplectic structures on their cotangent bundles. 
In this setting, one can also study a pair of manifolds, one embedded in the other---in particular, a knot in a $3$-manifold---via the conormal construction. If $K \subset \R^3$ is a knot,
then its unit conormal bundle, the conormal torus $\Lambda_K$,
is a Legendrian submanifold of the contact
cosphere bundle $ST^*\R^3$.
Isotopic knots produce conormal tori that are isotopic as Legendrian submanifolds,
i.e., the Legendrian isotopy type of the conormal torus is a knot invariant.
The fact that this invariant is nontrivial depends essentially on
the contact geometry:  the conormal tori of any two knots are smoothly isotopic, even if the knots themselves are not isotopic.

Symplectic field theory \cite{EGH} provides an algebraic knot invariant associated to this geometric invariant:
the Legendrian contact homology of $\Lambda_K$, also known as the
knot contact homology of $K$.
This is the homology of a differential graded algebra generated by Reeb chords of $\Lambda_K$ with differential given by counting holomorphic disks. In the past few years, there have been indications that knot contact homology and its higher genus generalizations are related via string theory to other knot invariants such as the A-polynomial, HOMFLY-PT polynomial, and possibly various knot homologies: the cotangent bundle equipped with Lagrangian branes along the conormal of the knot and the 0-section is the setting for an open topological string theory that has conjectural relations to all of these invariants. The physical account considers the holomorphic disks that go into knot contact homology, and also crucially takes into account higher genus information; this last part has not yet been fully developed in the mathematical literature, but some beginnings can be found in \cite{AENV}. In particular, it is explained there how certain quantum invariants should be conjecturally recovered from a quantization of knot contact homology arising from the consideration of non-exact Lagrangians. In any case, it appears that knot contact homology should be a very strong invariant, in the sense that it encodes a great deal of information about the underlying knot.

Recent work of the third author \cite{Shende} shows that the Legendrian conormal torus $\Lambda_K$ is in fact a complete invariant of $K$: two knots with Legendrian isotopic conormals must in fact be isotopic. Since $\Lambda_K$ is the starting point for knot contact homology,
this can be viewed as evidence for, or in any case is consistent with,
the possibility that knot contact homology itself is a complete invariant. Other
evidence in this direction is provided by the fact that knot contact homology recovers enough of the knot group (the fundamental group of the knot complement) to detect the unknot \cite{Ngframed} and torus
knots, among others \cite{GLid}. These results use the ring structure on the ``fully noncommutative'' version of knot contact homology, where the algebra is generated by
Reeb chords along with homology classes in $H_1(\Lambda_K)$ that do not commute with Reeb chords; see \cite{Ngsurvey,CELN}. However, the question of whether knot contact homology is a complete invariant remains open.

\subsection{Main results}

In this paper, we present an extension of knot contact homology
 by slightly enlarging the set of holomorphic disks that are counted.
We will show that this extension, which we call \textit{enhanced knot contact homology},
contains the knot group along with the peripheral subgroup, and this in turn is enough to completely determine the knot \cite{Wal,GL}. As a corollary, we have a new proof of the result from \cite{Shende}, using holomorphic curves rather than constructible sheaves.

For our purposes, we need the Legendrian contact homology of not just $\Lambda_K$
but the union of $\Lambda_K$ and a cotangent fiber $\Lambda_p$ of
$ST^*\R^3$; the inclusion of the latter is analogous to choosing a
basepoint for the fundamental group. This new invariant,
$\LCH_*(\Lambda_K \cup \Lambda_p)$, is a ring that
contains the knot contact homology of $K$ as a quotient. Using
the ``link grading'' of Mishachev \cite{Mishachev}, we can write:
\[
\LCH_*(\Lambda_K \cup \Lambda_p) \cong
(\LCH_*)_{\Lambda_K,\Lambda_K} \oplus
(\LCH_*)_{\Lambda_K,\Lambda_p} \oplus
(\LCH_*)_{\Lambda_p,\Lambda_K} \oplus
(\LCH_*)_{\Lambda_p,\Lambda_p},
\]
where $(\LCH_*)_{\Lambda_i,\Lambda_j}$ denotes the homology of the
subcomplex generated by composable words of Reeb chords ending on
$\Lambda_i$ and beginning on $\Lambda_j$; see Section~\ref{sec:enhanced}
for details.

From this set of data, we pick out what we
call the \textit{KCH-triple} $(R_{KK},R_{Kp},R_{pK})$ associated to
$\Lambda_K \cup \Lambda_p$, defined by:
\begin{align*}
R_{KK} &= (\LCH_0)_{\Lambda_K,\Lambda_K}, &
R_{Kp} &= (\LCH_0)_{\Lambda_K,\Lambda_p}, &
R_{pK} &= (\LCH_1)_{\Lambda_p,\Lambda_K}.
\end{align*}
Of these, $R_{KK}$ is precisely the degree $0$ knot contact
homology of $K$ and contains a subring $\Z[l^{\pm 1},m^{\pm 1}]$ once we
equip $K$ with an orientation and framing (which we choose to be the
Seifert framing), where $l,m$ denote the longitude and meridian of $K$; $R_{Kp}$ and $R_{pK}$ are left and right modules,
respectively, over $R_{KK}$. We remark that $(\LCH_*)_{\Lambda_K,\Lambda_K}$, $(\LCH_*)_{\Lambda_K,\Lambda_p}$, and $(\LCH_*)_{\Lambda_p,\Lambda_K}$ turn out to be supported in degrees $\geq 0$, $0$, and $1$, respectively, and so the KCH-triple is comprised of the lowest-degree summand of each.

We need one further piece of data in addition to
the KCH-triple: a product $\mu :\thinspace R_{Kp} \otimes R_{pK} \to
R_{KK}$. While the differential in $\LCH_*(\Lambda_K \cup \Lambda_p)$
counts holomorphic disks in the symplectization $\R \times ST^*\R^3$
with boundary on $\R\times (\Lambda_K \cup \Lambda_p)$ and one
positive puncture at a Reeb chord of $\Lambda_K \cup \Lambda_p$, the
product $\mu$ counts holomorphic disks with \textit{two} positive punctures, at
mixed Reeb chords of $\Lambda_K \cup \Lambda_p$. Extending
Legendrian contact homology to ``Legendrian Rational Symplectic Field Theory''
by counting disks with multiple positive punctures has not yet been
successfully implemented in general;
the difficulty comes from
boundary breaking for holomorphic disks, which contributes to the codimension-$1$ strata of moduli spaces. However, partial results in this direction have been
obtained by the first author \cite{E08} in the case of
multiple-component Legendrian links when boundary breaking can be avoided for topological reasons, and (with less relevance for our
purposes) by the second author \cite{NgSFT} in complete generality in the case of Legendrian
knots in $\R^3$. In particular, the fact that $\mu$ is well-defined and invariant follows from \cite{E08}.

Our main result is now as follows:

\begin{theorem}
Let $K \subset \R^3$ be an oriented knot and $p \in \R^3\setminus K$ be a point,
and let $\Lambda_K$, $\Lambda_p$ denote the Legendrian submanifolds of
$ST^*\R^3$ given by the unit conormal torus to $K$ and the unit
cotangent fiber over $p$.
\label{thm:main}
Then the KCH-triple $(R_{KK},R_{Kp},R_{pK})$ constructed from the
Legendrian contact homology of $\Lambda_K \cup \Lambda_p$, equipped with
the product $\mu :\thinspace R_{Kp} \otimes R_{pK} \to R_{KK}$,
is a complete invariant for $K$.

More precisely, if there is an isomorphism between the KCH-triples for
two oriented knots $K_0,K_1$ that preserves $\mu$, then:
\begin{enumerate}
\item
$K_0$ and $K_1$ are smoothly isotopic
 \label{it:main1}
up to mirroring and orientation reversal;
\item
if the isomorphism from $R_{K_0K_0}$ to $R_{K_1K_1}$ restricts to the
identity map on the subring
$\Z[l^{\pm 1},m^{\pm 1}]$,
\label{it:main2}
then $K_0$ and $K_1$ are smoothly isotopic
as oriented knots.
\end{enumerate}
\end{theorem}

A Legendrian isotopy between $\Lambda_{K_0} \cup \Lambda_p$ and $\Lambda_{K_1} \cup \Lambda_p$ induces an isomorphism between the KCH-triples that respects the product $\mu$. Since $\R^3$ is noncompact, any Legendrian isotopy between the conormal tori $\Lambda_{K_0}$ and $\Lambda_{K_1}$ can be extended to an isotopy between $\Lambda_{K_0} \cup \Lambda_p$ and $\Lambda_{K_1} \cup \Lambda_p$ by pushing $p$ away from the (compact) support of the isotopy. Thus we deduce from Theorem~\ref{thm:main} a new proof of the following result.

\begin{theorem}[\cite{Shende}]
Let $K_0,K_1$ be smooth knots in $\R^3$
\label{thm:complete}
and let $\Lambda_{K_0},\Lambda_{K_1}$ denote their conormal tori.
\begin{enumerate}
\item \label{it:cor1}
If $\Lambda_{K_0}$ and $\Lambda_{K_1}$ are Legendrian isotopic, then $K_0$ and $K_1$ are smoothly isotopic up to mirroring and orientation reversal.
\item \label{it:cor2}
If $\Lambda_{K_0}$ and $\Lambda_{K_1}$ are parametrized Legendrian isotopic, then $K_0$ and $K_1$ are smoothly isotopic as oriented knots.
\end{enumerate}
\end{theorem}

\noindent
Here ``parametrized Legendrian isotopic'' means the following: each conormal torus $\Lambda_K$ of an oriented knot $K$ has two distinguished classes in $H_1(\Lambda_K)$ given by the meridian and Seifert-framed longitude, and a parametrized Legendrian isotopy between conormal tori is an isotopy that sends meridian and longitude to meridian and longitude. 

Our proof of Theorem \ref{thm:main} depends crucially on the
results of \cite{CELN}, which relates knot contact homology to string topology. It is shown there that one can construct an isomorphism from degree $0$ knot contact homology, $\LCH_0(\Lambda_K)$, to a certain string homology constructed from paths (``broken strings'') on the singular Lagrangian given by the union, inside the cotangent bundle, of the zero section and the conormal. This isomorphism is induced by mapping a Reeb chord to the chain of boundaries of all holomorphic disks asymptotic to the Reeb chord with boundary on the singular Lagrangian.

In this paper, we extend the isomorphism from \cite{CELN} to show that the KCH-triple can also be computed using broken strings.
Using this presentation, we prove a ring isomorphism
\[
\Z[\pi_1(\R^3\setminus K)] \cong \Z \oplus (R_{pK} \otimes_{R_{KK}} R_{Kp})
\]
where multiplication on the right is induced by the product $\mu :\thinspace R_{Kp} \otimes R_{pK} \to R_{KK}$ (see Section~\ref{ssec:kchtriple} for details).
Knot groups are known to be left orderable \cite{HowieShort}, i.e., they have a total ordering invariant under left multiplication, and left orderable groups are determined by their group ring \cite{Higman}; it follows that we can recover the knot group itself from the KCH-triple. A further consideration of the subring $\Z[l^{\pm 1},m^{\pm 1}]$, which sits naturally in enhanced knot contact homology (more precisely, in $R_{KK}$), shows that we can also recover the longitude and meridian inside the knot group, and thus by \cite{Wal} we have a complete knot invariant.

We emphasize that the extra cosphere fiber is critical for our
argument. It is shown in \cite{CELN} that knot contact homology
$\LCH_0(\Lambda_K)$ is isomorphic to a certain subring of
$\Z[\pi_1(\R^3\setminus K)]$, and this can be used to prove that
$\LCH_0(\Lambda_K)$ detects the unknot and torus knots, as mentioned
earlier. It is not clear whether this subring suffices to give a
complete invariant. By contrast, the extra cosphere fiber allows the
direct recovery of $\Z[\pi_1(\R^3\setminus K)]$ and thus
$\pi_1(\R^3\setminus K)$.

\subsection{Relation to sheaves}

We conclude this introduction by sketching a Floer-theoretic path from the arguments
of \cite{Shende} to those of the present work.  The body of the paper does not depend
on any of the claims below; we include them solely for motivational purposes and conceptual clarity.
These claims could be established rigorously by a variant of \cite{BEE} in the partially wrapped context (a special case of \cite[Conjecture 3]{EL}), together with a 
proof of Kontsevich's localization conjecture \cite{kontsevich2009symplectic}.  
A significantly more detailed sketch of the following arguments appears in 
section 6 of the arXiv version of the present paper, \href{http://arxiv.org/abs/1606.07050}{arXiv:1606.07050}.

In \cite{Shende} the basic tool is the category of sheaves on $\R^3$, constructible with respect
to the stratification by the knot $K$ and its complement $\R^{3}\setminus K$.  This category is identified with the \emph{infinitesimal Fukaya category}
whose objects are, roughly, exact Lagrangians in $T^* \R^3$ asymptotic to the conormal torus $\Lambda_{K}\subset ST^{\ast}\R^{3}$, and whose
morphisms are the intersections between the Lagrangians after perturbing infinitesimally along the Reeb flow at infinity \cite{Nadler-Zaslow, Nad}.  

There is another Floer-theoretic category one can associate to the same geometry, the \emph{partially wrapped category with wrapping stopped by $\Lambda_{K}$}. Here the objects are exact Lagrangians 
asymptotic to Legendrian submanifolds in $ST^{\ast} \R^{3}$, in the complement of the conormal torus $\Lambda_{K}$. The morphisms are computed by wrapping using 
a Reeb flow which stops at $\Lambda_{K}$ in the sense of \cite{Sylvan}. (A cut and paste model of the Reeb flow is obtained by attaching $T^{\ast}([0,\infty)\times \Lambda_{K})$ to $ST^{\ast}\R^{3}$ along $\Lambda_{K}$, see \cite[Section B.3]{EL}.) 

The infinitesimally wrapped category embeds into the partially wrapped category: pushing a Lagrangian asymptotic to $\Lambda_{K}$ slightly backwards along the Reeb flow gives a Lagrangian with trivial wrapping at infinity. To see this, note that the Reeb flow starting at the shifted $\Lambda_{K}$ arrives immediately at the stop $\Lambda_{K}$ and hence will flow no further.  The image
of this embedding is expected to be categorically characterized as the ``pseudo-perfect modules''.  
In particular, the partially wrapped category should know at least as much as the sheaf category.  

Two notable objects of the partially wrapped category are the cotangent fiber $\B$ at a point not on the knot and the Lagrangian disk $\C$ which fills a small ball linking the conormal torus. (In the cut and paste model, $\C$ is a cotangent fiber in $T^{\ast}([0,\infty)\times\Lambda_{K})$.) Taking Hom with these Lagrangians gives functors
from the partially wrapped category, hence by restriction followed by the Nadler-Zaslow isomorphism,
from the sheaf category, to chain complexes.  

In fact, the partially wrapped category also has a conjectural identification with a certain
category of sheaves \cite{kontsevich2009symplectic, Nad-wrapped}.  Under these identifications, the functor associated to $\B$ is computing the stalk at the point away from the knot, and the functor associated to $\C$ is computing the microsupport of the sheaf at the knot.  
These are the main operations used in \cite{Shende} and having both is crucial to the argument there.  

It is also expected 
that the Lagrangians $\B$ and $\C$ generate the partially wrapped Fukaya category, i.e., the partially wrapped category can be identified with the category
of perfect modules over the endomorphism algebra of $F\cup C$. This means that the partially wrapped Floer cohomology $HW^{\ast}(\B\cup \C,\B\cup \C)$ of these two disks should contain all the information of the sheaf category, and
moreover in a way which makes the information needed in the arguments of \cite{Shende} immediately accessible.

Both wrapped Floer cohomology and Legendrian contact homology are algebras on Reeb chords; a precise relation between
them is established in \cite{BEE} and generalized to the partially wrapped context in \cite[Conjecture 3]{EL}. Specifically, the partially wrapped Floer cohomologies of the disks $\B$ and $\C$ can be computed from contact homology algebras and in the notation above, we have:
\begin{align*}
HW^1(\C,\C) &\cong R_{KK}, \\
HW^1(\C,\B) &\cong R_{pK}, \\
HW^1(\B,\C) &\cong R_{Kp}, \\
HW^1(\B,\B) &\cong \Z \oplus (R_{pK} \otimes R_{Kp}).
\end{align*}
Moreover, the product $\mu :\thinspace R_{Kp} \otimes R_{pK} \to R_{KK}$ is identified with
the ordinary pair-of-pants product $HW^{1}(\B,\C)\otimes HW^{1}(\C,\B)\to HW^{1}(\C,\C)$ in wrapped Floer cohomology. Then the KCH-triple and product determines a ring structure on $HW^1(\B,\B)$, and our results show that $HW^1(\B,\B)$ is ring isomorphic to the group ring $\Z[\pi_1(\R^3\setminus K)]$.

In fact, this ring isomorphism can be induced from moduli spaces of holomorphic disks as follows. Applying Lagrange surgery to $\R^3\cup L_K$ (i.e., removing the interiors of small disk bundle neighborhoods of $K$ in $\R^{3}$ and $L_{K}$ and joining the resulting boundary fiber circles over $K$ by a family of $1$-handles), we obtain a Lagrangian $M_K$ with the topology of $\R^{3}\setminus K$. The disk $\B$ intersects $M_K$ transversely in one point and the map above is induced from moduli spaces of holomorphic disks with one positive puncture at a Reeb chord of $\B$, two Lagrangian intersection punctures at $M_K\cap \B$, and boundary on $\B\cup M_K$. This map is then directly analogous to the corresponding map in the cotangent bundle of a closed manifold and, as there, it gives an isomorphism of rings, intertwining the pair of pants product in wrapped Floer chomology with the Pontryagin product on chains of loops. 

\subsection{Outline of the paper}

In Section~\ref{sec:enhanced}, we introduce enhanced knot contact homology and the KCH-triple, along with the product map $\mu$. We reformulate these structures in terms of string topology in Section~\ref{sec:stringtop} and then in terms of the knot group in Section~\ref{sec:homotopy}, leading to a proof of Theorem~\ref{thm:main} in Section~\ref{sec:mainpf}. In Section~\ref{sec:wrapped}, we discuss the relation to wrapped Floer cohomology and sheaves. Please note that Section~\ref{sec:wrapped}, which is speculative in nature, is included in the arXiv version of this paper, but not in the published version.

\subsection*{Acknowledgments}

TE was supported by the Knut and Alice Wallenberg Foundation and the Swedish Research Council.
LN was partially supported by NSF grant DMS-1406371 and a grant from the Simons Foundation (\# 341289 to Lenhard Ng).
VS was partially supported by NSF grant DMS-1406871 and a Sloan Fellowship.


\section{Enhanced Knot Contact Homology}
\label{sec:enhanced}
In this section we present the ingredients of enhanced knot contact homology. In Section \ref{ssec:LCHlink} we discuss the structure of the contact homology algebra of a two component Legendrian link, in Section \ref{ssec:enhanced} we specialize to the case of a link consisting of the conormal of a knot and the fiber sphere over a point. Finally, in Section \ref{ssec:product} we introduce the product operation on enhanced knot contact homology.

\subsection{Legendrian contact homology for a link}
\label{ssec:LCHlink}

Let $V = J^1(M)$ be the $1$-jet space of a compact manifold $M$ with
the standard contact structure, and let $\Lambda \subset V$ be a connected
Legendrian submanifold.
The Legendrian contact homology of $\Lambda$, which we will write as $\LCH_*(\Lambda)$, is the homology of a
differential graded algebra $(\A_\Lambda,\d)$, where $\A_\Lambda$ is a
noncommutative unital algebra generated by Reeb chords of $\Lambda$
and homology classes in $H_1(\Lambda)$, with the differential given by
a count of certain holomorphic curves in the symplectization $\R\times V$ with
boundary on $\R\times\Lambda$. 

\begin{remark}
For a Legendrian submanifold $\Lambda$ of a general contact manifold $V$ the Legendrian algebra $\A_{\Lambda}$ is an algebra generated by both Reeb chords and closed Reeb orbits, where the orbits generate a (super)commutative subalgebra. In the case of a 1-jet space there are no closed Reeb orbits and the algebra and its differential involves chords only.
\end{remark}

\begin{remark}\label{r:coeffs}
Legendrian contact homology is often defined with coefficients in
the group ring of $H_2(V,\Lambda)$ rather than $H_1(\Lambda)$, the
difference being whether one associates to a holomorphic disk its
relative homology class in $H_2(V,\Lambda)$ or the homology class of
its boundary in $H_1(\Lambda)$. In the case of knot contact
homology, our setup amounts to specializing to $U=1$ in the language
of \cite{EENStransverse,Ngtransverse,Ngsurvey} or $Q=1$ in the language of \cite{AENV}.
Also, as mentioned in the introduction, the version of the DGA that we
consider here is the fully noncommutative DGA, in which homology
classes in $H_1(\Lambda)$ do not commute with Reeb chords. To get loops rather than paths we fix a base point in each component of $\Lambda$ and capping paths connecting the base point to each Reeb chord endpoint, see Figure \ref{fig:disks}.
\end{remark}

If $\Lambda \subset V$ is a disconnected Legendrian submanifold, then
there is additional structure on the DGA of $\Lambda$ first
described by Mishachev \cite{Mishachev}; in modern language this is
the ``composable algebra'', and we follow the treatment from
\cite{BEE,EENS,NRSSZ}. For simplicity we restrict to
the case $\Lambda = \Lambda_1 \cup \Lambda_2$. For $i,j=1,2$, let
$\mathcal{R}^{ij}$ denote the set of Reeb chords that \textit{end} on
$\Lambda_i$ and \textit{begin} on $\Lambda_j$. The composable algebra
$\A_{\Lambda_1 \cup \Lambda_2}$ is the noncommutative $\Z$-algebra
generated by Reeb chords of $\Lambda_1 \cup \Lambda_2$, elements of
$\Z[H_1(\Lambda_1)]$, elements of $\Z[H_1(\Lambda_2)]$, and two
idempotents $e_1,e_2$, subject to
the relations (where $\delta_{ij}$ is the Kronecker delta):
\begin{itemize}
\item
$e_i e_j = \delta_{ij}$
\item
$e_{i'} a = \delta_{ii'} a$ and $a e_{j'} = \delta_{jj'} a$ for $a \in \mathcal{R}^{ij}$
\item
$e_j \alpha = \alpha e_j = \delta_{ij} \alpha$ for $\alpha \in
\Z[H_1(\Lambda_i)]$.
\end{itemize}
Note that $\A_{\Lambda_1 \cup \Lambda_2}$ is unital with unit
$e_1+e_2$. For $i,j=1,2$, define
$\A_{\Lambda_i,\Lambda_j} = e_i \A_{\Lambda_1 \cup \Lambda_2} e_j$;
then
\[
\A_{\Lambda_1 \cup \Lambda_2} = \bigoplus_{i,j\in\{1,2\}}
\A_{\Lambda_i,\Lambda_j}.
\]
In more concrete terms, $\A_{\Lambda_1 \cup \Lambda_2}$ is generated
as a $\Z$-module by monomials of the form
\[
\alpha_0 a_1 \alpha_1 a_2 \cdots a_n \alpha_n,
\]
where there is some sequence $(i_0,\ldots,i_n)$ with $i_k \in \{1,2\}$
such that $\alpha_k \in \Z[H_1(\Lambda_{i_k})]$ and $a_k \in
\mathcal{R}^{i_{k-1}i_k}$ for all $k$ (and one empty monomial $e_{i}$ for each component $\Lambda_{i}$). Monomials of this form are the ``composable
words''. Generators of
$\A_{\Lambda_i,\Lambda_j}$ are of the same form but specifically with $i_0=i$ and $i_n=j$.
Note that multiplication $\A_{\Lambda_i,\Lambda_j} \otimes
\A_{\Lambda_{i'},\Lambda_{j'}}\to \A_{\Lambda_{i},\Lambda_{j'}}$ is concatenation if $j=i'$ and $0$ otherwise.

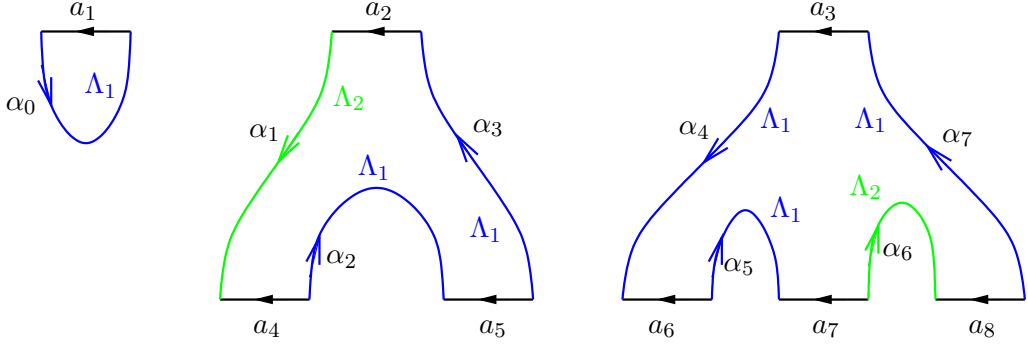
\begin{figure}
\labellist
\small\hair 2pt
\pinlabel $a_1$ at 19 137
\pinlabel $a_2$ at 138 137
\pinlabel $a_3$ at  317 137
\pinlabel $a_4$ at 93 9
\pinlabel $a_5$ at 184 9
\pinlabel $a_6$ at 252 9
\pinlabel $a_7$ at 318 9
\pinlabel $a_8$ at 381 9
\pinlabel $\alpha_0$ at -6 99
\pinlabel $\alpha_1$ at 92 88
\pinlabel $\alpha_2$ at 123 37
\pinlabel $\alpha_3$ at 182 90
\pinlabel $\alpha_4$ at 265 89
\pinlabel $\alpha_5$ at 283 34
\pinlabel $\alpha_6$ at 347 40
\pinlabel $\alpha_7$ at 371 86
\pinlabel ${\color{blue} \Lambda_1}$ at 26 107
\pinlabel ${\color{blue} \Lambda_1}$ at 136 74
\pinlabel ${\color{blue} \Lambda_1}$ at 181 49
\pinlabel ${\color{blue} \Lambda_1}$ at 298 93
\pinlabel ${\color{blue} \Lambda_1}$ at 302 57
\pinlabel ${\color{blue} \Lambda_1}$ at 336 93
\pinlabel ${\color{green} \Lambda_2}$ at 126 102
\pinlabel ${\color{green} \Lambda_2}$ at 334 66
\endlabellist
\centering
\includegraphics[width=0.8\textwidth]{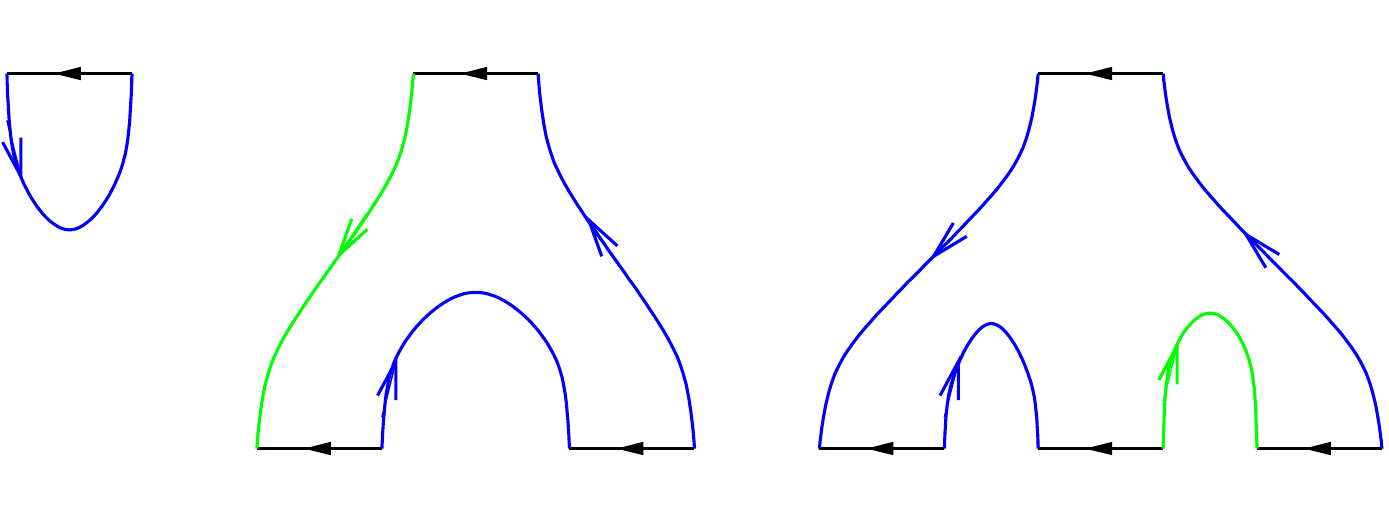}
\caption{
Terms contributing to the differential of $a_1 \in \mathcal{R}^{11}$, $a_2 \in \mathcal{R}^{21}$, $a_3 \in \mathcal{R}^{11}$: $\alpha_0$, $\alpha_1 a_4 \alpha_2 a_5 \alpha_3$, $\alpha_4 a_6 \alpha_5 a_7 \alpha_6 a_8 \alpha_7$, respectively. Here $a_5,a_6 \in \mathcal{R}^{11}$, $a_4,a_8 \in \mathcal{R}^{21}$, $a_7 \in \mathcal{R}^{12}$, $\alpha_1,\alpha_6 \in H_1(\Lambda_1)$, and $\alpha_0,\alpha_2,\alpha_3,\alpha_4,\alpha_5,\alpha_7 \in H_1(\Lambda_2)$. The boundaries of the disks lie on $\R\times \Lambda_1$ and $\R\times \Lambda_2$ as shown, and represent the homology classes indicated after closing up with suitably oriented capping paths, see Remark \ref{r:coeffs}. Small arrows denote orientations on Reeb chords.
}
\label{fig:disks}
\end{figure}

The differential $\partial$ on $\A_{\Lambda_1 \cup \Lambda_2}$
is defined to be $0$ on $e_i$ and on elements of $\Z[H_1(\Lambda_i)]$
and is given by a holomorphic-disk count for Reeb chords of $\Lambda_1
\cup \Lambda_2$. For a Reeb chord $a\in\mathcal{R}^{ij}$ the disks counted in $\partial a$ are maps into $\R\times V$ and have boundary on $(\R\times\Lambda_1)\cup(\R\times\Lambda_2)$, one positive puncture where it is asymptotic to $\R\times a$ and several negative punctures. The contribution to the differential is the composable word of homology classes and Reeb chords in the complement of the positive puncture along the boundary of the disk, see Figure~\ref{fig:disks}.
The differential thus respects the direct-sum decomposition $\bigoplus_{i,j\in\{1,2\}}
\A_{\Lambda_i,\Lambda_j}$, and this decomposition descends to the
homology:
\[
\LCH_*(\Lambda_1 \cup \Lambda_2) = H_*(\A_{\Lambda_1 \cup
  \Lambda_2},\partial) =
\bigoplus_{i,j\in\{1,2\}} (\LCH_*)_{\Lambda_i,\Lambda_j},
\]
where $(\LCH_*)_{\Lambda_i,\Lambda_j} = H_*(\A_{\Lambda_i,\Lambda_j},\partial)$.
Recall that Legendrian isotopies induce isomorphisms on Legendrian contact homology via counts of holomorphic disks similar to the differential, see \cite{E08,EES2,EHK}. It follows that a Legendrian isotopy between $2$-component Legendrian links induces a quasi-isomorphism between the DGAs that also respects the decomposition.

We can further refine the structure of $\A_{\Lambda_1 \cup \Lambda_2}$
by considering the filtration
\[
\A_{\Lambda_1 \cup \Lambda_2} = \F^0 \A_{\Lambda_1 \cup \Lambda_2}
\supset \F^1 \A_{\Lambda_1 \cup \Lambda_2} \supset \F^2 \A_{\Lambda_1
  \cup \Lambda_2} \supset \cdots
\]
where $\F^k \A_{\Lambda_1 \cup \Lambda_2}$ is the subalgebra generated
as a $\Z$-module by words involving at least $k$ mixed chords (Reeb
chords either from $\Lambda_1$ to $\Lambda_2$ or from $\Lambda_2$ to
$\Lambda_1$). This also gives a filtration on the
summands of $\A_{\Lambda_1 \cup \Lambda_2}$:
\begin{align*}
\A_{\Lambda_1,\Lambda_1} &= \F^0 \A_{\Lambda_1,\Lambda_1}
\supset \F^2 \A_{\Lambda_1,\Lambda_1} \supset
\F^4 \A_{\Lambda_1,\Lambda_1} \supset \cdots \\
\A_{\Lambda_1,\Lambda_2} &= \F^1 \A_{\Lambda_1,\Lambda_2}
\supset \F^3 \A_{\Lambda_1,\Lambda_2} \supset
\F^5 \A_{\Lambda_1,\Lambda_2} \supset \cdots \\
\A_{\Lambda_2,\Lambda_1} &= \F^1 \A_{\Lambda_2,\Lambda_1}
\supset \F^3 \A_{\Lambda_2,\Lambda_1} \supset
\F^5 \A_{\Lambda_2,\Lambda_1} \supset \cdots \\
\A_{\Lambda_2,\Lambda_2} &= \F^0 \A_{\Lambda_2,\Lambda_2}
\supset \F^2 \A_{\Lambda_2,\Lambda_2} \supset
\F^4 \A_{\Lambda_2,\Lambda_2} \supset \cdots.
\end{align*}

We note two properties of the filtration. First, it is compatible with
multiplication: the product of elements of $\F^{k_1}$ and $\F^{k_2}$ is
an element of $\F^{k_1+k_2}$. Second, the differential $\partial$ respects the
filtration, since the differential of any mixed chord is a sum of
words that each includes a mixed chord. As a consequence of this
second property, there is an induced filtration on $\LCH_*(\Lambda_1
\cup \Lambda_2)$ as well as its summands $(\LCH_*)_{\Lambda_i,\Lambda_j}(\Lambda_1
\cup \Lambda_2)$.

We abbreviate successive filtered quotients as follows: for $k$ even when $i=j$ and $k$ odd when $i\neq j$, write
\[
\A_{\Lambda_i,\Lambda_j}^{(k)} := \F^k \A_{\Lambda_i,\Lambda_j} / \F^{k+2} \A_{\Lambda_i,\Lambda_j}.
\]
Then $\A_{\Lambda_i,\Lambda_j}^{(k)}$ is generated as a $\Z$-module by words with exactly $k$ mixed chords. We will especially be interested in the following filtered quotients with their induced differentials:
\begin{itemize}
\item
$\A_{\Lambda_i,\Lambda_i}^{(0)}$, which is the DGA of $\Lambda_i$ itself;
\item
$\A_{\Lambda_i,\Lambda_j}^{(1)}$ with $i\neq j$, which is generated by words with exactly $1$ mixed chord.
\end{itemize}
Note that for $i\neq j$, the DGAs of $\Lambda_i$ and
of $\Lambda_j$ act on $\A_{\Lambda_i,\Lambda_j}^{(1)}$ on the left and right, respectively, by
multiplication, and this gives $\A_{\Lambda_i,\Lambda_j}^{(1)}$ the structure of a differential bimodule.

\subsection{Legendrian contact homology for the conormal and fiber}
\label{ssec:enhanced}

We now restrict to the case where $V$ is the contact manifold
$ST^*\R^3 = J^1(S^2)$. If $K \subset \R^3$ is a knot and $\Lambda_{K}
\subset V$
is the unit conormal bundle of $K$, then the knot contact
homology of $K$ is defined to be the Legendrian
contact homology of $\Lambda_K$:
\[
\LCH_*(\Lambda_K) = H_*(\A_{\Lambda_K},\d).
\]
The conormal $\Lambda_K$ is topologically a $2$-torus and has trivial Maslov class. The triviality of the Maslov class gives a well-defined integer grading on $\A_{\Lambda_{K}}$, by the Conley--Zehnder index, see \cite{EES2,EENS}. A choice of
orientation for $K$ gives a distinguished set of generators $l,m \in
\pi_1(\Lambda_K) \cong \Z^2$, where $m$ is the meridian and $l$ is the
Seifert-framed longitude. The group ring $\Z[H_1(\Lambda_K)] \cong
\Z[l^{\pm 1},m^{\pm 1}]$ is a subring of $\A_{\Lambda_K}$ in degree
$0$, and there is an induced map $\Z[l^{\pm 1},m^{\pm 1}] \to HC_0(K)$
that is injective as long as $K$ is not the unknot (see \cite{CELN}).

Since the Reeb flow on $ST^*\R^3$ is the geodesic flow, Reeb
chords correspond under the projection $ST^*\R^3 \to \R^3$
in a one to one fashion to oriented binormal chords of $K$: for the flat metric on $\R^{3}$ these are simply oriented line segments with endpoints on $K$ that
are perpendicular to $K$ at both endpoints. Furthermore the Conley--Zehnder grading of such a chord agrees with the Morse index for the corresponding critical point of the distance function $K \times K \to \R$, and hence takes on only the values $0,1,2$, see \cite[Section 3.3.3]{EENS}.

Next suppose that in addition to the knot $K$, we choose a point $p
\in \R^3 \setminus
K$. Then we can form the Legendrian link $\Lambda = \Lambda_K \cup
\Lambda_p \subset ST^*\R^3$, where $\Lambda_K$ is the unit conormal to $K$ as
before and $\Lambda_p$ is the unit cotangent fiber of $ST^*\R^3$ at
$p$.

Let $(\A_{\Lambda_K \cup \Lambda_p},\partial)$ be the DGA associated
to the link $\Lambda_K \cup \Lambda_p$. Then $\A_{\Lambda_K \cup \Lambda_p}$ is generated by Reeb
chords of $\Lambda_K \cup \Lambda_p$ (along with homology classes). There are no Reeb chords from
$\Lambda_p$ to itself, and so the Reeb chords of $\Lambda_K \cup
\Lambda_p$ come in three types:
from $\Lambda_K$ to itself, to $\Lambda_K$ from
$\Lambda_p$, and to $\Lambda_p$ from $\Lambda_K$. These all correspond
to binormal chords of $K \cup \{p\}$, where the normality condition is
trivial at $p$.

We now discuss the grading on $\A_{\Lambda_K \cup \Lambda_p}$. Homology classes are graded by $0$. The
grading on pure Reeb chords from $\Lambda_K$ to itself is as for
$\A_{\Lambda_K}$. In order to define the grading of mixed Reeb chords
of a two-component
Legendrian submanifold of Maslov index $0$ such as
$\Lambda_K\cup\Lambda_p$, it is customary to choose a path connecting
the two Legendrians, along with a continuous field of Legendrian
tangent planes along this path (i.e., isotropic 2-planes in the contact hyperplanes along the path) interpolating between the tangent
planes to the Legendrians at the two endpoints. There is a $\Z$'s
worth of homotopy classes of such fields of tangent planes, and
different choices affect the grading of mixed chords, shifting the
grading of chords from $\Lambda_K$ to $\Lambda_p$ up by some uniform
constant $k$ and shifting the grading of chords from $\Lambda_p$ to
$\Lambda_K$ down by $k$. Note here that the usual dimension formulas
for holomorphic disks hold and are independent of the path chosen
since for any actual disk the path is traversed algebraically zero
times.

To assign a specific grading to mixed chords, it is convenient to
place $K$ and $p$ in a specific configuration in $\R^3$. Let $(x,y,z)$ be linear coordinates on $\R^{3}$. The unit
circle in the $xy$ plane is an unknot, and we can braid $K$ around this
unknot so that it lies in a small tubular neighborhood of the circle;
also, choose $p$ to lie in the $xy$ plane, outside a disk containing the projection of $K$. If we view
$\Lambda_K,\Lambda_p \subset J^1(S^2)$ as fronts in $J^{0}(S^{2})=S^2\times\R$,
then the front of $\Lambda_p$ is the graph of the function $v \mapsto
p \cdot v$ for $v \in S^2$, and in particular the tangent planes to
this front over the equator $\{z=0\}\cap S^{2}$ are horizontal. On the
other hand, if the braid for $K$ has $n$ strands, then the front of
$\Lambda_K$ has $2n$ sheets near the equator, $n$ with positive $\R$-coordinate and $n$ with negative, and the tangent planes to these sheets are nearly horizontal. We can now take the connecting
path between $\Lambda_p$ and $\Lambda_K$ as follows: choose a point $v$ in the equator with $p \cdot v < 0$, and over $v$ join the unique point in the front of $\Lambda_p$ to any of the $n$ points in $\Lambda_K$ with negative $\R$-coordinate, see Figure \ref{fig:mixedchords}. The tangent planes are horizontal at the $\Lambda_p$ endpoint and nearly horizontal at the $\Lambda_K$ endpoint; choose the path to consist of nearly horizontal planes over $v$ joining these without rotation.

\begin{figure}
\labellist
\small\hair 2pt
\pinlabel ${\color{green} \Lambda_p}$ at 191 318
\pinlabel ${\color{blue} \Lambda_K}$ at 7 356
\pinlabel $S^2$ at 154 290
\pinlabel $c'$ at 106 249
\pinlabel $c$ at 106 259
\pinlabel $c'''$ at 380 233
\pinlabel $c''$ at 380 247
\pinlabel ${\color{red} \ast_p}$ at 110 235
\pinlabel ${\color{red} \ast_K}$ at 19 235
\pinlabel ${\color{red} \gamma_0}$ at 83 229
\pinlabel ${\color{green} p}$ at 308 35
\pinlabel ${\color{blue} K}$ at 93 71 
\pinlabel $\gamma'$ at 194 56
\pinlabel $\gamma$ at 194 21
\pinlabel $\gamma'''$ at 246 56
\pinlabel $\gamma''$ at 246 21
\endlabellist
\centering
\includegraphics[width=0.65\textwidth]{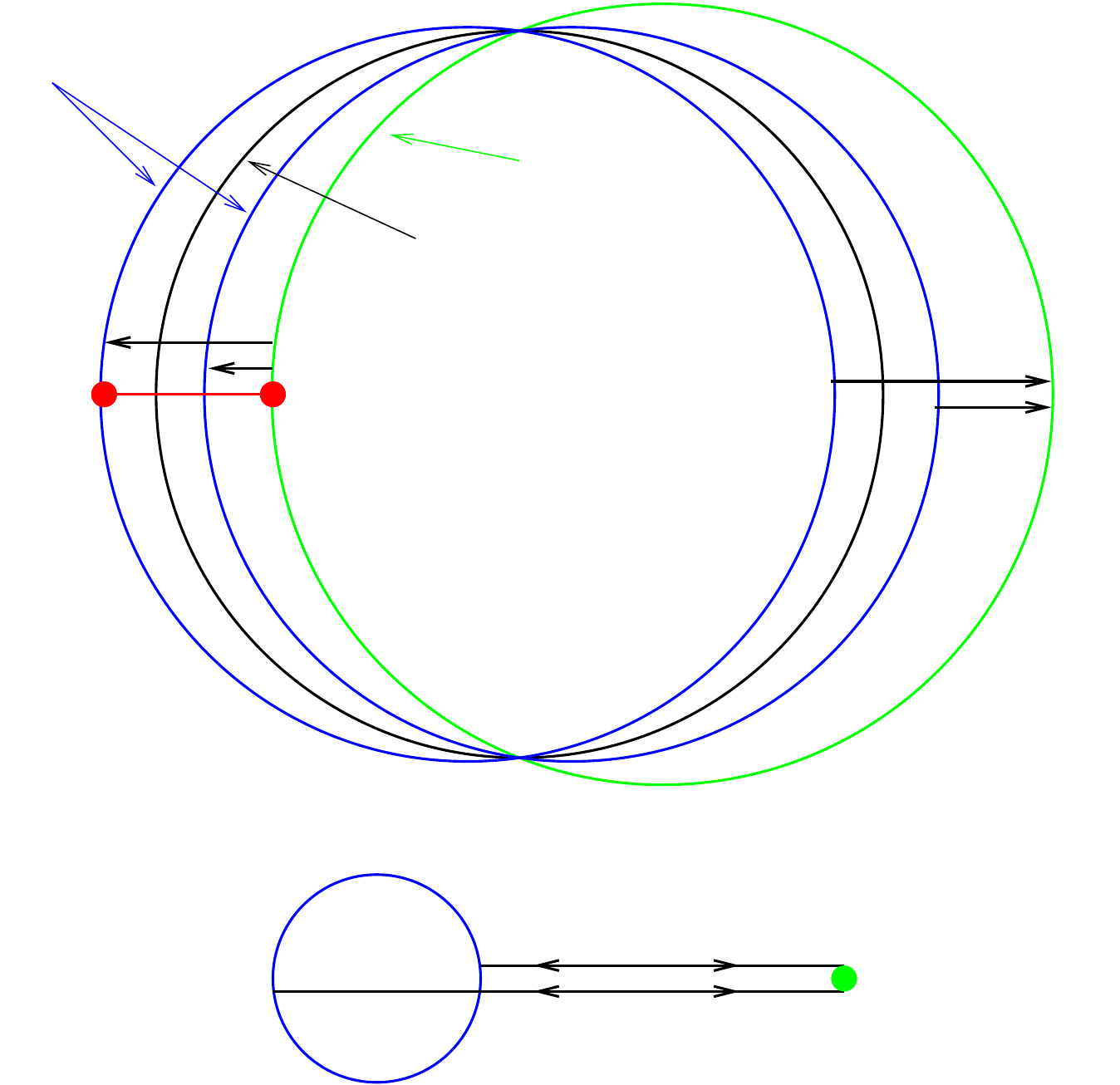}
\caption{Degrees of mixed chords $c,c',c'',c'''$ between the fronts of $\Lambda_K$ and $\Lambda_p$ in $S^2\times\R$, and the corresponding binormal chords $\gamma,\gamma',\gamma'',\gamma'''$ in $\R^3$. The connecting path $\gamma_0$ between $\Lambda_K$ and $\Lambda_p$ is also shown, with endpoints $\ast_K,\ast_p$.}
\label{fig:mixedchords}
\end{figure}
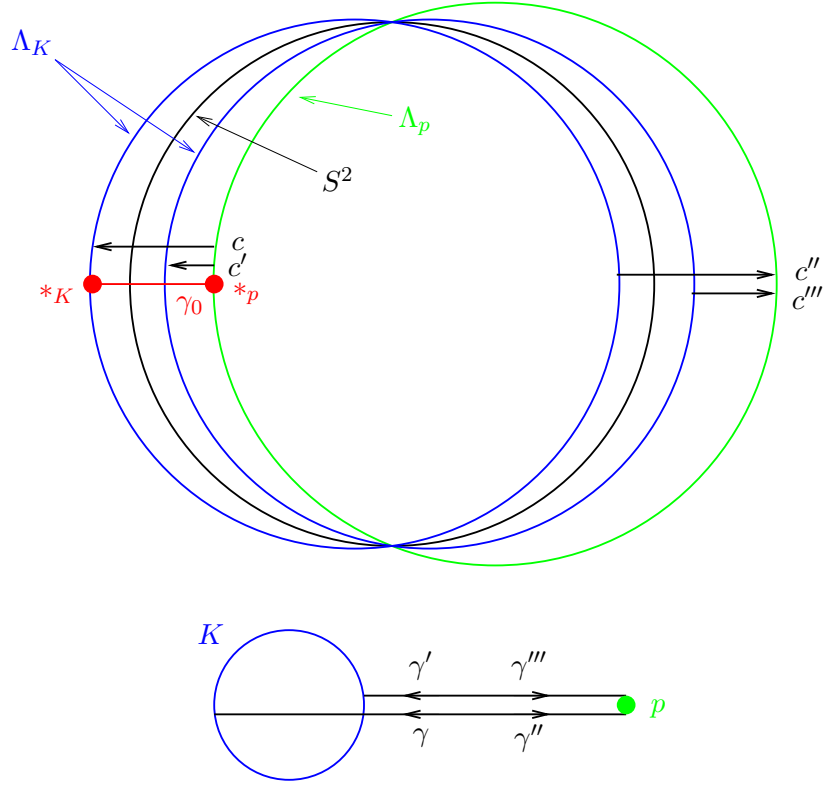

\begin{proposition}
With this choice of configuration, the Reeb chords of $\Lambda_K \cup \Lambda_p$ have grading as follows. Let $\gamma$ be a binormal chord of $K \cup \{p\}$ corresponding to a Reeb chord $c$ of $\Lambda_K \cup \Lambda_p$. Let ``$\;\ix$'' denote the Morse index of the critical point corresponding to $\gamma$ for the distance function on $K \cup
\{p\}$. Then:
\label{prop:degrees}

\begin{itemize}
\item if $c \in \mathcal{R}^{KK}$ ($c$ goes to $\Lambda_{K}$ from $\Lambda_{K}$)  then
\[
|c|=\ix(\gamma);
\]
\item if $c \in\mathcal{R}^{Kp}$ ($c$ goes to $\Lambda_{K}$ from $\Lambda_{p}$) then
\[
|c|=\ix(\gamma);
\]
\item if $c \in \mathcal{R}^{pK}$ ($c$ goes to $\Lambda_{p}$ from $\Lambda_{K}$) then
\[
|c|=\ix(\gamma)+1.
\]
\end{itemize}
\end{proposition}

\begin{proof}
We begin with mixed Reeb chords between $\Lambda_K$ and $\Lambda_p$ in either direction.
For these, we can use \cite[Lemma~2.5]{EENS} (cf.\ \cite[Lemma~3.4]{EESnonisotopic}), which writes the degree $|c|$ of a Reeb chord $c$ between two sheets of a front projection in terms of the Morse index $\ix_{\mathrm{loc}}$ of the difference between the functions corresponding to these two sheets, and the difference $D-U$ between the number of up and down cusps along a capping path for the chord, as 
\[
|c|=\ix_{\mathrm{loc}}+(D-U)-1.
\] 
In our case $\ix_{\mathrm{loc}}$ is $2$ for all mixed Reeb chords: the difference functions between the sheets near the Reeb chords look roughly like the difference function between the front of $\Lambda_{p}$ and the $0$-section and hence has local maxima, see Figure \ref{fig:mixedchords}. 

To count up and down cusps, we recall the definition of capping path. Let $\ast_K,\ast_p$ denote the endpoints of the fixed path $\gamma_0$ connecting $\Lambda_K$ and $\Lambda_p$. If $c$ is a mixed chord of $\Lambda_K \cup \Lambda_p$, then the capping path for $c$ is given as follows,  cf.\ \cite[Lemma~2.5]{EENS}: if $c$ goes to $\Lambda_K$ (respectively $\Lambda_p$) from $\Lambda_p$ ($\Lambda_K$), then take the union of a path in $\Lambda_K$ ($\Lambda_p$) from the endpoint of $c$ to $\ast_K$ ($\ast_p$) and a path in $\Lambda_p$ ($\Lambda_K$) from $\ast_p$ ($\ast_K$) to the beginning point of $c$. Any capping path that passes through the north or south pole of $S^2$ traverses an up cusp if it goes from a negative sheet of $\Lambda_K$ to a positive sheet, and a down cusp if it goes in the opposite direction; see \cite[section~3.1]{EENS}.

There are four types of mixed chords, which we denote by $c,c',c'',c'''$ as shown in Figure~\ref{fig:mixedchords}.
The longer chords $c$ (with corresponding binormal chord $\gamma$) from $\Lambda_{p}$ to $\Lambda_{K}$ begin near $\ast_p$ and end on the sheet near $\ast_K$; the capping path for $c$ can be chosen to avoid the poles of $S^2$, and so the degree of $c$ is $|c|= 2-1=1=\ix(\gamma)$. The shorter chords $c'$ (with corresponding binormal chord $\gamma'$) from $\Lambda_p$ to $\Lambda_K$ end on one of the negative sheets of $\Lambda_K$; the capping path for $c'$ passes from a negative sheet to a positive sheet of $\Lambda_K$ through one of the poles, traversing one up cusp in the process, and so 
$|c'|=2-1-1=0=\ix(\gamma')$. For the mixed chords $c'',c'''$ from $\Lambda_K$ to $\Lambda_p$ with binormal chords $\gamma'',\gamma'''$, similar computations give $|c''|=2+1-1=\ix(\gamma'')+1$ and $|c'''|=2-1=\ix(\gamma''')+1$. This establishes the result for mixed chords.

For pure chords the calculation is similar; we give a brief description and refer to \cite[Lemma 3.7]{EENS} for details. There are the longer chords corresponding to the chords of the unknot: for the round unknot there is an $S^{1}$ Bott-family of chords which after perturbation gives rise to two chords. We write $c$ (with corresponding binormal chord $\gamma$) and $e$ (with corresponding binormal chord $\epsilon$) to denote a chord of $K$ corresponding to the shorter and longer chord of the unknot, respectively. The local index at $e$ (respectively $c$) is $2$ ($1$), and a path connecting the endpoint to the start point has one down cusp. This gives $|e|=2+1-1=2=\ix(\epsilon)$, and $|c|=1+1-1=1=\ix(\gamma)$. Finally, there are short chords of $K$ that are contained in a tubular neighborhood of the unknot. These are of two types, depending on whether the underlying binormal chord has Morse index $0$ or $1$. Let $a$ (with corresponding binormal chord $\alpha$) be of the former type and $b$ (with corresponding binormal chord $\beta$) of the latter. Noting that there are paths connecting their start and endpoints without cusps and that the local index is $1$ for $a$ and $2$ for $b$, it follows that $|a|=1-1=0=\ix(\alpha)$ and $|b|=2-1=1=\ix(\beta)$. The formulas relating degrees and indices thus hold for all types of chords.
\end{proof}

From Proposition~\ref{prop:degrees}, since $\ix(\gamma)$ is in $\{0,1,2\}$ if $\gamma$ joins $K$ to itself and $\{0,1\}$ otherwise, we find that Reeb chords in $\mathcal{R}^{KK}$, $\mathcal{R}^{Kp}$, and $\mathcal{R}^{pK}$ have degrees in $\{0,1,2\}$, $\{0,1\}$, and $\{1,2\}$, respectively.
It follows that $\A_{\Lambda_K,\Lambda_K}$ (respectively
$\A_{\Lambda_K,\Lambda_p}$, $\A_{\Lambda_p,\Lambda_K}$) is supported
in degree $\geq 0$ (respectively $\geq 0$, $\geq 1$), and in lowest
degree is generated as a $\Z$-module by only words with the minimal
possible number of mixed chords. In
particular, $\F^2 \A_{\Lambda_K,\Lambda_K}$, $\F^3
\A_{\Lambda_K,\Lambda_p}$, and $\F^3 \A_{\Lambda_p,\Lambda_K}$ are all
zero in degree $0$, $0$, and $1$ respectively, and so:
\begin{align*}
H_0(\A_{\Lambda_K,\Lambda_K}) &\cong H_0(
  \A_{\Lambda_K,\Lambda_K}^{(0)}) \\
H_0(\A_{\Lambda_K,\Lambda_p}) &\cong H_0(
\A_{\Lambda_K,\Lambda_p}^{(1)}) \\
H_1(\A_{\Lambda_p,\Lambda_K}) &\cong H_1(
\A_{\Lambda_p,\Lambda_K}^{(1)}).
\end{align*}

As noted in Section~\ref{ssec:LCHlink}, the first of these,
$H_0(
  \A_{\Lambda_K,\Lambda_K}^{(0)})$, is exactly
  the degree $0$ Legendrian contact homology of $\Lambda_K$, that is,
  degree $0$ knot contact homology. The homology coefficients
$\Z[H_1(\Lambda_K)] \cong \Z[l^{\pm 1},m^{\pm 1}]$ form a degree $0$ subalgebra
of the DGA of $\Lambda_K$ with zero differential, and so we have a map
$\Z[l^{\pm 1},m^{\pm 1}]$ into the degree $0$ knot contact homology of
$K$. In addition, $\A_{\Lambda_K,\Lambda_K}$ acts on the left
(respectively right) on $\A_{\Lambda_K,\Lambda_p}$ (respectively
$\A_{\Lambda_p,\Lambda_K}$), with an induced action on homology.

\begin{definition}
The \textit{KCH-triple}
\label{def:KCH-triple}
of $\Lambda_K \cup \Lambda_p$ is
\[
(R_{KK},R_{Kp},R_{pK}) = (H_0(\A_{\Lambda_K,\Lambda_K}),
H_0(\A_{\Lambda_K,\Lambda_p}),H_1(\A_{\Lambda_p,\Lambda_K})).
\]
Here $R_{KK}$ is viewed as a ring equipped with a map $\Z[l^{\pm 1},m^{\pm
  1}] \to R_{KK}$, and $R_{Kp}$ and $R_{pK}$ are
an $(R_{KK},\Z)$-bimodule and a $(\Z,R_{KK})$-bimodule, respectively.
\end{definition}

Note that although we have chosen a particular placement of $K$ and
$p$ above,
the KCH-triple of $\Lambda_K \cup \Lambda_p$ is unchanged by isotopy of $\Lambda_K$ and
$\Lambda_p$, since it can be defined strictly in terms of
graded pieces of the
homology of $\A_{\Lambda_K \cup \Lambda_p}$, which is invariant under Legendrian isotopy up to quasi-isomorphism. That is:

\begin{proposition}
If $\Lambda_{K_0} \cup \Lambda_p$ and $\Lambda_{K_1} \cup \Lambda_p$
are Legendrian isotopic,
\label{prop:KCH-triple}
then they have isomorphic KCH-triples
$(R_{K_0K_0},R_{K_0p},R_{pK_0})$ and $(R_{K_1K_1},R_{K_1p},R_{pK_1})$,
in the sense that there are isomorphisms
\begin{align*}
\psi_{KK} &:\thinspace R_{K_0K_0} \stackrel{\cong}{\to} R_{K_1K_1} \\
\psi_{Kp} &:\thinspace R_{K_0p} \stackrel{\cong}{\to} R_{K_1p} \\
\psi_{pK} &:\thinspace R_{pK_0} \stackrel{\cong}{\to} R_{pK_1}
\end{align*}
compatible with multiplications $R_{K_iK_i} \otimes R_{K_iK_i} \to
R_{K_iK_i}$, $R_{K_iK_i} \otimes R_{K_ip} \to
R_{K_ip}$, and $R_{pK_i} \otimes R_{K_iK_i} \to
R_{pK_i}$. If furthermore the Legendrian isotopy is parametrized in
the sense that it sends the basis $l_0,m_0$ of $H_1(\Lambda_{K_0})$ to
the basis $l_1,m_1$ of $H_1(\Lambda_{K_1})$, then $\psi_{KK}(m_0) =
m_1$, $\psi_{KK}(l_0) = l_1$.
\end{proposition}

\begin{remark}
As mentioned above, the gradings in $\A_{\Lambda_{K},\Lambda_{p}}$ and
$\A_{\Lambda_{p},\Lambda_{K}}$ are not canonically defined but rather
depend on a choice of homotopy class of a path connecting the tangent
planes at base points in the components of the Legendrian link
(possible choices are in one to one correspondence with $\Z$).
In general, in Definition~\ref{def:KCH-triple} we would want to set
$R_{Kp} = H_d(\A_{\Lambda_K,\Lambda_p})$ and $R_{pK} =
H_{1-d}(\A_{\Lambda_p,\Lambda_K})$, where $d\in\Z$ corresponds to the
choice of homotopy class of path. (In all cases we still have $R_{KK}
= H_0(\A_{\Lambda_K,\Lambda_K})$.)

This indeterminacy would seem to pose problems for
Proposition~\ref{prop:KCH-triple}. However, we can eliminate the
ambiguity by stipulating that we have picked the unique choice of
grading for which
\[
\min \{d \,|\, H_d(\A_{\Lambda_K,\Lambda_p}) \neq 0\} = 0.
\]
This is because with our preferred choice of grading,
$H_d(\A_{\Lambda_K,\Lambda_p}) = 0$ for $d < 0$, while we will show that
$R_{Kp} = H_0(\A_{\Lambda_K,\Lambda_p})$ is nonzero (see for instance
Proposition~\ref{prop:triple-isom}).
\end{remark}

\begin{remark}
If we choose $p$ sufficiently far away from $K$, then by action
considerations (the action of the Reeb chord at the positive puncture of a holomorphic disk is greater than the sum of the actions at the negative punctures), the differential in $(\A_{\Lambda_K \cup
\Lambda_p},\partial)$ of any word containing exactly $k$ mixed chords must
only involve words again containing exactly $k$ mixed chords. In this
case, the DGA $(\A_{\Lambda_K \cup \Lambda_p},\partial)$ is isomorphic
to its associated graded DGA under the filtration $\F^k$, and the
homology $\LCH_*(\Lambda_K \cup \Lambda_p)$ decomposes as a direct sum
by number of mixed chords.
\end{remark}

We will use the KCH-triple of $\Lambda_K \cup \Lambda_p$ to produce a complete knot invariant. More specifically, we have the following object created from the KCH-triple:

\begin{definition}
Let $R_{pp}$ denote the $\Z$-module
\[
R_{pp} = R_{pK} \otimes_{R_{KK}} R_{Kp}.
\]
\end{definition}

Alternatively, we can write $R_{pp}$ in terms of the homology of $\A_{\Lambda_K \cup \Lambda_p}$:

\begin{proposition}
We have
\label{prop:Rpp}
\[
R_{pp} \cong  H_1(\A_{\Lambda_p,\Lambda_p}) \cong H_1(\A_{\Lambda_p,\Lambda_p}^{(2)}).
\]
\end{proposition}

\begin{proof}
The first isomorphism is immediate from the definition of $R_{pp}$.
Since there are no self Reeb chords of $\Lambda_p$, and since any mixed Reeb chord to $\Lambda_p$ from $\Lambda_K$ has degree $\geq 1$, any degree $1$ generator of $\A_{\Lambda_p,\Lambda_p}$
must consist of a mixed chord to $\Lambda_p$ from $\Lambda_K$, followed by some number of Reeb chords of $\Lambda_K$, followed by a mixed chord to $\Lambda_K$ from $\Lambda_p$. The result now follows from the definition of the KCH-triple.
\end{proof}

In fact, we will show (Proposition~\ref{prop:ring-isom}) that there is a ring isomorphism
\[
\Z \oplus R_{pp} \cong \Z[\pi_1(\R^3\setminus K)]
\]
and this is the key to proving our main result, Theorem~\ref{thm:main}. To get this, we in particular need a multiplication operation on $R_{pp}$. In the next subsection, we will define a product map
\[
\mu :\thinspace R_{Kp} \otimes R_{pK} \to R_{KK}.
\]
This will then induce a map
\[
\mu :\thinspace R_{pp} \otimes_\Z R_{pp} =
R_{pK} \otimes_{R_{KK}} R_{Kp} \otimes_\Z R_{pK} \otimes_{R_{KK}}
R_{Kp}
\to
R_{pK} \otimes_{R_{KK}} R_{KK} \otimes_{R_{KK}} R_{Kp} =
R_{pp},
\]
which is the desired multiplication on $R_{pp}$.

\subsection{Product}
\label{ssec:product}

Recall that the differential in the contact homology DGA $\A_{\Lambda_{K}\cup\Lambda_p}$ that is used to define the KCH-triple $(R_{KK},R_{Kp},R_{pK})$ counts holomorphic disks with one positive puncture in the symplectization $\R\times ST^*\R^3$ with boundary on $\R \times(\Lambda_p \cup \Lambda_K)$. As
described in \cite{E08,BEE2}, one can also produce
invariants by counting holomorphic disks with two positive punctures
at mixed Reeb chords, along with an arbitrary number of negative
punctures at pure Reeb chords.

For general two-component Legendrian links, the resulting algebraic
structure is a bit complicated to describe, but in our case it is
simple because $\Lambda_p$ has no self Reeb chords: reading along the
boundary of any of these two-positive-punctured disks, we see a
positive puncture from $\Lambda_K$ to $\Lambda_p$, followed by a
positive puncture from $\Lambda_p$ to $\Lambda_K$, followed by some
number of negative punctures from $\Lambda_K$ to $\Lambda_K$.
This allows us to define the product of a Reeb chord from $\Lambda_K$
to $\Lambda_p$ with a Reeb chord from $\Lambda_p$ to $\Lambda_K$, or
more generally the product of composable words in
$\A_{\Lambda_K,\Lambda_p}$ and $\A_{\Lambda_p,\Lambda_K}$, each of
which contains exactly one mixed Reeb chord. The result of the product
in either case will be an alternating word of pure chords from $\Lambda_K$ to $\Lambda_K$ and homotopy classes of loops in $\Lambda_K$. (No mixed Reeb chords are involved.)

\begin{figure}
\labellist
\small\hair 2pt
\pinlabel $-1$ at 0 129
\pinlabel $1$ at 177 129
\pinlabel $+$ at 144 129
\pinlabel $+$ at 36 129
\pinlabel $-$ at 56 84
\pinlabel $-$ at 90 73
\pinlabel $-$ at 118 84
\pinlabel $a_1$ at 368 253
\pinlabel $a_2$ at 495 253
\pinlabel $b_1$ at 323 4
\pinlabel $b_2$ at 433 4
\pinlabel $b_3$ at 541 4
\pinlabel $\beta_0$ at 304 117
\pinlabel $\beta_1$ at 378 58
\pinlabel $\beta_2$ at 486 58
\pinlabel $\beta_3$ at 555 148
\pinlabel ${\color{green} \Lambda_p}$ at 430 147
\pinlabel ${\color{blue} \Lambda_K}$ at 348 168
\pinlabel ${\color{blue} \Lambda_K}$ at 339 58
\pinlabel ${\color{blue} \Lambda_K}$ at 447 58
\pinlabel ${\color{blue} \Lambda_K}$ at 514 187
\endlabellist
\centering
\includegraphics[width=0.7\textwidth]{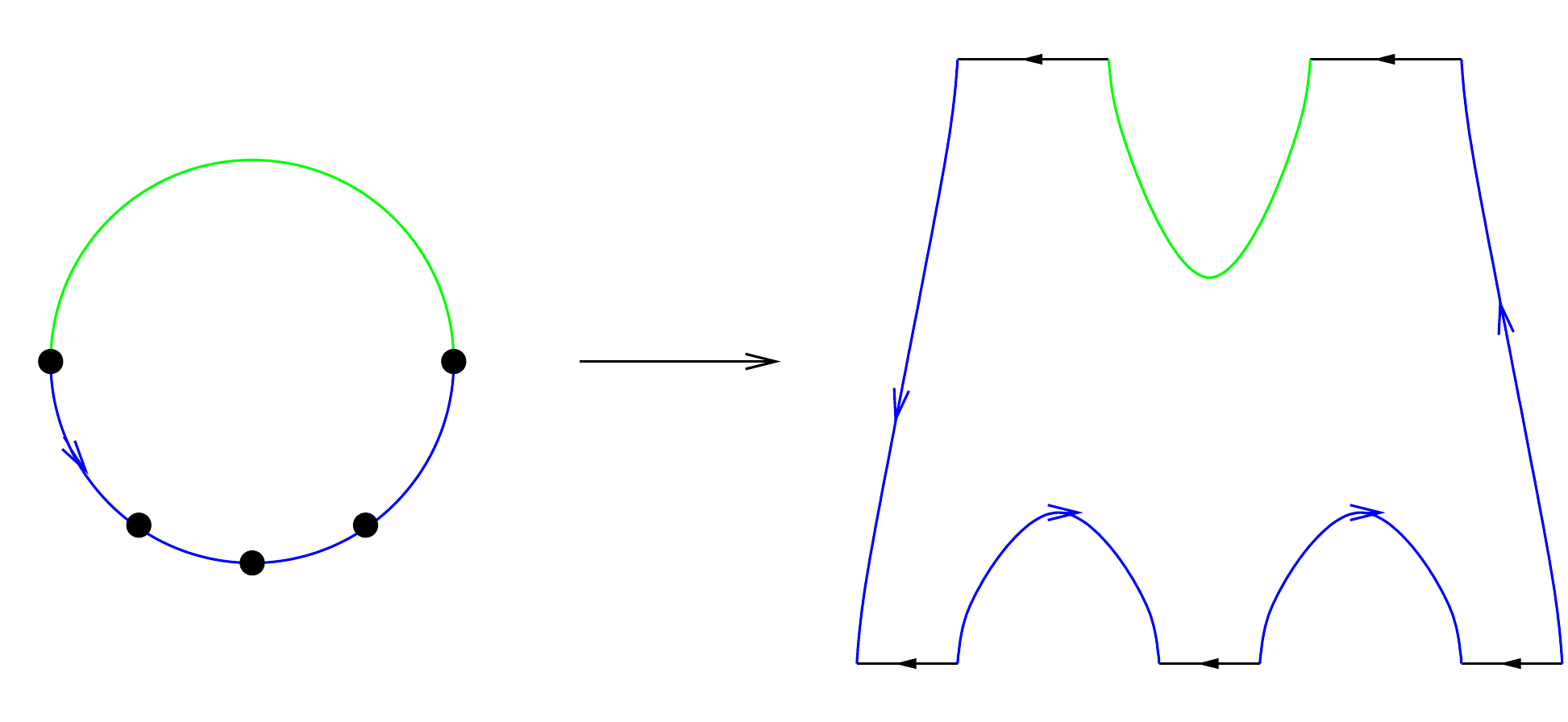}
\caption{
A disk in the moduli space
$\M(a_1,a_2;\beta_0b_1\beta_1b_2\beta_2b_3\beta_3)$. Here the
$\beta_i$ coefficients
record the homology classes of the depicted arcs in $\Lambda_K$. The arrows on Reeb chords denote their positive orientations.
}
\label{fig:product}
\end{figure}
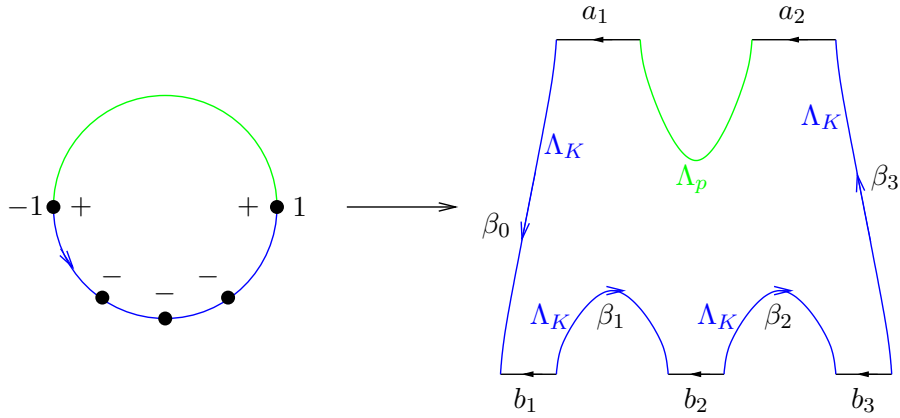

We now describe this construction in more detail. Let $a_1$ be a Reeb
chord to $\Lambda_{K}$ from $\Lambda_{p}$ and $a_2$ a Reeb chord to
$\Lambda_{p}$ from $\Lambda_K$. Let $\mathbf{b}=\beta_0 b_1 \beta_1
\cdots b_m \beta_m$ be a word in Reeb chords $b_i$ from $\Lambda_K$ to
itself and homology classes $\beta_i$ in $H_1(\Lambda_K)$.
Consider the moduli space of holomorphic disks in the symplectization
$\R\times ST^{\ast}\R^{3}$ of the following form. We take the domain
of the disks to be the unit disk in the complex plane with punctures
and boundary data as follows: there are two positive punctures at $1$ and $-1$;
the arc in the upper half plane connecting these two punctures maps to
$\R\times\Lambda_{p}$; there are $m\ge 0$ negative punctures along the boundary
arc in the lower half plane; and the boundary components in the lower
half plane all map to $\R\times\Lambda_{K}$ according to
$\mathbf{b}$. See Figure~\ref{fig:product}.
We write
\[
\M(a_1,a_2;\mathbf{b})
\]
for the moduli space of holomorphic disks in the symplectization $\R\times ST^{\ast}\R^{3}$ with punctures and boundary data as described.

The dimension of this moduli space is then the following, where $|c|$ denotes the grading of the Reeb chord $c$:
\[
\dim(\M(a_1,a_2;\mathbf{b}))= |a_1|+|a_2|-|\mathbf{b}|.
\]
\begin{remark}\label{r:gendim}
This is a special case of a general dimension formula for holomorphic disks in the symplectization $\R\times ST^{\ast}Q$ of the cosphere bundle over an $n$-manifold, with boundary in $\R\times\Lambda$, where $\Lambda$ is a Legendrian submanifold of Maslov class $0$. If such a disk $u$ has positive punctures at Reeb chords $a_1,\dots,a_p$ and negative punctures at Reeb chords $b_1,\dots,b_q$ then its formal dimension $\dim(u)$ is (see e.g. \cite[Theorem A.1]{CEL} or \cite[Section 3.1]{E08}):
\[
\dim(u)=(n-3) + \sum_{j=1}^{p} (|a_j|-(n-3)) -\sum_{k=1}^{q}|b_k|.
\]
\end{remark}

As in the definition of the differential in Legendrian contact
homology, we need to consider orientations of these moduli spaces induced by capping operators and the Fukaya orientation on the space of linearized Cauchy--Riemann operators on the disk with trivialized Lagrangian boundary condition.\footnote{In fact, for the purposes of this paper and in particular the proof of Theorem~\ref{thm:main}, one could ignore orientations and work over $\Z/2$ rather than $\Z$. However, for the purposes of the general theory, we will work over $\Z$ throughout.}
There is basically only one point where the construction here differs from that used for the differential. The disks in the differential have a unique positive puncture and we write the capped-off linearized problem for a disk with positive puncture at $a$ and negative punctures and boundary data according to $\mathbf{b}=\beta_0b_1\cdots b_m\beta_m$ as above as (with $C^{\pm}_{c}$ denoting the positive/negative capping operator at the Reeb chord $c$ and $L$ denoting the linearized Cauchy-Riemann operator at the holomorphic disk under consideration)
\[
C_{a}^{+}\oplus L\oplus C^{-}_{b_1}\oplus\cdots\oplus C^{-}_{b_m} \approx F,
\]
where $F$ denotes a trivialized boundary condition on the closed disk
and where ``$\approx$'' means ``is related to via a linear gluing exact
sequence'', see \cite{EESori}. For the product, we have disks with two
positive punctures and there is no natural way to order the punctures
in general. However, in our case the two positive punctures are distinguished since both are mixed and have different endpoint configurations. We choose the following ordering:
\[
C_{a_1}^{+}\oplus L\oplus C^{-}_{b_1}\oplus\cdots\oplus C^{-}_{b_m}\oplus C_{a_2}^{+} \approx F.
\]
As usual this then induces a linear gluing sequence which in the
transverse case orients the moduli space.

With these orientations determined, we can now define $\mu$. Suppose
that we have $\mathbf{c}_1 a_1\in \A_{\Lambda_K,\Lambda_p}^{(1)}$ and
$a_2\mathbf{c}_{2}\in \A_{\Lambda_p,\Lambda_K}^{(1)}$, where $a_1$,
$a_2$ are mixed chords to $\Lambda_K$ from $\Lambda_p$ (respectively
to $\Lambda_p$ from $\Lambda_K$), and $\mathbf{c}_1$, $\mathbf{c}_2$
are words in pure Reeb chords on $\Lambda_K$ and homology classes in
$\Lambda_K$. Define:
\[
\mu(\mathbf{c}_1 a_1,a_2\mathbf{c}_{2})=
\sum_{|a_1|+|a_2|-|\mathbf{b}|=1}|\M(a_1,a_2;\mathbf{b})/\R|\;\mathbf{c_1}\mathbf{b}\mathbf{c}_2.
\]
This produces a map
\[
\mu :\thinspace
\A_{\Lambda_K,\Lambda_p}^{(1)}
\otimes
\A_{\Lambda_p,\Lambda_K}^{(1)}
\to
\A_{\Lambda_K,\Lambda_K}^{(0)}.
\]

\begin{proposition}[\cite{E08}]
The product map $\mu$ has degree $-1$
\label{prop:mu}
and satisfies the Leibniz rule:
\[
\mu \circ(\partial \otimes 1 + 1 \otimes \partial) = \partial \circ \mu.
\]
Thus $\mu$ descends to a map on homology.
\end{proposition}

\noindent
Here and in the rest of the paper, we use Koszul signs when defining the tensor product of maps: in particular, $(\partial \otimes 1)(a\otimes b) = (\partial a) \otimes b$ while $(1\otimes\partial)(a\otimes b) = (-1)^{|a|} a \otimes(\partial b)$ if $\partial$ has odd degree.
Although Proposition~\ref{prop:mu} is implicitly contained in
\cite{E08}, we give the proof for definiteness.

\begin{proof}[Proof of Proposition~\ref{prop:mu}]
Once we know that the moduli spaces are transversely cut out for
generic data then the fact that $\mu$ has degree $-1$ follows from the dimension formula.
The disks with two positive punctures considered here cannot be multiply covered for topological reasons (e.g.~ only one positive puncture is asymptotic to a chord from $\Lambda_{p}$ to $\Lambda_K$). Thus the same argument as for disks with one positive puncture can be used to show transversality for generic almost complex structure where the formal dimension then equals the actual dimension, see e.g. \cite[Proposition 2.3]{EES2}.

To see the displayed equation, we look at the boundary of moduli spaces $\M(a_1,a_2;\mathbf{b})$ of dimension $2$. It follows by SFT compactness that the boundary consists of broken curves. We must check that there cannot be any boundary breaking. To see this note that any splitting arc in the domain that separates the positive punctures must connect boundary points that map to distinct components of the Legendrian submanifold. Thus there is no boundary breaking and several-level disks account for the whole boundary. The equation follows from identifying contributing terms with the boundary of an oriented 1-dimensional manifold.
\end{proof}

We will also need the fact that $\mu$ is invariant under Legendrian
isotopy. As in \cite{E08} this can be understood by looking at
cobordism maps and homotopies of such. We will only need invariance on the
level of homology, and this is slightly easier to prove: we need only
the statement that the multiplication on homology induced by $\mu$ is
invariant under Legendrian isotopy and this follows from properties of
cobordism maps and analogues of these for the product. To see this
first recall that a Legendrian isotopy
$\Lambda_{K_t}\cup\Lambda_{p_t}$, $0\le t\le 1$,
from $\Lambda_{K_0}\cup\Lambda_{p_0}$ to $\Lambda_{K_1}\cup\Lambda_{p_1}$ gives an exact Lagrangian cobordism $L_K\cup L_p\subset \R\times ST^{\ast}\R^{3}$ that agrees with $(\R\times\Lambda_{K_0})\cup (\R\times\Lambda_{p_0})$ in the positive end and $(\R\times\Lambda_{K_1})\cup (\R\times\Lambda_{p_1})$ in the negative. Furthermore there is a cobordism map
\[
\Phi\colon\A_{\Lambda_{K_0}\cup\Lambda_{p_0}}\to
\A_{\Lambda_{K_1}\cup\Lambda_{p_1}}
\]
that is a quasi-isomorphism respecting the filtration with respect to the number of mixed chords. This cobordism map counts holomorphic disks in the cobordism as follows. If $a$ is a Reeb chord of $\Lambda_{K_0}\cup\Lambda_{p_0}$ and if $\mathbf{b}$ is an alternating word of Reeb chords and homotopy classes of paths in $\Lambda_{K_1}\cup\Lambda_{p_1}$ then let $\M^{\rm co}(a,\mathbf{b})$ denote the moduli space of holomorphic disks in $\R\times ST^{\ast}\R^{3}$ with boundary on $L_{K}\cup L_{p}$, with one positive puncture where the disk is asymptotic to the Reeb chord $a$ and with several negative punctures which together with the boundary arcs give the word $\mathbf{b}$. The map $\Phi$ is then given by
\[
\Phi(a)=\sum_{|a|-|\mathbf{b}|=0}|\M^{\rm co}(a,\mathbf{b})|\mathbf{b}.
\]

In order to study invariance for the product, we look at moduli spaces
in the cobordism analogous to the moduli spaces used in the definition
of $\mu$. If $a_1$ is a Reeb chord from $\Lambda_{K_0}$ to $\Lambda_{p_0}$, $a_2$ a chord from $\Lambda_{p_{0}}$ to $\Lambda_{K_0}$, and $\mathbf{b}$ a word of Reeb chords and homotopy classes of paths on $\Lambda_{K_1}$ as above, then let $\M^{\rm co}(a_1,a_2;\mathbf{b})$ denote the moduli space of disks with boundary on $L_K\cup L_p$ with two positive punctures asymptotic to $a_1$ and $a_2$, and with negative punctures and boundary arcs mapping according to $\mathbf{b}$. Then define
$\kappa\colon R_{K_0p_0}\otimes R_{p_0K_0}\to R_{K_1K_1}$ by:
\[
\kappa(\mathbf{c}_1a_1,a_2\mathbf{c_2}) = \sum_{|a_1|+|a_2|-|\mathbf{b}|=0}|\M^{\rm co}(a_1,a_2;\mathbf{b})|\Phi(\mathbf{c_1})\mathbf{b}\Phi(\mathbf{c_2}).
\]

\begin{proposition}[\cite{E08}]
Given a Legendrian isotopy $(\Lambda_{K_t},\Lambda_{p_t})$, $t\le 0\le
1$, and product maps $\mu_0$ and $\mu_1$ for $\Lambda_{K_0} \cup
\Lambda_{p_0}$ and $\Lambda_{K_1} \cup \Lambda_{p_1}$ as defined
above, we have:
\begin{equation}\label{eq:prod+cobmap}
\Phi\circ\mu_0 - \mu_1\circ \Phi - \kappa\circ (\partial_0\otimes 1 + 1\otimes\partial_0) +  \partial_1\circ\kappa = 0.
\end{equation}
Thus, on the level of homology, $\Phi$ sends $\mu_0$ to $\mu_1$.
\end{proposition}

\begin{proof}
The proof follows from an analysis of the boundary of $1$-dimensional
moduli spaces of the form $\M^{\rm co}(a_1,a_2;\mathbf{b})$. By gluing
and SFT compactness the boundary of such a moduli space consists of
two-level buildings. Hence the terms contributing to
\eqref{eq:prod+cobmap} are in 1-to-1 correspondence with the
boundary of an oriented 1-manifold. The homology statement
 follows from \eqref{eq:prod+cobmap} together with the fact that
 $\Phi$ is a quasi-isomorphism respecting the filtration.
\end{proof}

We now connect this general discussion of the product $\mu$ with the
KCH-triple. Recall from Section~\ref{ssec:enhanced} that we have:
\begin{align*}
R_{KK} &\cong H_0(\A_{\Lambda_K,\Lambda_K}^{(0)}) &
R_{Kp} &\cong H_0(\A_{\Lambda_K,\Lambda_p}^{(1)}) &
R_{pK} &\cong H_1(\A_{\Lambda_p,\Lambda_K}^{(1)}).
\end{align*}
Then the product gives a map
\[
\mu :\thinspace R_{Kp} \otimes R_{pK} \to R_{KK}.
\]
We can now write the invariance property as follows.

\begin{proposition}
The $R_{KK}$-bimodule map $\mu :\thinspace R_{Kp} \otimes_{\Z} R_{pK} \to R_{KK}$ is invariant under Legendrian isotopy of
$\Lambda_K \cup \Lambda_p$: given isotopic $\Lambda_{K_0} \cup
\Lambda_p$ and $\Lambda_{K_1} \cup \Lambda_p$ and isomorphisms
$\psi_{KK},\psi_{Kp},\psi_{pK}$ of KCH-triples as in
Proposition~\ref{prop:KCH-triple}, we have
\[
\psi_{KK} \circ \mu_0 = \mu_1 \circ (\psi_{Kp} \otimes \psi_{pK}).
\]
\end{proposition}


\section{String Topology}
\label{sec:stringtop}
In this section we will describe how to extend the results from
\cite{CELN} to enhanced knot contact homology with the product
$\mu$. This will allow us to interpret $\LCH_*(\Lambda_K \cup
\Lambda_p)$ in low degree in terms of a version of string topology and homotopy data; in
particular, we will proceed in Section~\ref{sec:homotopy} to use string topology arguments
to write the KCH-triple in terms of $\pi_1(\R^3\setminus K)$.

The main result of \cite{CELN} is an isomorphism in degree $0$
between the Legendrian contact homology of $\Lambda_K$ and a ``string
homology'' defined using chains of broken strings, where a broken
string is a loop in the union $\R^3 \cup L_K$ of the
zero section and the conormal bundle of $L_K$ in $T^*\R^3$. Here we
will give a modification of this approach that produces an
isomorphism between the Legendrian contact homology of $\Lambda_K \cup
\Lambda_p$ (in the appropriate degree) and string homology for broken
strings in $\R^3 \cup L_K \cup L_p$, where $L_p$ is the conormal of $p$ in $T^*\R^3$, i.e.\ the fiber $T_{p}^{\ast}\R^{3}$. We will then prove that the product $\mu$ defined
in Section~\ref{ssec:product} corresponds to the Pontryagin product on string
homology under this isomorphism.

The discussion in this section closely parallels the treatment in \cite{CELN}, as our setup is nearly identical to the one there, differing only in the introduction of $L_p$. Where convenient, we adopt notation from \cite{CELN} to make the correspondence clearer.

\subsection{Broken strings}
\label{ssec:broken-strings}

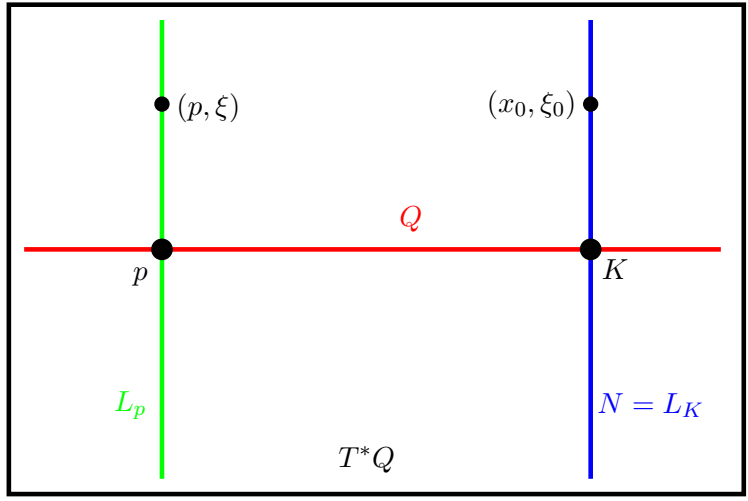
\begin{figure}
\labellist
\small\hair 2pt
\pinlabel $T^*Q$ at 213 23
\pinlabel ${\color{red} Q}$ at 239 165
\pinlabel ${\color{green} L_p}$ at 74 55
\pinlabel ${\color{blue} N=L_K}$ at 380 55
\pinlabel $(x_0,\xi_0)$ at 310 232
\pinlabel $(p,\xi)$ at 120 230
\pinlabel $p$ at 80 131
\pinlabel $K$ at 359 135
\endlabellist
\centering
\includegraphics[width=0.6\textwidth]{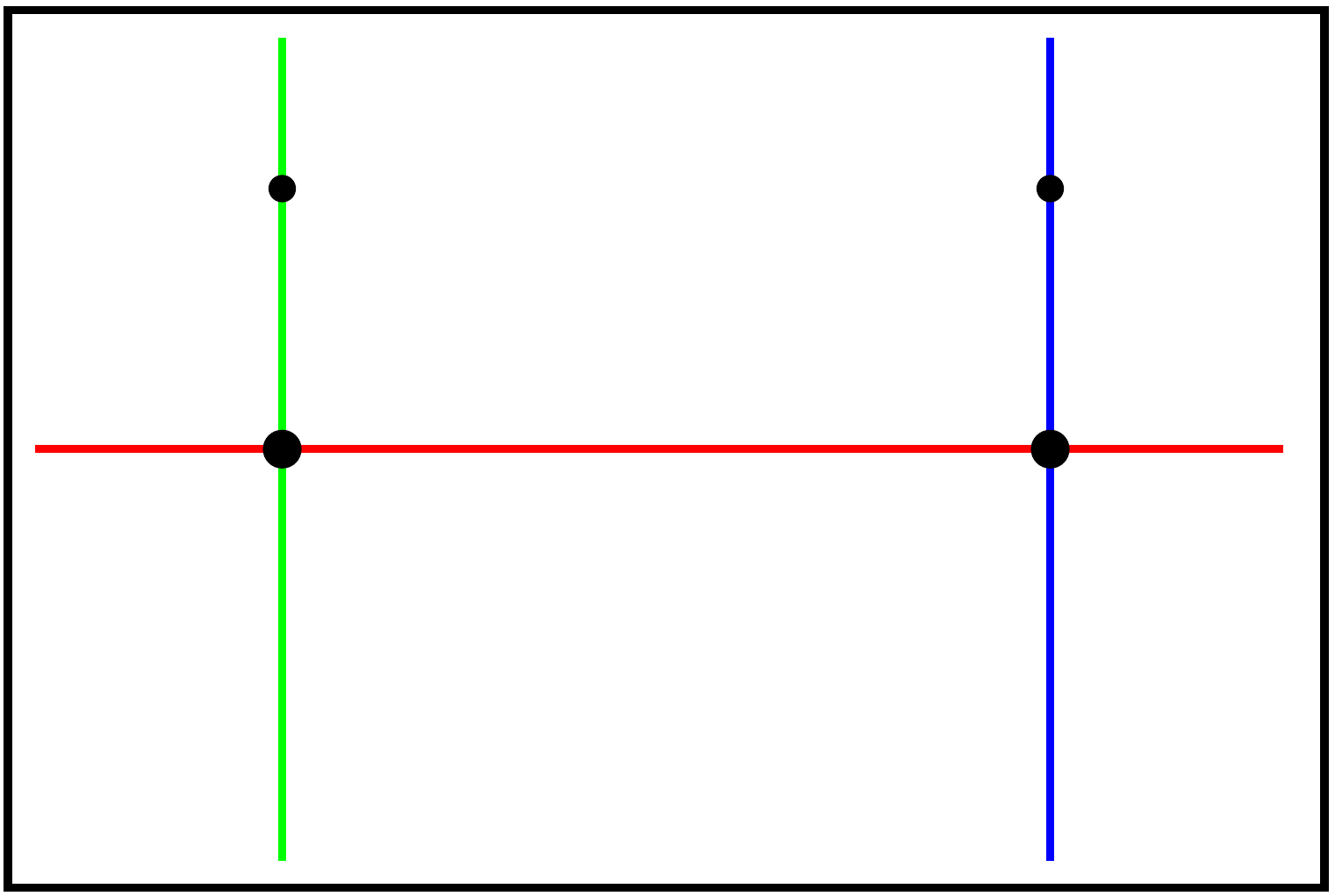}
\caption{The cotangent bundle $T^*Q$ with Lagrangians $Q,N,L_p$.}
\label{fig:cotangent}
\end{figure}

Here we recall the definition of broken strings from \cite{CELN},
suitably modified for our purposes. Let $K \subset \R^3$ be a knot and
$p \in \R^3$ be a point in the complement of $K$. Write $Q = \R^3$ and
view $Q$ as the zero section in $T^*Q$, and
let $N = L_K \subset T^*Q$ be the conormal bundle to $K$, while $L_p$
is the cotangent fiber $T_p^*Q$. We then have three Lagrangians
$Q$, $N$, and $L_p$ in $T^*Q$; $N$ and $L_p$ are disjoint, $Q$ and $L_p$ intersect
transversely at $p$, and $Q$ and $N$ intersect cleanly along $K$. See
Figure~\ref{fig:cotangent}.

Fix base points $(x_0,\xi_0) \in N \setminus K$ and $(p,\xi) \in L_p
\setminus \{p\}$. If we use a metric to identify $T^*Q$ and $TQ$, then
these points become $(x_0,v_0)$ with $v_0 \in T_{x_0}N$ and $(p,v)$
with $v \in T_pQ$. This metric also gives a diffeomorphism between a neighborhood of the zero section in $N$ (which in turn is diffeomorphic to all of $N$) and
a tubular neighborhood of $K$ in $Q$, and we can view $Q \cup N$ as
the disjoint union of $Q$ and $N \subset Q$ glued along $K$. This
allows us to identify $T_{x}N$ with $T_{x}Q$ for $x\in K$. Similarly
we view $Q \cup N \cup L_p$ as the disjoint union of $Q \cup N$ and
$L_p$ with $p\in Q$ and $0 \in L_p$ identified, and the metric
identifies $T_{p}Q$ with $T_0 L_p = T_{p}^*Q$.

Now consider a piecewise $C^1$ path in $Q \cup N \cup L_p$. This path
can move between $Q$ and $N$ (in either direction) at a point on $K =
Q \cap N$, and between
$Q$ and $L_p$ at $p$; we call the points where the path changes
components \textit{switches}, either at $K$ or at $p$.

\begin{definition}
A \textit{broken string} is a piecewise $C^1$ path $s :\thinspace
[a,b] \to Q \cup N \cup L_p$ such that:
\begin{itemize}
\item
the endpoints $s(a),s(b)$ are each at one of the two
base points $(x_0,v_0) \in N$ or $(p,v) \in L_p$;
\item
if $s(t_0)$ is a switch at $K$ from $N$ to $Q$ (i.e., for small $\epsilon>0$, $s((t_0-\epsilon,t_0])
\subset N$ and $s([t_0,t_0+\epsilon)) \subset Q$), then:
\[
\lim_{t\to {t_0}^-} (s'(t))^{\text{normal}} =
\lim_{t\to {t_0}^+} (s'(t))^{\text{normal}},
\]
where we identify $T_{s(t_0)}N$ with $T_{s(t_0)}Q$ and
$v^{\text{normal}}$ denotes the component of $v$ normal to $K$
with respect to the metric on $Q$;
\item
if $s(t_0)$ is a switch at $K$ from $Q$ to $N$, then:
\[
\lim_{t\to {t_0}^-} (s'(t))^{\text{normal}} = -
 \lim_{t\to {t_0}^+} (s'(t))^{\text{normal}};
\]
\item
if $s(t_0)$ is a switch at $p$ from $L_p$ to $Q$, then:
\[
\lim_{t\to {t_0}^-} s'(t) =
 \lim_{t\to {t_0}^+} s'(t);
\]
\item
if $s(t_0)$ is a switch at $p$ from $Q$ to $L_p$, then:
\[
\lim_{t\to {t_0}^-} s'(t) = -
 \lim_{t\to {t_0}^+} s'(t).
\]
\end{itemize}
\end{definition}

\noindent
The portions of $s$ in $Q$ (respectively $N$, $L_p$) are called $Q$-strings
(respectively $N$-strings, $L_p$-strings).

\begin{remark}
A broken string models the boundary of a holomorphic disk in $T^*Q$
with boundary on $Q \cup N \cup L_p$ and one positive puncture at
infinity at a Reeb chord for $\Lambda_K \cup \Lambda_p$. The condition
on the derivatives at a switch follows the
behavior of the boundary of such a disk at a point where the boundary
switches between $Q$ and $N$, or between $Q$ and $L_p$: if $v_{\rm in}$ and $v_{\rm out}$ denote the incoming and outgoing tangent vectors of a broken string at a switch then $v_{\rm out}=Jv_{\rm in}$, where $J$ is the almost complex structure along the $0$-section induced by the metric.
\end{remark}

If we project from $T^*Q$ to $Q$, then the endpoints of a broken
string are each either at $p$ or at the point on $K$ that is the
projection of $x_0$. With this in mind, we call a broken string $s$:
\begin{itemize}
\item
a $KK$ broken string if $s(a)=s(b)=(x_0,v_0)$
\item
a $Kp$ broken string if $s(a)=(x_0,v_0)$ and $s(b)=(p,v)$
\item
a $pK$ broken string if $s(a)=(p,v)$ and $s(b)=(x_0,v_0)$
\item
a $pp$ broken string if $s(a)=s(b)=(p,v)$.
\end{itemize}

\subsection{String homology}
\label{ssec:string-homology}

We now construct a complex from broken strings whose homology might be called ``string homology''; in Section~\ref{ssec:string-kch} below, we will describe an isomorphism between this homology and enhanced knot contact homology.

For $\ell \geq 0$, let $\Sigma_\ell$ denote the space of broken
strings with $\ell$ switches at $p$ (note that we do not count
switches at $K$ here), equipped with the $C^{k}$-topology for some $k\ge 3$. We write
\[
\Sigma_\ell = \Sigma_\ell^{KK} \sqcup \Sigma_\ell^{Kp} \sqcup
\Sigma_\ell^{pK} \sqcup \Sigma_\ell^{pp}
\]
where $\Sigma_\ell^{ij}$ denotes the subset of $\Sigma_\ell$
corresponding to $ij$ broken strings for $i,j\in\{K,p\}$, and then
\[
C_k(\Sigma_\ell) = C_k^{KK}(\Sigma_\ell) \oplus
C_k^{Kp}(\Sigma_\ell) \oplus
C_k^{pK}(\Sigma_\ell) \oplus
C_k^{pp}(\Sigma_\ell)
\]
for the free $\Z$-module generated by generic $k$-dimensional singular
simplices in $\Sigma_\ell$ ($C_k^{ij}$ is the summand corresponding to
$ij$ broken strings). Here ``generic'' refers to simplices that
satisfy the appropriate transversality conditions at switches and with respect to $K$ and to $p$; compare \cite[Definition~5.3]{CELN}.

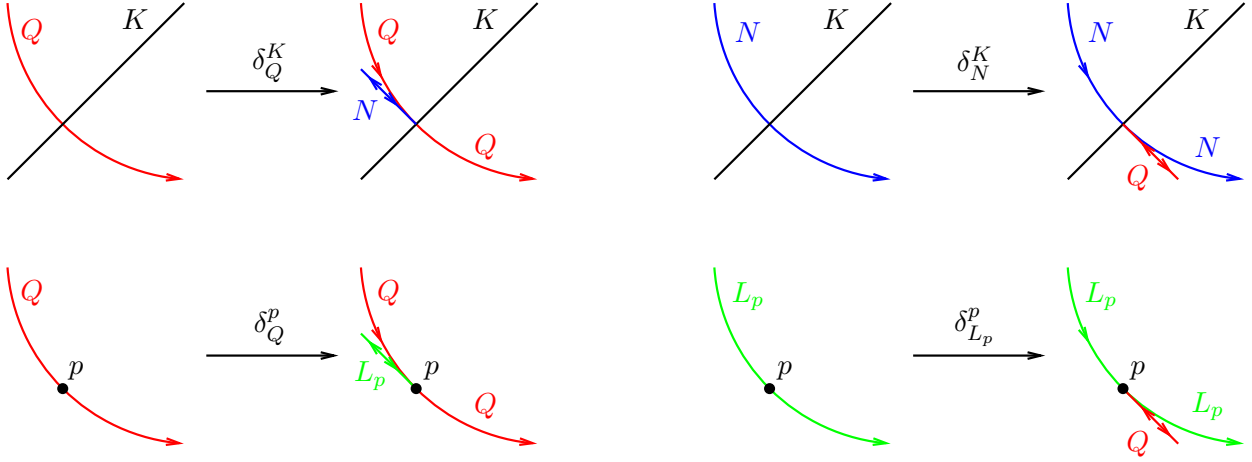
\begin{figure}
\labellist
\small\hair 2pt
\pinlabel $K$ at 106 350
\pinlabel $K$ at 392 350
\pinlabel $K$ at 678 350
\pinlabel $K$ at 972 350
\pinlabel $p$ at 58 62
\pinlabel $p$ at 348 62
\pinlabel $p$ at 636 62
\pinlabel $p$ at 924 62
\pinlabel $\delta^K_Q$ at 216 314
\pinlabel $\delta^K_N$ at 791 314
\pinlabel $\delta^p_Q$ at 216 98
\pinlabel $\delta^p_{L_p}$ at 791 98
\pinlabel ${\color{red} Q}$ at 22 339
\pinlabel ${\color{red} Q}$ at 313 339
\pinlabel ${\color{red} Q}$ at 392 245
\pinlabel ${\color{red} Q}$ at 924 220
\pinlabel ${\color{red} Q}$ at 22 126
\pinlabel ${\color{red} Q}$ at 313 126
\pinlabel ${\color{red} Q}$ at 392 33
\pinlabel ${\color{red} Q}$ at 924 2
\pinlabel ${\color{blue} N}$ at 294 275
\pinlabel ${\color{blue} N}$ at 606 339
\pinlabel ${\color{blue} N}$ at 894 339
\pinlabel ${\color{blue} N}$ at 981 245
\pinlabel ${\color{green} L_p}$ at 298 58
\pinlabel ${\color{green} L_p}$ at 606 123
\pinlabel ${\color{green} L_p}$ at 894 123
\pinlabel ${\color{green} L_p}$ at 981 33
\endlabellist
\centering
\includegraphics[width=\textwidth]{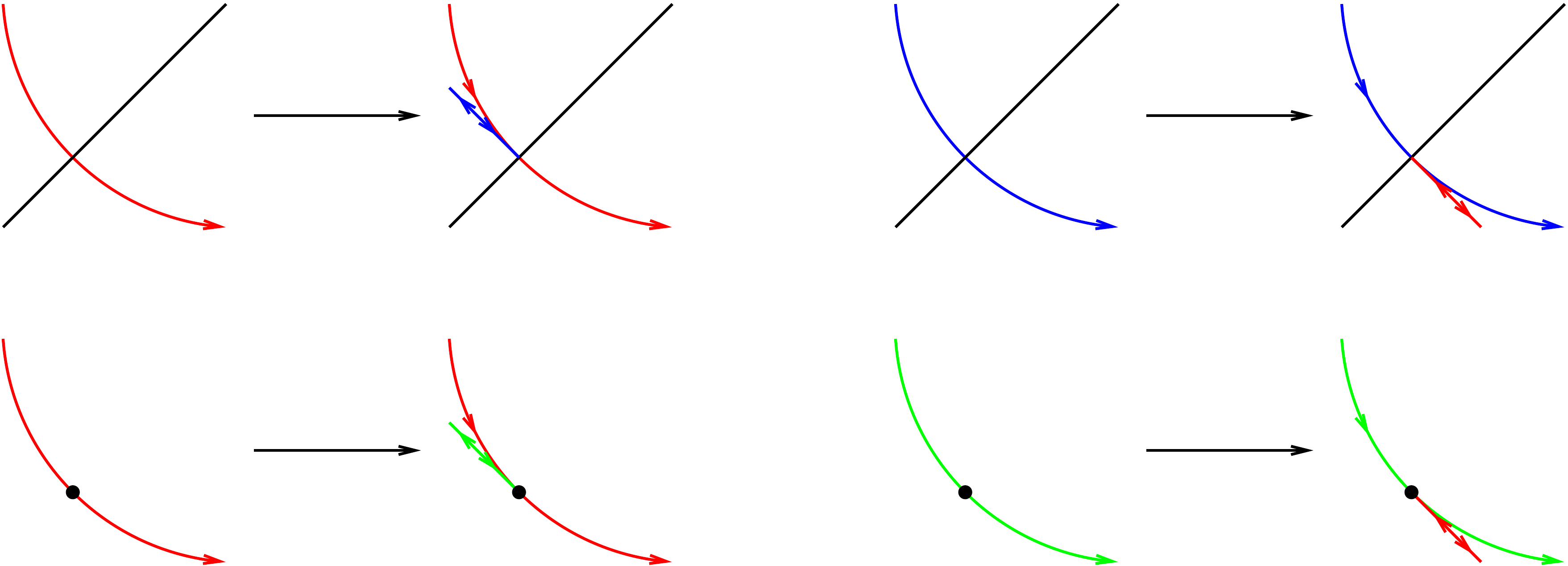}
\caption{The maps $\delta^K_Q$ (respectively $\delta^K_N$, $\delta^p_Q$,
  $\delta^p_{L_p}$) insert an $N$-string ($Q$-string, $L_p$-string, $Q$-string) at an interior point of a $Q$-string ($N$-string, $Q$-string, $L_p$-string) that lies on $K$ ($K$, $p$, $p$).}
\label{fig:delta-maps}
\end{figure}

In addition to the usual boundary operator $\partial :\thinspace
C_k(\Sigma_\ell) \to C_{k-1}(\Sigma_\ell)$ on singular simplices, there are two string
operations
\[
\delta_Q^K,~\delta_N^K :\thinspace C_k(\Sigma_\ell) \to
C_{k-1}(\Sigma_\ell)
\]
defined for $k \leq 2$ in \cite[Section 5.3]{CELN} (where they are called $\delta_Q$,
$\delta_N$). We refer to \cite{CELN} for details, but qualitatively
these operations take a generic $k$-dimensional family of broken
strings, identify the subfamily consisting of broken strings where a
$Q$-string or $N$-string has an interior point in $K$, and
insert a ``spike'' in $N$ or $Q$ at this point; see Figure~\ref{fig:delta-maps}. We note that this
interior intersection condition is codimension $1$, and that adding a
spike increases the number of switches at $K$ by $2$.
In our setting, there are two more string operations
\[
\delta_Q^p,~\delta_{L_p}^p :\thinspace C_k(\Sigma_\ell) \to
C_{k-2}(\Sigma_{\ell+2})
\]
that are defined in the same way as $\delta_Q^K$, $\delta_N^K$, but
inserting spikes in $L_p$ or $Q$ where a $Q$-string or $L_p$-string
has an interior point at $p$; see Figure~\ref{fig:delta-maps} again. Note now that the interior intersection
condition is codimension $2$, and that adding a spike increases the
number of switches at $p$ by $2$.

We then have the following result, which is a direct analogue of
Proposition~5.8 from \cite{CELN} and is proved in the same way.

\begin{lemma}
On generic $2$-chains, the operations $\partial$,
$\delta_{Q}^K+\delta_{N}^K$, and $\delta_{Q}^{p}+\delta_{L_p}^p$ each
have square $0$ and pairwise anticommute.
\label{lma:string}
In particular, we have
\begin{align*}
(\partial+\delta_{Q}^K+\delta_{N}^K)^{2}&=0,\\
(\partial+\delta_{Q}^K+\delta_{N}^K)(\delta_{Q}^{p}+\delta_{L_p}^p)
+(\delta_{Q}^{p}+\delta_{L_p}^p)(\partial+\delta_{Q}^K+\delta_{N}^K)&=0,\\
(\delta_{Q}^{p}+\delta_{L_p}^p)^2&=0.
\end{align*}
\end{lemma}

Lemma~\ref{lma:string} allows us to construct a complex out of broken
strings in the following way. For $m \in \frac{1}{2} \Z$, define
\[
C_m = \bigoplus_{k+\ell/2=m} C_k(\Sigma_\ell).
\]
By consideration of the parity of the number of switches at $p$, we
can write $C_m = C_m^{KK} \oplus C_m^{pp}$ when $m$ is an integer and
$C_m = C_m^{Kp} \oplus C_m^{pK}$ when $m$ is a half-integer. We define
a shifted complex $\tilde{C}_*$, $*\in\Z$, by:
\begin{gather*}
\tilde{C}_m^{KK} = C_m^{KK},
\hspace{3ex}
\tilde{C}_m^{Kp} = C_{m+1/2}^{Kp},
\hspace{3ex}
\tilde{C}_m^{pK} = C_{m-1/2}^{pK},
\hspace{3ex}
\tilde{C}_m^{pp} = C_m^{pp}, \\
\tilde{C}_m = \tilde{C}_m^{KK} \oplus \tilde{C}_m^{Kp} \oplus
\tilde{C}_m^{pK} \oplus \tilde{C}_m^{pp};
\end{gather*}
that is, we shift the grading up by $1/2$ if the beginning point is
$p$ and down by $1/2$ if the endpoint is $p$. By
Lemma~\ref{lma:string},
$\partial+\delta_{Q}^K+\delta_{N}^K+\delta_{Q}^{p}+\delta_{L_p}^p$ is
a differential on $\tilde{C}_*$ that lowers degree by $1$.

\begin{remark}\label{r:rotations}
The $\frac{1}{2}$-grading for strings broken at $p$ has the following geometric counterpart for holomorphic disks with switching Lagrangian boundary conditions on $L_p\cup Q$ and a punctures at the intersection point $p = L_p \cap Q$. Consider a disk $u\colon (D,\partial D)\to (T^{\ast}Q,L_p\cup Q)$ with $m$ punctures mapping to $p$ and with a positive puncture asymptotic to a Reeb chord $a$. The formal dimension of $u$ can then be expressed as follows, see \cite[Theorem A.1]{CEL}:
\begin{equation}\label{eq:dimcorner}
\dim(u) = (\dim(Q)-3) + \mu + m+1 = \mu + m + 1,
\end{equation}
where $\mu$ is the Maslov index of the loop of Lagrangian tangent planes along the boundary of $u$. Here we close this loop by the capping operator at $a$ and as follows at the punctures mapping to $p$: connect the incoming tangent plane ($T^{\ast}_{p}Q$ or $T^{\ast} L_{p}$) to the outgoing tangent plane ($T^{\ast} L_{p}$ or $T^{\ast}_{p}Q$) with a negative rotation along the K\"ahler angle (i.e.~act by $e^{-i\frac{\pi}{2}s}$, $0\le s\le 1$). In the case at hand the tangent planes along $Q$ and $L_{p}$ are stationary with respect to the standard trivialization and the dimension formula reduces to
\begin{equation}\label{eq:halfcontrib}
\dim(u) = |a| + m(-\tfrac{3}{2}+1). 
\end{equation}	   
(Note that $m$ is even, i.e., there is an even number of switches, because both the first and the last boundary component map to $L_{p}$.) 
In the dimension formula \eqref{eq:halfcontrib} there is a contribution of $-\frac{1}{2}$. In order to have each puncture contributing with an integer we can for example deform $L_p$ so that the K\"ahler angles between $Q$ and $L_{p}$ become $(\epsilon,\frac{\pi}{2},\frac{\pi}{2})$ instead of the original $(\frac{\pi}{2},\frac{\pi}{2},\frac{\pi}{2})$. This way the contribution to $\mu$ in \eqref{eq:dimcorner} for the punctures at $p$ switching from $L_{p}$ to $Q$ becomes $-2$ and the contribution for those switching in the opposite direction $-1$, giving total dimension contributions $-1$ and $0$, respectively. This deformation and the corresponding grading shift are chosen to match our choice of capping path connecting $\Lambda_{K}$ to $\Lambda_p$: that is, so that both Reeb chords that start at $L_p$ get shifted up by $1$ compared to the Morse grading and so that chains of broken strings starting at $p$ are also shifted up.
\end{remark}

\subsection{Switches at a point in an example}

On the complex of broken strings, there are four string operations, $\delta_Q^K$, $\delta_N^K$, $\delta_Q^p$, and $\delta_{L_p}^p$. The two that introduce switches on the knot, $\delta_Q^K$ and $\delta_N^K$, have appeared before and are studied at length in \cite{CELN}. For the other two, $\delta_Q^p$ and $\delta_{L_p}^p$, which introduce switches at a point, the only property we need for our main argument is their codimension; in the following section, Section~\ref{ssec:string-kch}, we use this to prove an isomorphism to knot contact homology.
Here we examine $\delta_Q^p$ and $\delta_{L_p}^p$ more closely in a model case. This is a digression from the main argument and can be skipped without loss of continuity, but provides some context for these operations within contact geometry.

Consider $Q=S^n$ and as usual let $L_p \subset T^*S^n$ be the cotangent fiber over $p$, with $\Lambda_p \subset ST^*S^n$ the Legendrian sphere given by the unit cotangent fiber. By the surgery result from \cite[\S 5.5]{BEE}, we can compute the DGA of $\Lambda_p$ in $ST^*S^n$ via the DGA for the Legendrian unknot $U\subset S^{2n-1}$, where $S^{2n-1}$ is the standard contact $(2n-1)$-sphere, i.e., the contact boundary of the standard symplectic $2n$-ball, with the differential in the latter DGA twisted by a point condition at $p$.

There is (effectively) only one Reeb chord $a$ of $U$ of grading $|a|=n-1$, see \cite{BEE}, and the differential is $\partial a = p$. To see this one can use the flow tree description of holomorphic disks: it is easy to see that for the standard front of the unknot there is exactly one rigid point constrained Morse flow tree with positive puncture at $a$.
Thus the DGA of $\Lambda_{p}$ is generated by chords $a^r$, $r \geq 1$, of grading $r(n-1) + (n-2)$, and the differential is
\[
\partial (a^{r}) = \sum_{j=2}^{r} a^{j-1}\cdot a^{r-j}.
\]

We claim that this DGA is chain isomorphic to the complex of broken strings in $Q \cup L_p$ with differential given by $\partial+\delta_{Q}^{p}+\delta_{L_p}^{p}$. For the latter, note that $L_p$ is contractible so we simply forget the $N$-strings and think of the chains of broken strings as the tensor algebra of chains on the based loop space of $S^n$ with differential $\partial+\delta_p$, where $\delta_p$ splits a chain over the locus where its evaluation map hits $p$.  By Morse theory, the space of non-constant based loops in $S^n$ is a cell complex with a cell in dimensions
\[
(n-1),\; 2(n-1),\; 3(n-1),\; 4(n-1),\; \dots.
\]
For degree reasons there are only quadratic terms in the differential $\partial+\delta_p$ and in order to compute $\delta_p$ we need to see the unstable manifolds of the cells that correspond to Morse flow in the Bott-manifolds followed by shrinking the loops over half-disks. It is not hard to see that $\delta_{p}$ acts on the Morse cells by splitting the cell of dimension $r(n-1)$ into two cells of dimensions $j(n-1)$ and $k(n-1)$, where $j+k=r-1$, in all possible ways.

We thus conclude that the complex of broken strings in $Q \cup L_p$ is indeed isomorphic to the DGA of the cosphere $\Lambda_p$. Furthermore, one can check that this isomorphism is induced by the map that associates to a Reeb chord $c$ of $\partial L_p$ the chain carried by the moduli-space of disks with positive puncture at $c$ and switching boundary condition on $Q\cup L_p$. Note that each pair of switches in the boundary of such a disk contributes $-(n-2)$ to the dimension of the moduli space, see Remark \ref{r:rotations}, which explains the difference in grading between the generators of the DGA of $\Lambda_p$ and generators of the complex of chains of broken strings ($r(n-1)+(n-2)$ versus $r(n-1)$ for $r \geq 1$).

\subsection{String homology and enhanced knot contact homology}
\label{ssec:string-kch}

In \cite{CELN} the DGA $\A_{\Lambda_K}$ was related to string homology via a chain map defined through a count of holomorphic disks with switching boundary condition. Here we similarly relate $\A_{\Lambda_{K}\cup\Lambda_{p}}$ to string homology. More precisely, if $a$ is a Reeb chord of $\Lambda_{K}\cup\Lambda_{p}$ then we let $\mathcal{M}^{\rm sw}(a)$ denote the moduli space of holomorphic disks in $T^{\ast}\R^{3}$ with one positive puncture asymptotic to the Reeb chord $a$ at infinity, and such that the disk has switching boundary on $Q \cup N \cup L_p$: that is, the boundary of the disk lies on $Q \cup N \cup L_p$, and there are several punctures where the boundary switches between the Lagrangians $L_p$ and $Q$ or between $L_K$ and $Q$, in either direction.

\begin{figure}
\labellist
\small\hair 2pt
\pinlabel ${\color{red} Q}$ at 206 122
\pinlabel ${\color{green} L_p}$ at 64 50
\pinlabel ${\color{blue} L_K}$ at 360 50
\pinlabel $p$ at 70 122
\pinlabel $K$ at 352 122
\pinlabel $a$ at 206 275
\pinlabel ${\color{magenta} s}$ at 169 157
\endlabellist
\centering
\includegraphics[width=0.4\textwidth]{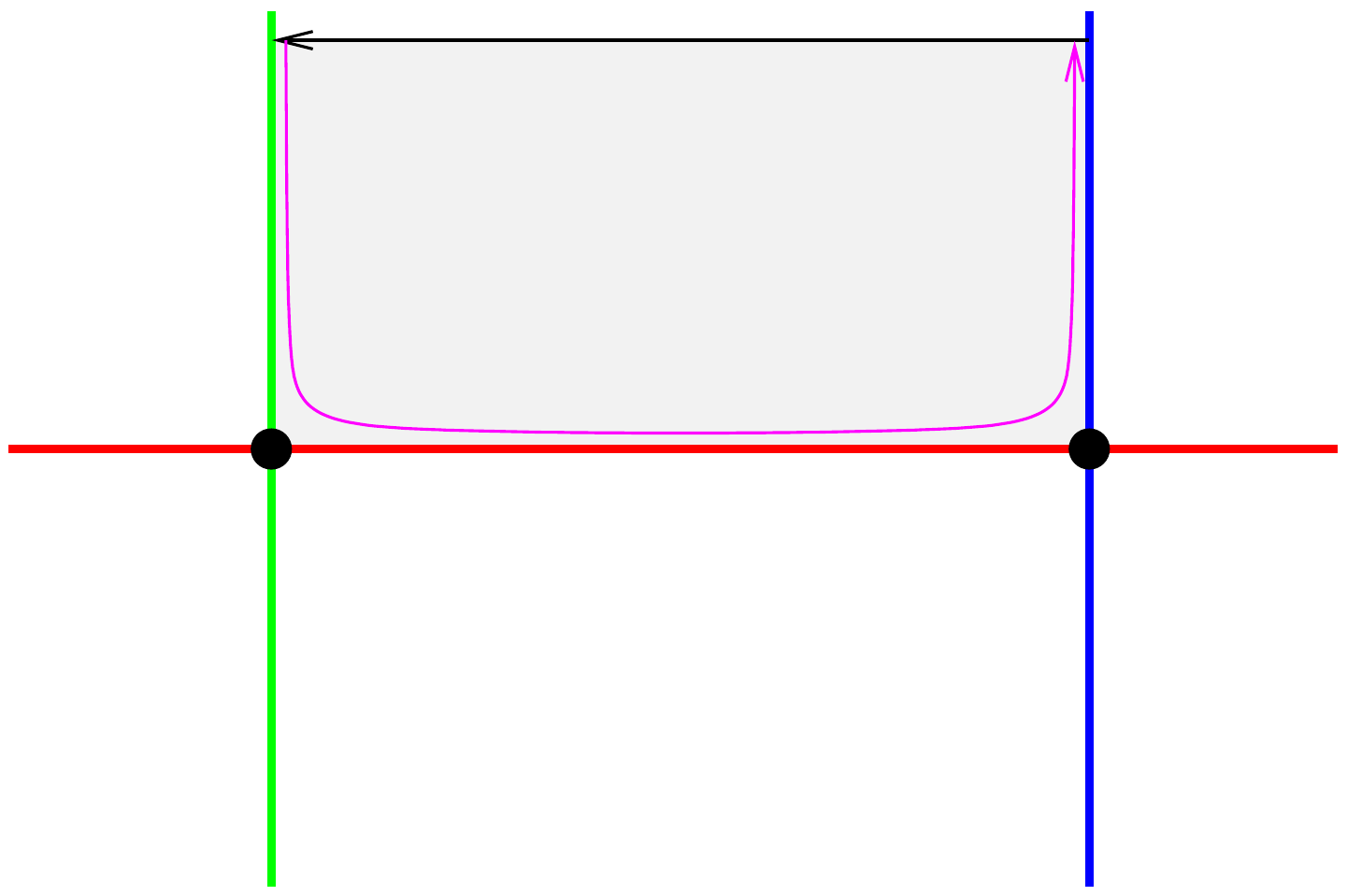}
\caption{A Reeb chord $a$ to $\Lambda_p$ from $\Lambda_K$ and a holomorphic disk in $\mathcal{M}^{\rm sw}(a)$ whose boundary is the depicted broken string $s$.}
\label{fig:cotangent-string}
\end{figure}
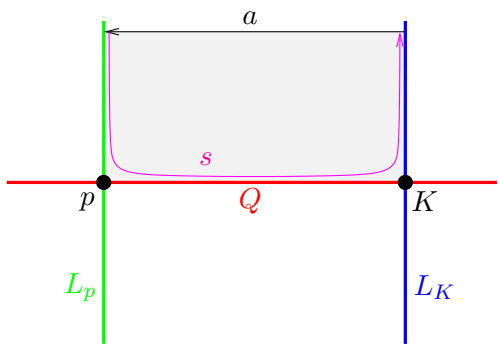

The boundary of a disk in $\mathcal{M}^{\rm sw}(a)$, oriented counterclockwise, is a broken string in $Q \cup N \cup L_p$; see Figure~\ref{fig:cotangent-string}. More precisely, each endpoint of a Reeb chord of $\Lambda_K \cup \Lambda_p$ is a point in $\Lambda_K \cup \Lambda_p$; fix paths in $L_p$ or $N$ that connect these points to the base points $(p,\xi)$ or $(x_0,\xi_0)$. Then the union of the boundary of a disk in $\mathcal{M}^{\rm sw}(a)$ and the paths for the endpoints of $a$ is a broken string.

We can stratify $\mathcal{M}^{\rm sw}(a)$ by the number of switches at $p$: for $\ell \geq 0$, let $\mathcal{M}^{\rm sw}_\ell(a)$ denote the subset of $\mathcal{M}^{\rm sw}(a)$ of disks with $\ell$ switches at $p$. The moduli space $\mathcal{M}^{\rm sw}_\ell(a)$ is an oriented $C^{1}$-manifold and we let $[\mathcal{M}^{\rm sw}_\ell(a)]$ denote the chain of broken strings in $\Sigma_\ell$ carried by this moduli space (that is, the chain given by the boundaries of disks in the moduli space).
Now define $\Phi\colon\A_{\Lambda_{K}\cup\Lambda_p}\to\tilde C_{\ast}$ by
\[
\Phi(a)=\sum_\ell [\mathcal{M}^{\rm sw}_\ell(a)].
\]

\begin{proposition}
The map
\[
\Phi :\thinspace (\A_{\Lambda_K \cup \Lambda_p},\partial) \to
(\tilde{C}_*,\partial+\delta_{Q}^K+\delta_{N}^K+\delta_{Q}^{p}+\delta_{L_p}^p)
\]
is a degree zero chain map of differential graded algebras, where multiplication on $\tilde{C}_*$
is given by chain-level concatenation of broken strings.
\end{proposition}

\begin{proof}
The proof is very similar to \cite[Proposition~5.8]{CELN}. We first check that the map $\Phi$ has degree $0$.
The dimension of $\mathcal{M}^{\rm sw}_\ell(a)$ is
\[
\dim(\mathcal{M}^{\rm sw}_\ell(a)) = |a| - \ell',
\]
where $\ell'$ is the number switches at $p$ along the boundary where the boundary switches from $L_{p}$ to $Q$.
To see this, 
recall from Remark~\ref{r:rotations} that the contribution to the dimension formula is $0$ for punctures switching from $Q$ to $L_{p}$ at $p$ and $-1$ for the puncture switching from $L_{p}$ to $Q$.

We now have three cases. If $a$ joins $\Lambda_K$ to itself, then $\ell= 2\ell'$ and
\[
[\mathcal{M}^{\rm sw}_\ell(a)] \in C_{|a|-\ell/2}(\Sigma_\ell^{KK}) \subset \tilde{C}_{|a|}^{KK}.
\]
If $a$ goes to $\Lambda_K$ from $\Lambda_p$, then if we traverse the boundary of a disk in $\mathcal{M}^{\rm sw}_\ell(a)$ beginning at the positive puncture, we begin on $N$, then alternately switch to and from $L_p$, and end on $L_p$; thus $\ell = 2\ell'+1$ and
\[
[\mathcal{M}^{\rm sw}_\ell(a)] \in C_{|a|-(\ell-1)/2}(\Sigma_\ell^{Kp}) \subset \tilde{C}_{|a|}^{Kp}.
\]
Finally, if $a$ goes to $\Lambda_p$ from $\Lambda_K$, then the same argument gives $\ell = 2\ell'-1$ and
\[
[\mathcal{M}^{\rm sw}_\ell(a)] \in C_{|a|-(\ell+1)/2}(\Sigma_\ell^{pK}) \subset \tilde{C}_{|a|}^{pK}.
\]
In all cases we find that $\Phi$ preserves degree.

We next study the chain map equation. To this end we must understand the codimension $1$ boundary of $\mathcal{M}^{\rm sw}(a)$ which contributes the singular boundary $\partial\Phi(a)$. This boundary consists of three parts:
\begin{itemize}
\item[$(i)$]	
$2$-level disks with one level of dimension $\dim(\mathcal{M}^{\rm sw}(a))-1$ and a level of dimension 1 in the symplectization end;
\item[$(ii)$]
$1$-level disks in which a boundary arc in $Q$ or in $L_K$ shrinks to a point in $K$, or equivalently a disk with one boundary arc that hits $K$ in an interior point;
\item[$(iii)$]
$1$-level disks in which a boundary arc in $Q$ or in $L_p$ shrinks to a point at $p$, or equivalently a disk with one boundary arc that hits $p$ in an interior point.  	
\end{itemize}
Configurations of type $(i)$ are counted by $\Phi(\partial a)$, configurations of type $(ii)$ by $(\delta_{Q}^{K}+\delta_N^{K})\Phi(a)$, and configurations of type $(iii)$ by $(\delta_{Q}^{p}+\delta_N^{p})\Phi(a)$. The chain map equation follows.
\end{proof}

We will be especially interested in the subcomplexes $\tilde{C}_*^{KK}$,
$\tilde{C}_*^{Kp}$, $\tilde{C}_*^{pK}$ in the lowest degree. These
are given as follows, where the differential is $d = \partial +
\delta_Q^K + \delta_N^K$ (the operations $\delta_Q^p$,
$\delta_{L_p}^p$ do not appear for degree reasons):
\begin{align*}
\tilde{C}_1^{KK} = C_1^{KK}(\Sigma_0) \oplus
C_0^{KK}(\Sigma_2) &\stackrel{d}{\longrightarrow}
\tilde{C}_0^{KK} = C_0^{KK}(\Sigma_0) \\
\tilde{C}_1^{Kp} = C_1^{Kp}(\Sigma_1) \oplus
C_0^{Kp}(\Sigma_3) &\stackrel{d}{\longrightarrow}
\tilde{C}_0^{Kp} = C_0^{Kp}(\Sigma_1) \\
\tilde{C}_2^{pK} = C_1^{pK}(\Sigma_1) \oplus
C_0^{pK}(\Sigma_3) &\stackrel{d}{\longrightarrow}
\tilde{C}_1^{pK} = C_0^{pK}(\Sigma_1).
\end{align*}
Note that $d$ acts on the first summand in each case, and is $0$ on
the second.

A main result from \cite{CELN} is that $\Phi$ induces an isomorphism in
degree $0$ homology. In our setting, this becomes the following:

\begin{proposition}\label{p:stringtopiso}
The map $\Phi$ induces isomorphisms
\label{prop:Phi-isom}
\begin{align*}
R_{KK} = H_0(\A_{\Lambda_K,\Lambda_K})
  &\stackrel{\cong}{\longrightarrow}
H_0(\tilde{C}_*^{KK},d) = \coker \left(\partial+\delta_Q^K+\delta_N^K :\thinspace
C_1^{KK}(\Sigma_0) \to C_0^{KK}(\Sigma_0)\right) \\
R_{Kp} = H_0(\A_{\Lambda_K,\Lambda_p})
&\stackrel{\cong}{\longrightarrow}
H_0(\tilde{C}_*^{Kp},d) = \coker \left(\partial+\delta_Q^K+\delta_N^K :\thinspace
C_1^{Kp}(\Sigma_1) \to C_0^{Kp}(\Sigma_1)\right) \\
R_{pK} = H_1(\A_{\Lambda_p,\Lambda_K})
  &\stackrel{\cong}{\longrightarrow}
H_1(\tilde{C}_*^{pK},d) = \coker \left(\partial+\delta_Q^K+\delta_N^K :\thinspace
C_1^{pK}(\Sigma_1) \to C_0^{pK}(\Sigma_1)\right).
\end{align*}
\end{proposition}

\begin{proof}
The isomorphism for $R_{KK}$ is proven in \cite[\S 7]{CELN} via an action/length filtration argument, and the other isomorphisms use exactly the same argument.
A short description of the argument is as follows. A length filtration on chains of broken strings given by the supremum norm of the sum of the lengths of the $Q$-strings is introduced. On the DGA there is the action filtration and for a suitable choice of almost complex structure on $T^{\ast}\R^{3}$ the chain map $\Phi$ respects this filtration. A standard approximation argument shows that the string homology complex is quasi-isomorphic to the string homology complex of piecewise linear broken strings. On the complex of piecewise linear broken strings, a length-decreasing flow (with splittings when the segments cross the knot) then deforms the complex to a complex generated by certain chains associated to binormal chords and using basic holomorphic strips over binormal chords and the action/length filtrations then shows that $\Phi$ is a quasi-isomorphism.
\end{proof}

\begin{remark}
It is likely that $\Phi$ is in fact an isomorphism in all degrees. The reason that we restrict to the lowest degree ($0$ for $\A_{\Lambda_K,\Lambda_K}$ and $\A_{\Lambda_K,\Lambda_p}$ and $1$ for $\A_{\Lambda_p,\Lambda_K}$) here, following the same restriction in \cite{CELN}, is that the proof of the isomorphism in \cite{CELN}
involves an explicit examination of moduli spaces of holomorphic disks with switching boundary conditions of dimensions $\le 2$. To extend the isomorphism to higher degrees would require one to work out the relevant string homology in degree $d+2$, imposing conditions at endpoints of the strings that match degenerations in higher dimensional moduli spaces of holomorphic disks, and this has not been worked out for moduli spaces of dimensions $\geq 3$.
\end{remark}

\begin{remark}
In \cite{CELN}, $\coker \left(\partial+\delta_Q^K+\delta_N^K :\thinspace
C_1^{KK}(\Sigma_0) \to C_0^{KK}(\Sigma_0)\right)$ is written as
``string homology'' $H_0^\text{string}(K)$, and the first isomorphism in
Proposition~\ref{prop:Phi-isom} states
that $H_0^\text{string}(K)$ is isomorphic to knot contact homology in
degree $0$. A variant of this construction, modified string homology
$\tilde{H}_0^\text{string}(K)$, is also considered in \cite[\S
2]{CELN}, and it is observed there that $\tilde{H}_0^\text{string}(K) \cong
\Z[\pi_1(\R^3\setminus K)]$. In our language, modified string homology
is defined by
\[
\tilde{H}_0^\text{string}(K) = \coker \left(\partial+\delta_Q^K+\delta_N^K :\thinspace
C_1^{pp}(\Sigma_2) \to C_0^{pp}(\Sigma_2)\right)
\]
and the above map is part of the differential $d :\thinspace \tilde{C}_2^{pp} \to \tilde{C}_1^{pp}$.

Two things prevent us from using modified string homology to show that $\LCH_*(\Lambda_K \cup \Lambda_p)$ is a complete invariant. First, $\tilde{H}_0^\text{string}(K)$ is not directly a summand of the homology of $\tilde{C}_*$, although it does map to $\tilde{C}_1^{pp}$. Second, the isomorphism of 
$\tilde{H}_0^\text{string}(K)$ with $\Z[\pi_1(\R^3\setminus K)]$ is as a $\Z$-module, without the product structure. One could try to recover the product on $\Z[\pi_1(\R^3\setminus K)]$, which is crucial to recovering the knot group itself, via 
the natural concatenation product on $\tilde{C}_*^{pp}$, but this sends $\tilde{C}_1^{pp} \otimes \tilde{C}_1^{pp}$ to $\tilde{C}_2^{pp}$ rather than to $\tilde{C}_1^{pp}$.

Instead, we need a product on $\tilde{C}_*$ that (in our grading convention) reduces degree by $1$. Intuitively this is given by concatenating a broken string ending at $p$ and a broken string beginning at $p$, and deleting the two switches at $p$. More precisely, we use the Pontryagin product; we discuss this product and its holomorphic-curve counterpart next.

\end{remark}

\subsection{String topology and the product}
\label{ssec:string-product}

Having established isomorphisms $\Phi$ in low degree between enhanced knot
contact homology and string homology, we now examine the behavior of
the product map $\mu$ under this isomorphism. Recall from
Section~\ref{ssec:product} that $\mu$ is a map
\[
\mu :\thinspace
\A_{\Lambda_K,\Lambda_p}^{(1)} \otimes \A_{\Lambda_p,\Lambda_K}^{(1)} \to
\A_{\Lambda_K,\Lambda_K}^{(0)}.
\]
We will show that under the isomorphism $\Phi$, $\mu$ maps to the Pontryagin product at the base point $(p,v) \in L_p$, which we now define.

Consider two chains of broken strings in $C^{Kp}_{k_1}(\Sigma_{\ell_1})$ and $C^{pK}_{k_2}(\Sigma_{\ell_2})$. We define their Pontryagin product at $(p,v)$ as the concatenation at $(p,v)$ followed by
removing the path between the switches at $p$ that precede and
follow this concatenation.
This gives a map
\[
C_{k_1}^{Kp}(\Sigma_{\ell_1}) \otimes
C_{k_2}^{pK}(\Sigma_{\ell_2}) \to
C_{k_1+k_2}^{KK}(\Sigma_{\ell_1+\ell_2-2}).
\]
Summing over integers $k_1,k_2$ and half-integers $\ell_1,\ell_2$, we
get the Pontryagin product at $p$
\[
P :\thinspace \tilde{C}_*^{Kp} \otimes \tilde{C}_*^{pK} \to
\tilde{C}_*^{KK}
\]
which has degree $-1$.

We now treat the relation between $P$ and $\mu$. Note that $\Phi :\thinspace \A_{\Lambda_K \cup \Lambda_p}
\to \tilde{C}_*$ induces maps
\begin{align*}
\A_{\Lambda_K,\Lambda_p}^{(1)} &\to \tilde{C}_*^{Kp}, &
\A_{\Lambda_p,\Lambda_K}^{(1)} &\to \tilde{C}_*^{pK}.
\end{align*}
We claim that these $\Phi$ maps intertwine $P$ and $\mu$ on the level of homology. This is not true on the chain level, but the
difference can be measured by a map $\Psi$ that we now define.

If $a_1$ and $a_2$ are Reeb chords to $\Lambda_{K}$ from $\Lambda_{p}$
and to $\Lambda_{p}$ from $\Lambda_{K}$, respectively, then we write
$\M^{\rm sw}(a_1,a_2)$ for the moduli space of holomorphic disks
$u\colon D\to T^*Q$ that send
positive punctures at $-1$ and $1$ to $a_1$ and $a_2$ at
infinity, the arc in the upper half plane connecting these punctures
to $L_p$, and the arc in the lower half plane to $Q \cup L_K \cup L_p$.
We can stratify $\M^{\rm sw}(a_1,a_2)$ by the number of switches that
the boundary of a holomorphic disk has at $p$; for $\ell\geq 0$ even, write
$\M^{\rm sw}_\ell(a_1,a_2)$ for the subset of $\M^{\rm sw}(a_1,a_2)$
corresponding to disks with $\ell$ switches at $p$. The formal dimension
of this moduli space is, see Remark \ref{r:gendim},
\[
\dim(\M^{\rm sw}_\ell(a_1,a_2))=|a_1|+|a_2|-\ell/2.
\]

With notation as in Section~\ref{ssec:product}, define a map
\[
\Psi\colon\thinspace
\A_{\Lambda_K,\Lambda_p}^{(1)}
\otimes
\A_{\Lambda_p,\Lambda_K}^{(1)}
\to
C^{KK}_*
\]
as follows:
\[
\Psi(\mathbf{b}_1a_1\otimes a_2\mathbf{b}_2)=
\sum_{|a_1|+|a_2|-\ell/2=0} \Phi(\mathbf{b}_{1})\cdot[\M^{\rm sw}_\ell(a_1,a_2)]\cdot\Phi(\mathbf{b}_2),
\]
where $[\M^{\rm sw}_\ell(a_1,a_2)]$ denotes the chain in
$C^{KK}_0(\Sigma_\ell) \subset C_{\ell/2}^{KK} = \tilde{C}_{\ell/2}^{KK}$ carried by the moduli
space (i.e., the chain of broken strings corresponding to the disks in
$\M^{\rm sw}_\ell(a_1,a_2)$) and where $\cdot$ denotes the concatenation product: given broken strings in $\Phi(\mathbf{b}_1)$,
$\M^{\rm sw}_\ell(a_1,a_2)$, and $\Phi(\mathbf{b}_2)$, we concatenate the three to obtain another broken string.

\begin{remark}
To see that $[\M^{\rm sw}_\ell(a_1,a_2)]$ is a chain of broken strings we use \cite[Theorem 1.2]{CEL} which implies that there is a uniform bound on the number of switches on the boundary of a disk in any moduli space with two positive punctures.
\end{remark}

\begin{proposition}

On $\A_{\Lambda_K,\Lambda_p}^{(1)}\otimes\A_{\Lambda_p,\Lambda_K}^{(1)}$
\label{prop:Psi-mu}
we have the following:
\[
\Phi\circ \mu - P\circ(\Phi\otimes\Phi)
+ \Psi\circ(1\otimes\partial+\partial\otimes 1)
-(\partial+\delta_{Q}^{K}+\delta_{N}^{K}+\delta_Q^p+\delta_{L_p}^p)\circ\Psi=0.
\]
\end{proposition}

\begin{proof}
To see this we note that the codimension one boundary $\partial\M^{\rm sw}(a_1,a_2)$ consists of the following breakings:
\begin{itemize}
\item Two-level disks with one level in the symplectization of dimension one and one level in the cotangent bundle. These are accounted for by the first and third terms and $\partial$ in the last term.
\item Lagrangian intersection breaking at $K$, accounted for by the operations $\delta_Q^{K}+\delta_{N}^{K}$ in the last term.
\item Lagrangian intersection breaking at $p$ in the upper half disk,
  accounted for by the second term.
\item
Lagrangian intersection breaking at $p$ in the lower half disk,
accounted for by the operations $\delta_Q^p+\delta_{L_p}^p$ in the
last term.
\end{itemize}
The formula follows.
\end{proof}

We can now assemble our results in low degree into the following result.
Recall that
$R_{KK} \cong H_0(\A_{\Lambda_K,\Lambda_K}^{(0)})$, $R_{Kp} \cong H_0(\A_{\Lambda_K,\Lambda_p}^{(1)})$, and
$R_{pK} \cong H_1(\A_{\Lambda_p,\Lambda_K}^{(1)})$.
From Proposition~\ref{prop:Phi-isom}, we have an isomorphism
$\Phi :\thinspace R_{KK} \to H_0(\tilde{C}_*^{KK},d)$. The following
is now an immediate consequence of Proposition~\ref{prop:Psi-mu}.

\begin{proposition}
The following diagram commutes: \label{prop:P-mu}
\[
\xymatrix{
R_{Kp} \otimes R_{pK} \ar[r]^{\mu} \ar[d]^{\Phi\otimes\Phi}
& R_{KK} \ar[d]^{\Phi} \\
H_0(\tilde{C}^{Kp}_*,d) \otimes
H_1(\tilde{C}^{pK}_*,d) \ar[r]^-{P} &
H_0(\tilde{C}^{KK}_*,d).
}
\]
\end{proposition}


\section{Legendrian Contact Homology and the Knot Group}
\label{sec:homotopy}

In this section, we use the isomorphism from
Section~\ref{sec:stringtop} between Legendrian contact homology and
string topology to write the KCH-triple $(R_{KK},R_{Kp},R_{pK})$
defined in Section~\ref{ssec:enhanced} in terms of the knot
group $\pi_1(\R^3\setminus K)$. This will allow us to recover the knot
group from the KCH-triple along with the product $\mu :\thinspace
R_{Kp} \otimes R_{pK} \to R_{KK}$. Along the way, we present the KCH-triple in terms of the cord algebra and deduce that enhanced knot contact homology encodes the Alexander module.

\subsection{String homology and the cord algebra}
\label{ssec:cordalg}

From Proposition~\ref{prop:Phi-isom}, we have isomorphisms between the
KCH-triple $(R_{KK},R_{Kp},R_{pK})$ and parts of the homology of the
string complex $(\tilde{C}_*,d)$. As in \cite{CELN}, we can interpret
this string homology in terms of the ``cord algebra'' of $K$,
essentially by considering only the $Q$-strings. Here we give this
cord algebra interpretation of string homology, which will allow us in
Section~\ref{ssec:brackets} to rewrite the KCH-triple in terms of the
knot group. The cord-algebra approach has the added benefit of readily
yielding the Alexander module of the knot as the homology of a certain
linearization of enhanced knot contact homology, as we will see.

We first review the cord algebra as presented in \cite[\S 2.2]{CELN}, adapted
to our purposes.
Let $K \subset Q$ be an oriented knot and $p \in Q$ be a point in the knot
complement, where $Q=\R^3$ as before. Let $K'$ be a
parallel copy of $K$ in the Seifert framing, and choose a base point
$\ast$ on $K'$.

\begin{definition}
A \textit{cord} is a continuous map $\gamma :\thinspace [0,1] \to
Q$ with $\gamma(0),\gamma(1) \in (K' \setminus \{*\}) \cup \{p\}$ and
$\gamma([0,1]) \cap
K = \emptyset$. A cord is a $KK$ (respectively $Kp$; $pK$; $pp$) cord
if $\gamma(0),\gamma(1) \in K'$ (respectively $\gamma(0)\in K'$,
$\gamma(1) = p$; $\gamma(0) = p$, $\gamma(1) \in K'$;
$\gamma(0)=\gamma(1)=p$).
\end{definition}

\begin{definition}[\cite{CELN}]
The \textit{cord algebra} of $K$, $\Cord_{KK}$,
\label{def:cordalg}
is the noncommutative
unital ring freely generated by homotopy classes of $KK$ cords and
$\Z[l^{\pm 1},m^{\pm 1}]$, modulo the following skein
relations, where the cord is drawn in red, $K$ in black, and $K'$ in gray:
\begin{enumerate}
\item \label{it:cord1}
$\raisebox{-3ex}{\includegraphics[height=7ex]{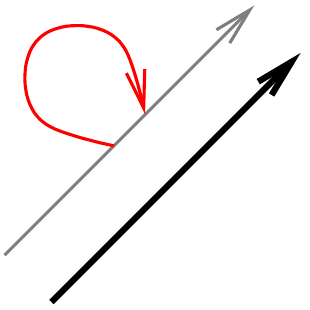}} = 1-m$
\item \label{it:cord0}
$\raisebox{-3ex}{\includegraphics[height=7ex]{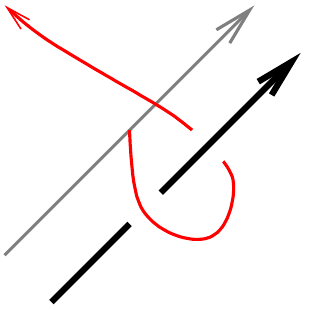}} = m
\cdot \raisebox{-3ex}{\includegraphics[height=7ex]{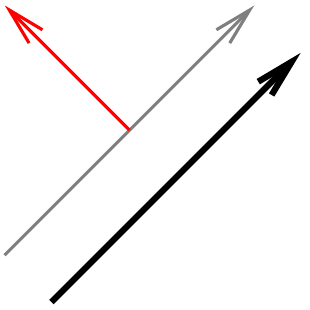}}$
\hspace{3ex} and
\hspace{3ex}
$\raisebox{-3ex}{\includegraphics[height=7ex]{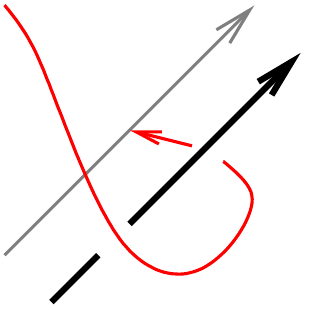}}
=\raisebox{-3ex}{\includegraphics[height=7ex]{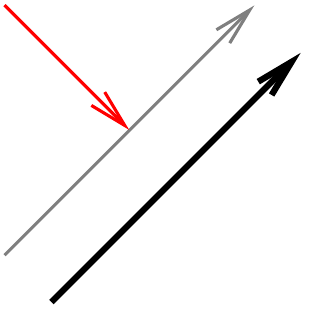}} \cdot m$
\item \label{it:cord2}
$\raisebox{-3ex}{\includegraphics[height=7ex]{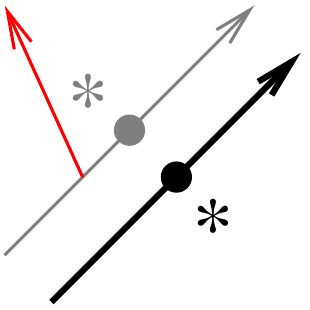}} = l
\cdot \raisebox{-3ex}{\includegraphics[height=7ex]{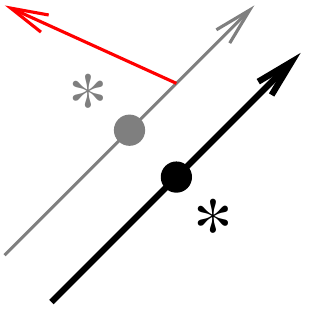}}$
\hspace{3ex} and
\hspace{3ex}
$\raisebox{-3ex}{\includegraphics[height=7ex]{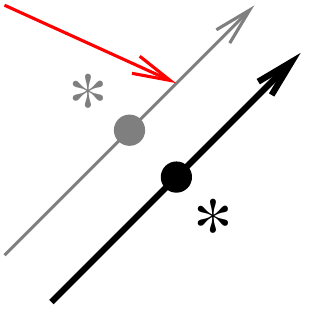}} =
\raisebox{-3ex}{\includegraphics[height=7ex]{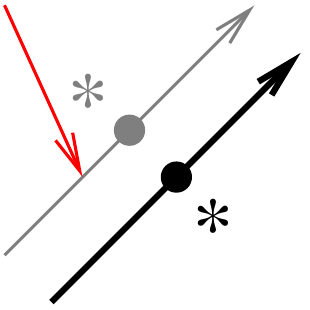}} \cdot
l$
\item \label{it:cord3}
$\raisebox{-3ex}{\includegraphics[height=7ex]{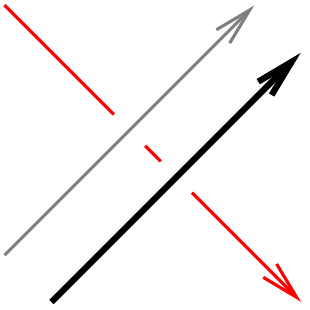}}
- \raisebox{-3ex}{\includegraphics[height=7ex]{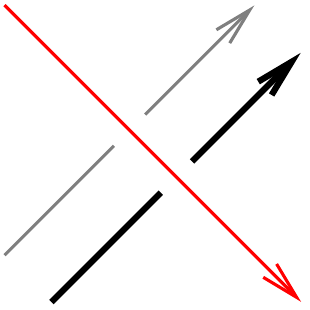}}
= \raisebox{-3ex}{\includegraphics[height=7ex]{skein3c}} \cdot
\raisebox{-3ex}{\includegraphics[height=7ex]{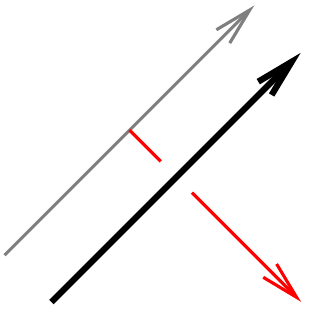}}$.
\end{enumerate}
Note that a typical element of $\Cord_{KK}$ is a linear combination of products of cords and elements of $\Z[l^{\pm 1},m^{\pm 1}]$, and multiplication in $\Cord_{KK}$ is given by formal concatenation of products.
\end{definition}

We can extend Definition~\ref{def:cordalg} to cover cords with endpoints at $p$ as well, where the relations only apply near $K$ and do not affect the ends at $p$. Note that the skein relations may still involve $KK$ cords:
for example, if the beginning and end points of the cords on the left hand side of \eqref{it:cord3} lie at $K$ and $p$ respectively, then 
\eqref{it:cord3} gives a relation between two
$Kp$ cords (the left hand side) and a product of a $KK$ cord and a
$Kp$ cord (the right hand side).

\begin{definition}
The $Kp$ \textit{cord module} of $K$, $\Cord_{Kp}$, is the left
$\Cord_{KK}$-module freely generated by $Kp$ cords, modulo the skein
relations \eqref{it:cord0}, \eqref{it:cord2}, and \eqref{it:cord3}
from Definition~\ref{def:cordalg}. Similarly, the $pK$ \textit{cord
  module} of $K$, $\Cord_{pK}$, is the right $\Cord_{KK}$-module freely
generated by $pK$, modulo the same skein relations.
\end{definition}

Now a broken string in $C_0^{KK}(\Sigma_0)$ (respectively
$C_0^{Kp}(\Sigma_1)$, $C_0^{pK}(\Sigma_1)$) produces an element of
$\Cord_{KK}$ (respectively $\Cord_{Kp}$, $\Cord_{pK}$) given by the
product of the $Q$-strings taken in order. This map induces maps from string homology to
the cord algebra and modules, and as in \cite{CELN}
we can show that these maps are isomorphisms. Combined with Proposition~\ref{prop:Phi-isom}, this shows that the cord algebra and modules are isomorphic to the KCH-triple $(R_{KK},R_{Kp},R_{pK})$, and this is the fact that we will exploit in this section to prove Theorem~\ref{thm:main}.

\begin{proposition}
There are isomorphisms
\begin{align*}
\Cord_{KK} &\cong H_0(\tilde{C}_*^{KK},d) \cong R_{KK} \\
\Cord_{Kp} &\cong H_0(\tilde{C}_*^{Kp},d) \cong R_{Kp} \\
\Cord_{pK} &\cong H_1(\tilde{C}_*^{pK},d) \cong R_{pK}
\end{align*}
where the first line is a ring isomorphism,
\label{prop:cord-triple}
and the second and third
lines send the left and right actions of $\Cord_{KK}$ to the left and
right actions of $R_{KK}$. Under these isomorphisms, the map $\mu :\thinspace R_{Kp} \otimes R_{pK} \to R_{KK}$ is the concatenation map
\[
\Cord_{Kp} \otimes \Cord_{pK} \to \Cord_{KK}.
\]
\end{proposition}

\begin{proof}
The first line is proved in Proposition 2.9 of \cite{CELN}, and the
other two lines have the same proof. The fact that
these isomorphisms preserve multiplication follows formally from the
construction of the cord algebra and modules. The description of $\mu$ as a concatenation product is a direct consequence of Proposition~\ref{prop:P-mu}.
\end{proof}

\subsection{Enhanced knot contact homology and the Alexander module}
\label{ssec:Alexander}

Here we digress from the main argument to observe that we can use the cord modules to recover the Alexander module $H_1(\tilde{X}_K)$ of
$K$, where $\tilde{X}_K$ is the infinite cyclic cover of $\R^3\setminus K$ and $H_1(\tilde{X}_K)$ is viewed as a $\Z[m^{\pm 1}]$-module as usual by deck transformations. As a consequence, we show that a certain canonical linearization of enhanced knot contact homology contains the Alexander module and thus the Alexander polynomial.

It was previously known \cite{Ngframed} that the Alexander module can be extracted from the same linearization of usual knot contact homology $\LCH_*(\Lambda_K)$, but in a somewhat obscure way---essentially, the degree $1$ linearized homology is the second tensor product of $H_1(\tilde{X}_K) \oplus \Z[m^{\pm 1}]$, with the proof involving an examination of the combinatorial form of the DGA of $\Lambda_K$ in terms of a braid representative for $K$, and a relation to the Burau representation. Here we will see that with the introduction of the fiber $\Lambda_p$ alongside $\Lambda_K$, we can instead deduce the Alexander module in a significantly simpler way. In particular, we will use linearized homology not in degree $1$ but in degree $0$, which is more geometrically natural (for instance, it relates more easily to the cord algebra).

We first present a variant of the cord algebra and modules, following \cite{Ngframed} and especially the discussion in \cite[\S 2.2]{CELN}. Choose a base point $\ast$ on $K$ corresponding to the base point $\ast$ on $K'$.
Let an \textit{unframed cord} of $K$ be a path whose endpoints are in $(K \setminus \{\ast\}) \cup \{p\}$ and which is disjoint from $K$ in its interior; we can divide these into $KK$, $Kp$, $pK$, $pp$ cords depending on where the endpoints lie.

\begin{definition}[\cite{Ngframed}]
The \textit{unframed cord algebra} of $K$, $\Cord'_{KK}$,
\label{def:unframedcordalg}
is the noncommutative algebra over $\Z[l^{\pm 1},m^{\pm 1}]$ generated by homotopy classes of unframed $KK$ cords, modulo the following skein relations:
\begin{enumerate}
\item \label{it:unframedcord1}
$\raisebox{-3ex}{\includegraphics[height=7ex]{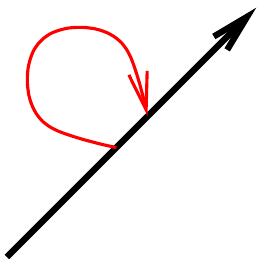}} = 1-m$
\item \label{it:unframedcord2}
$\raisebox{-3ex}{\includegraphics[height=7ex]{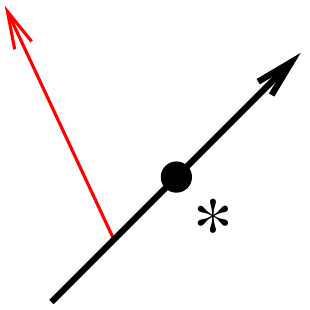}} = l
\cdot \raisebox{-3ex}{\includegraphics[height=7ex]{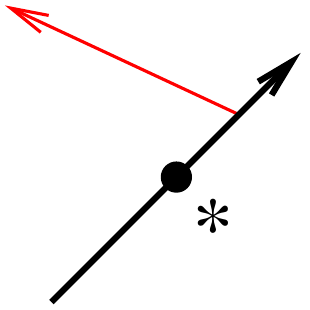}}$
\hspace{3ex} and
\hspace{3ex}
$\raisebox{-3ex}{\includegraphics[height=7ex]{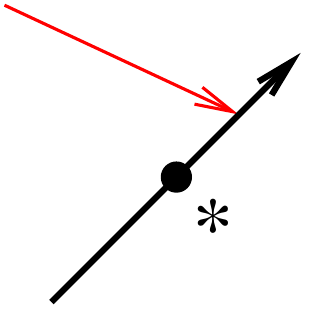}} =
\raisebox{-3ex}{\includegraphics[height=7ex]{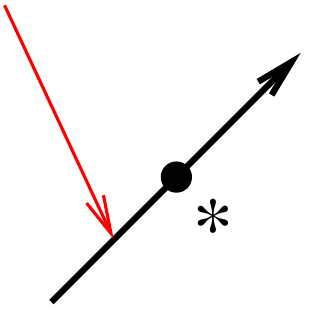}} \cdot
l$
\item \label{it:unframedcord3}
$\raisebox{-3ex}{\includegraphics[height=7ex]{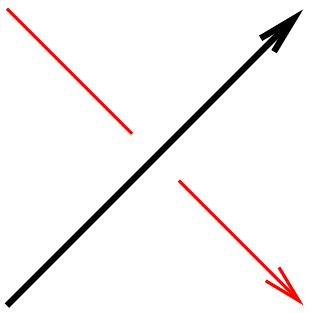}}
- m \raisebox{-3ex}{\includegraphics[height=7ex]{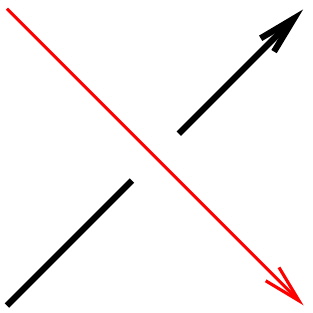}}
= \raisebox{-3ex}{\includegraphics[height=7ex]{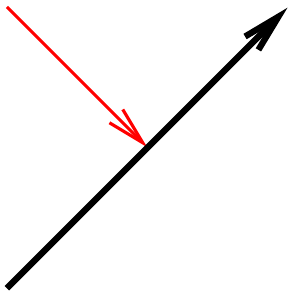}} \cdot
\raisebox{-3ex}{\includegraphics[height=7ex]{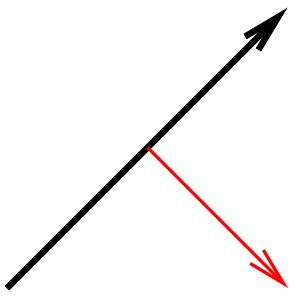}}$.
\end{enumerate}
The \textit{unframed Kp (respectively pK) cord module} of $K$, $\Cord'_{Kp}$ (respectively $\Cord'_{pK}$), is the right (respectively left) $\Cord'_{KK}$-module generated by unframed $Kp$ (respectively $pK$) cords, modulo the skein relations \eqref{it:unframedcord2} and \eqref{it:unframedcord3}.
\end{definition}

Note that $\Cord'_{KK}$, $\Cord'_{Kp}$, and $\Cord'_{pK}$ are all $\Z[l^{\pm 1},m^{\pm 1}]$-modules, unlike their framed counterparts $\Cord_{KK}$, $\Cord_{Kp}$, $\Cord_{pK}$, where elements of $\Z[l^{\pm 1},m^{\pm 1}]$ do not necessarily commute with cords. However, we have the following.

\begin{proposition}
The unframed cord algebra and modules $\Cord'_{KK}$, $\Cord'_{Kp}$, $\Cord'_{pK}$ are isomorphic
\label{prop:framed-unframed}
to the quotients of the cord algebra and modules $\Cord_{KK}$, $\Cord_{Kp}$, $\Cord_{pK}$ obtained by imposing the relations that elements of $\Z[l^{\pm 1},m^{\pm 1}]$ commute with cords.
\end{proposition}

\begin{proof}
This is essentially laid out in \cite[\S 2.2]{CELN}. Fix a cord $\gamma_0$ from $p$ to a point $x_0 \in K' \setminus \{\ast\}$. Given any cord $\gamma$, we can produce a loop $\tilde{\gamma}$ in $\R^3\setminus K$ based at $p$, by joining any endpoint of $\gamma$ on $K'$ to $x_0$ along (any path in) $K'$, and appending $\gamma_0$ or $-\gamma_0$ as necessary. Let $\gamma'$ be the unframed cord obtained from $\gamma$ by joining any endpoint of $\gamma$ on $K'$ to the corresponding point on $K$ by a straight line segment normal to $K$. Then the map
\[
\gamma \mapsto m^{-\operatorname{lk}(\tilde{\gamma},K)} \gamma'
\]
gives the desired isomorphisms from the quotients of $\Cord_{KK}$, $\Cord_{Kp}$, $\Cord_{pK}$ to $\Cord'_{KK}$, $\Cord'_{Kp}$, $\Cord'_{pK}$. (For the inverse maps from $\Cord'$ to $\Cord$, homotope any cord with a beginning or end point on $K$ so that it begins or ends with $\gamma_0$ or $-\gamma_0$, and then remove $\pm \gamma_0$.)
Note that the displayed map from $\Cord$ to $\Cord'$ sends the skein relations \eqref{it:cord1}, \eqref{it:cord2}, \eqref{it:cord3} in Definition~\ref{def:cordalg} to \eqref{it:unframedcord1}, \eqref{it:unframedcord2}, \eqref{it:unframedcord3} in Definition~\ref{def:unframedcordalg}, and the normalization by powers of $m$ means that \eqref{it:cord0} from Definition~\ref{def:cordalg} becomes trivial under this map.
\end{proof}

Now from \cite{Ngframed}, there is a canonical augmentation of the DGA for $K$,
\[
\epsilon :\thinspace (\A_{\Lambda_K},\partial) \to
(\Z[m^{\pm 1}],0),
\]
whose definition we recall here. Since $\A_{\Lambda_K}$ is supported in nonnegative degree, the graded map $\epsilon$ is determined by its action on the degree $0$ part of $\A_{\Lambda_K}$, or equivalently (since $\epsilon \circ \partial = 0$) by the induced action on $H_0(\A_{\Lambda_K},\partial)$. This in turn is determined by the induced action on $\Cord'_{KK}$, which by Proposition~\ref{prop:framed-unframed} is the quotient of $H_0(\A_{\Lambda_K},\partial)$ by setting $l,m$ to commute with everything. On $\Cord'_{KK}$, $\epsilon$ is defined as follows:
\begin{align*}
\epsilon(l) &= 1 \\
\epsilon(m) &= m \\
\epsilon(\gamma) &= 1-m
\end{align*}
for any unframed $KK$ cord $\gamma$. (Note that $\epsilon$ preserves the skein
relations for $\Cord'_{KK}$ and is thus well-defined.) We can extend
$\epsilon$ from an augmentation of $\A_{\Lambda_K}$ to an augmentation of $\A_{\Lambda_K \cup \Lambda_p}$ by
setting $\epsilon$ to be $0$ for any mixed chord between $\Lambda_K$
and $\Lambda_p$.
\begin{remark}
Applying \cite[Theorem 6.15]{AENV} to the holomorphic strips over binormal chords shows that the augmentation $\epsilon$ is induced by an exact Lagrangian filling $M_K$ diffeomorphic to the knot complement, obtained by joining the conormal $L_K$ and the zero-section $Q$ via Lagrange surgery along the knot $K$.
\end{remark}

Linearizing with respect to this augmentation gives
the linearized contact homology
\[
\LCH^\epsilon_*(\Lambda_K \cup \Lambda_p) =
(\LCH^\epsilon_*)_{\Lambda_K,\Lambda_K} \oplus
(\LCH^\epsilon_*)_{\Lambda_K,\Lambda_p} \oplus
(\LCH^\epsilon_*)_{\Lambda_p,\Lambda_K} \oplus
(\LCH^\epsilon_*)_{\Lambda_p,\Lambda_p}.
\]

As discussed previously, in \cite{Ngframed} it is shown that
$(\LCH^\epsilon_1)_{\Lambda_K,\Lambda_K}$ recovers the Alexander
module $H_1(\tilde{X}_K)$. Here instead we have the following.

\begin{proposition}
We have isomorphisms of $\Z[m^{\pm 1}]$-modules
\[
(\LCH^\epsilon_0)_{\Lambda_K,\Lambda_p} \cong
(\LCH^\epsilon_1)_{\Lambda_p,\Lambda_K} \cong H_1(\tilde{X}_K) \oplus
\Z[m^{\pm 1}].
\]
\end{proposition}

\begin{proof}
We will prove the isomorphism for
$(\LCH^\epsilon_1)_{\Lambda_p,\Lambda_K}$; the isomorphism for
$(\LCH^\epsilon_0)_{\Lambda_K,\Lambda_p}$ follows by symmetry between
$\Cord'_{Kp}$ and $\Cord'_{pK}$. The complex whose homology computes
$(\LCH^\epsilon_*)_{\Lambda_p,\Lambda_K}$ is the
free $\Z[m^{\pm 1}]$-module generated by Reeb chords to $\Lambda_p$
from $\Lambda_K$, with the differential $\d^{\text{lin}}$ given by applying the
augmentation $\epsilon$ to all pure Reeb chords from $\Lambda_K$ to
itself to the usual differential $\d$. In particular, since the degree $1$ homology $(\LCH^\epsilon_1)_{\Lambda_p,\Lambda_K}$ is the quotient of the $\Z[m^{\pm 1}]$-module generated by degree $1$ Reeb chords to $\Lambda_p$ from $\Lambda_K$ by the image of $\d^{\text{lin}}$, we have:
\[
(\LCH^\epsilon_1)_{\Lambda_p,\Lambda_K}
\cong R_{pK} \otimes_\epsilon \Z[m^{\pm 1}]
\cong \Cord'_{pK} \otimes_\epsilon \Z[m^{\pm 1}].
\]
Here by ``$\otimes_\epsilon$'' we mean $\otimes_{R_{KK}}$ (or $\otimes_{\Cord'_{KK}}$) where we use $\epsilon$ to give $\Z[m^{\pm 1}]$ the structure of an $R_{KK}$-module (or $\Cord'_{KK}$-module), and
implicitly we are setting $l,m$ to commute with everything in
$(\LCH^\epsilon_1)_{\Lambda_p,\Lambda_K}$ and $R_{pK}$.

\begin{figure}
\labellist
\small\hair 2pt
\pinlabel ${\color{green} p}$ at 3 67
\pinlabel ${\color{green} p}$ at 456 153
\pinlabel $K$ at 183 124
\pinlabel $K$ at 240 122
\pinlabel $K$ at 484 37
\pinlabel ${\color{red} \gamma_1}$ at 160 32
\pinlabel ${\color{red} \gamma_2}$ at 173 81
\pinlabel ${\color{red} \gamma_3}$ at 134 62
\pinlabel ${\color{red} \gamma_1}$ at 419 62
\pinlabel ${\color{red} \gamma_3}$ at 440 41
\pinlabel ${\color{red} \gamma_2}$ at 459 20
\endlabellist
\centering
\includegraphics[width=\textwidth]{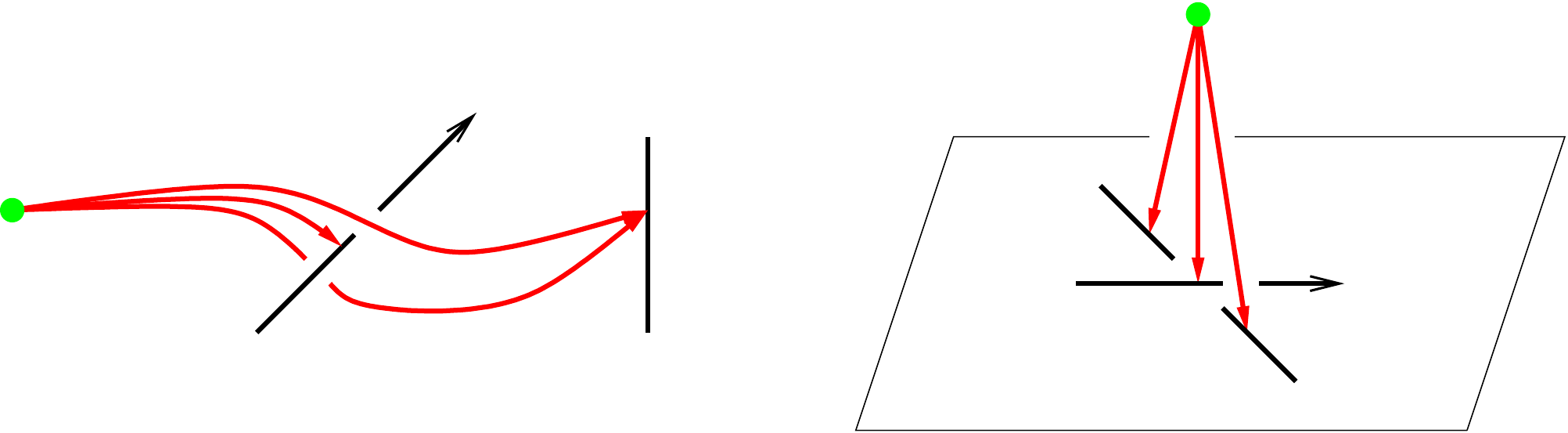}
\caption{Three $pK$ cords $\gamma_1,\gamma_2,\gamma_3$ related by
$\gamma_1-m\gamma_2 = (1-m)\gamma_3$.}
\label{fig:quandle}
\end{figure}
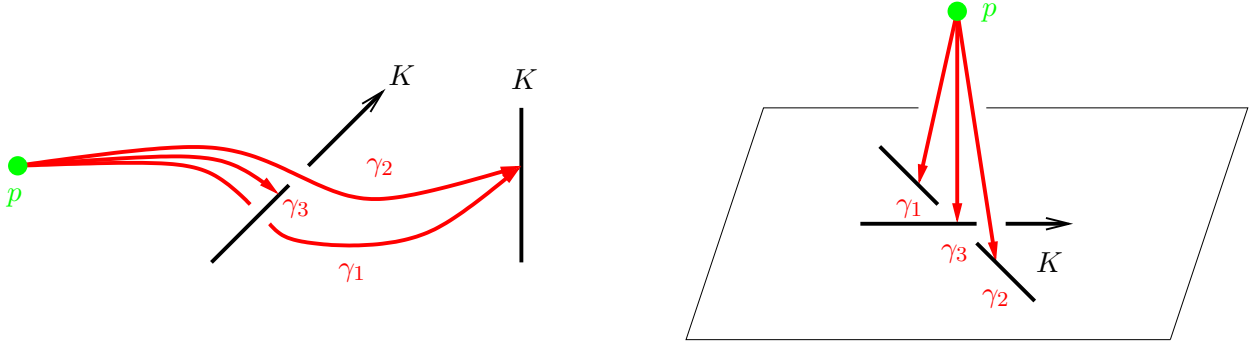

Now $\Cord'_{pK} \otimes_\epsilon \Z[m^{\pm 1}]$ is the quotient of the
free $\Z[m^{\pm 1}]$-module generated by unframed $pK$ cords by the skein
relations \eqref{it:unframedcord2} and \eqref{it:unframedcord3} from Definition~\ref{def:unframedcordalg}, where $l$ is sent to $1$
and all $KK$ cords are sent to $1-m$. Relation \eqref{it:unframedcord2} then says that $pK$ cords are unchanged if we move their $K$ endpoint over $\ast$, while relation \eqref{it:unframedcord3} becomes:
\[
\raisebox{-3ex}{\includegraphics[height=7ex]{skein31a}}
- m \raisebox{-3ex}{\includegraphics[height=7ex]{skein31b}}
= (1-m)\raisebox{-3ex}{\includegraphics[height=7ex]{skein31c}}.
\]
That is, if $\gamma_1,\gamma_2,\gamma_3$ are unframed $pK$ cords that are related as shown in the left side of Figure~\ref{fig:quandle}, then we impose the relation:
\[
\gamma_1 - m \gamma_2 = (1-m) \gamma_3.
\]

Thus we can describe $\Cord'_{pK} \otimes_\epsilon \Z[m^{\pm 1}]$ in terms of a knot diagram for $K$ as follows. Use the diagram to place $K$ in a neighborhood of the $xy$ plane in $\R^3$, and place $p$ high above the $xy$ plane along the $z$ axis. If the diagram has $n$ crossings, then it divides $K$ into $n$ strands from undercrossing to undercrossing. Then
$\Cord'_{pK} \otimes_\epsilon \Z[m^{\pm 1}]$ is generated by $n$ unframed $pK$ cords, namely straight line segments from $p$ to any point on each of these strands, and each crossing gives a relation $\gamma_1 - m \gamma_2 = (1-m) \gamma_3$ if $\gamma_1,\gamma_2,\gamma_3$ are as shown in the right side of Figure~\ref{fig:quandle}. But this is the well-known presentation of $H_1(\tilde{X}_K) \oplus \Z[m^{\pm 1}]$ from knot colorings. In particular, what we have just described is the Alexander quandle of $K$, see \cite{Joyce}.
\end{proof}

\begin{remark}
The description we have given in this section for the unframed cord modules is highly reminiscent of the construction of the knot quandle from \cite{Joyce}, which is known to be a complete invariant. However, we do not know how to extract the entire knot quandle, rather than just the Alexander quandle (which is a quotient), from the unframed cord module.
\end{remark}

\subsection{String homology in terms of the knot group}
\label{ssec:brackets}

Having expressed the KCH-triple $(R_{KK},R_{Kp},R_{pK})$ in terms of cords in Section~\ref{ssec:cordalg}, our next step en route to proving Theorem~\ref{thm:main} is to rewrite the KCH-triple further, in terms of the
knot group and the peripheral subgroup of the knot $K$. For $R_{KK}$,
which is the degree $0$ knot contact homology of $K$, this was done in
\cite[\S 2.3--2.4]{CELN}, and we follow the treatment there.

Write
$\pi = \pi_1(\R^3\setminus K)$ for the knot group and $\hat\pi =
\pi_1(\Lambda_K)$ for the peripheral subgroup. A framing and
orientation on $K$ gives meridian and longitude classes $m,l \in
\hat\pi$, which we can then view as classes in $\pi$ as well.
In what follows, we place square
brackets around elements of $\pi$ and curly brackets around elements
of $\hat\pi$.

Define $S$ to be the $\Z$-module freely generated by words
that are formal products of nontrivial words whose letters are
alternately in $\pi$ and $\hat\pi$,
divided by the following string relations, where we use $x$ and $\alpha$ to denote elements of $\pi$ and $\hat\pi$ respectively:
\begin{enumerate}
\item \label{it:str1}
$\cdots_1 [x\alpha_1]\{\alpha_2\} \cdots_2 = \cdots_1 [x]\{\alpha_1\alpha_2\}
\cdots_2 $
\item \label{it:str2}
$\cdots_1 \{\alpha_1\}[\alpha_2 x] \cdots_2  = \cdots_1 \{\alpha_1\alpha_2\}[x] \cdots_2$
\item \label{it:str3}
$(\cdots_1 [x_1x_2] \cdots_2) - (\cdots_1 [x_1m x_2] \cdots_2) = \cdots_1
[x_1]\{1\}[x_2] \cdots_2$
\item \label{it:str4}
$(\cdots_1 \{\alpha_1\alpha_2\} \cdots_2) - (\cdots_1
\{\alpha_1m\alpha_2\} \cdots_2 ) = \cdots_1 \{\alpha_1\}[1]\{\alpha_2\} \cdots_2$.
\end{enumerate}

Note that there is no restriction on generators of $S$ as to whether
the first or last letters are in $\pi$ or $\hat\pi$.
We can define a product on $S$ as follows:
multiplication of two words $w_1,w_2$ generating $S$ is zero
unless the last letter of $w_1$ and the first letter of $w_2$ are both
in $\pi$ or both in $\hat\pi$, in which case it is concatenation
combined with the product in $\pi$ or $\hat\pi$; that is,
\begin{align*}
(\cdots_1 \{\alpha_1\}) \cdot (\{\alpha_2\} \cdots_2)
&= \cdots_1 \{\alpha_1\alpha_2\} \cdots_2 \\
(\cdots_1 [x_1]) \cdot ([x_2] \cdots_2)
&= \cdots_1 [x_1x_2] \cdots_2.
\end{align*}

We now have the following result identifying $R_{KK}$, $R_{Kp}$, $R_{pK}$
from Section~\ref{ssec:broken-strings} with summands of $S$.

\begin{proposition}
$R_{KK}$, $R_{Kp}$, and $R_{pK}$
\label{prop:KCHtriple-string}
are isomorphic to the
$\Z$-submodules of $S$ generated by the following sets:
\begin{itemize}
\item
for $R_{KK}$, words
beginning and ending in $\hat\pi$;
\item
for $R_{Kp}$, words beginning in $\hat\pi$ and ending in
$\pi$;
\item
for $R_{pK}$, words
beginning in $\pi$ and ending in $\hat\pi$.
\end{itemize}
Multiplication in $S$
induces maps $R_{KK} \otimes R_{KK} \to R_{KK}$, $R_{KK} \otimes
R_{Kp} \to R_{Kp}$, $R_{pK} \otimes R_{KK} \to R_{pK}$ that agree
with, respectively,
the ring structure on $R_{KK}$ and the $R_{KK}$-module structure on
$R_{Kp}$ and $R_{pK}$.
\end{proposition}

\begin{proof}
Same as the proof of \cite[Proposition 2.14]{CELN}. Briefly, by Proposition~\ref{prop:cord-triple}, the KCH-triple is isomorphic to $(\Cord_{KK},\Cord_{Kp},\Cord_{pK})$. Given a cord, we can produce a closed loop in $\R^3\setminus K$ based at $p$, and hence an element of $\pi$, as in the proof of Proposition~\ref{prop:framed-unframed}. Thus products of cords, with elements of $\Z[l^{\pm 1},m^{\pm 1}]$ in between, correspond to alternating products of elements of $\pi$ and $\hat\pi$. The string relations on $S$ come from the skein relations on cords. Note that the distinct behaviors of $R_{KK}$, $R_{Kp}$, and $R_{pK}$ in the statement of Proposition~\ref{prop:KCHtriple-string} come from the construction of the cord algebra and modules: an element of $\Cord_{KK}$ begins and ends with an element of $\hat\pi$ (possibly $1$), while an element of $\Cord_{Kp}$ begins with an element of $\hat\pi$ and ends with a cord (which maps to $\pi$), and similarly for $\Cord_{pK}$.
\end{proof}

To clarify:
\begin{itemize}
\item
$R_{KK}$ is generated by $\{\alpha_1\}$, $\{\alpha_1\}[x_1]\{\alpha_2\}$,
$\{\alpha_1\}[x_1]\{\alpha_2\}[x_2]\{\alpha_3\}$, $\ldots$
\item
$R_{Kp}$ is generated by $\{\alpha_1\}[x_1]$,
$\{\alpha_1\}[x_1]\{\alpha_2\}[x_2]$, $\ldots$
\item
$R_{pK}$ is generated by $[x_1]\{\alpha_1\}$,
$[x_1]\{\alpha_2\}[x_2]\{\alpha_2\}$, $\ldots$
\end{itemize}
where $\alpha_i \in \hat\pi$ and $x_i \in \pi$.
To this, we can then add:
\begin{itemize}
\item
$R_{pp} = R_{pK} \otimes_{R_{KK}} R_{Kp}$ is generated by
$[x_1]\{\alpha_1\}[x_2]$,
$[x_1]\{\alpha_1\}[x_2]\{\alpha_2\}[x_3]$,
$\ldots$.
\end{itemize}

Finally, the product $\mu :\thinspace R_{Kp} \otimes R_{pK} \to
R_{KK}$ has a simple interpretation in terms of $S$, since by
Proposition~\ref{prop:cord-triple} it is the concatenation product:
\[
\mu(\{\alpha_0\} \cdots \{\alpha_1\}[x_1],
[x_2] \{\alpha_2\} \cdots \{\alpha_3\}) =
\{\alpha_0\} \cdots \{\alpha_1\}[x_1x_2] \{\alpha_2\} \cdots \{\alpha_3\}.
\]
The product on $R_{pp} = R_{pK} \otimes_{R_{KK}} R_{Kp}$ induced by
$\mu$ is then also given by concatenation.

\subsection{The KCH-triple within $\Z[\pi_1(\R^3\setminus K)]$}
\label{ssec:kchtriple}

Although the notation from Section~\ref{ssec:brackets} using square
and curly brackets is natural from the viewpoint of broken strings, it
will be convenient for our purposes to reinterpret the KCH-triple
$(R_{KK},R_{Kp},R_{pK})$ directly in terms of the group ring of the knot group,
which we henceforth denote by
\[
R := \Z[\pi_1(\R^3\setminus K)].
\]
This is the content of Proposition~\ref{prop:triple-isom} below.

To prepare for this result, extend the notation $\cdots [x_1] \{\alpha_1\} [x_2]
\cdots \in S$, where up to now we have $x_i \in \pi$ and $\alpha_i \in
\hat\pi$, by linearity to allow for arbitrary $x_i \in R = \Z\pi$.
Given any element of $S$ of the form $\cdots [x_1] \{\alpha_1\} [x_2]
\cdots$ with $\alpha_i \in
\hat\pi$ and $x_i \in \Z\pi$, the string relations on $S$ allow us to
get rid of any internal part in curly braces, where ``internal'' means
not at the far left or far right. More precisely, by
\eqref{it:str1} and \eqref{it:str3} from the defining relations for
$S$ in Section~\ref{ssec:brackets}, we can write:
\[
\cdots [x_1] \{\alpha_1\} [x_2]
\cdots =
\cdots [x_1\alpha_1x_2] \cdots - \cdots [x_1\alpha_1mx_2] \cdots.
\]
This allows us to inductively reduce the number of internal curly
braces until none are left.

Thus for instance we can write any element of $R_{Kp}$ as a linear
combination of elements of the form $\{\alpha_1\}[x_1]$, where $x_1
\in \Z\pi$, and this in turn is equal to $\{1\}[\alpha_1x_1]$ by
string relation \eqref{it:str2}. Similar results hold for $R_{pK}$ and
$R_{KK}$, as well as for $R_{pp}$, and we conclude the following:

\begin{proposition}
As $\Z$-submodules of $S$, we have:
\label{prop:generated}
\begin{itemize}
\item
$R_{KK}$ is generated by the elements of the form $\{\alpha\}$ and
$\{1\}[x]\{1\}$ for
$\alpha\in\hat\pi$ and $x\in\pi$;
\item
$R_{Kp}$ is generated by $\{1\}[x]$ for $x\in\pi$;
\item
$R_{pK}$ is generated by $[x]\{1\}$ for $x\in\pi$;
\item
$R_{pp}$ is generated by $[x_1]\{1\}[x_2]$ for $x_1,x_2\in\pi$.
\end{itemize}
\end{proposition}

Write $\hat R = \Z[\hat\pi]=\Z[l^{\pm 1},m^{\pm 1}]$, and view
$\hat R$ as a subring of $R$. In \cite[Proposition~2.20]{CELN}, it is
shown that the map $\{\alpha\} \mapsto \alpha$, $\{1\}[x]\{1\} \mapsto
x(1-m)$ induces an isomorphism from $R_{KK}$ to $\hat R+R(1-m)$, where
the latter is viewed as a subring of $R$ (and $R(1-m)$ is the left
ideal generated by $1-m$).
\begin{remark}\label{r:choiceofsmoothing}
To be precise,
the map in \cite[Proposition~2.20]{CELN} is from $R_{KK}$ to $\hat
R+(1-m)R$ rather than $\hat R+R(1-m)$, and sends $\{1\}[x]\{1\}$ to
$(1-m)x$ rather than $x(1-m)$. This however is just a choice of where
to place the $(1-m)$ factors. Our convention can be derived from the
convention in \cite{CELN} by the symmetry that reverses the order of
words in $S$.
\end{remark}
We can now generalize this
isomorphism to the entire KCH-triple.

\begin{proposition}
We have $\Z$-module isomorphisms between the KCH-triple and the following
$\Z$-submodules of $R = \Z[\pi]$:
\label{prop:triple-isom}
\begin{align*}
R_{KK} &\stackrel{\cong}{\to} \hat R+R(1-m) &
\{\alpha\} &\mapsto \alpha\, ,
\; \{1\}[x]\{1\} \mapsto x(1-m); \\
R_{Kp} &\stackrel{\cong}{\to} R &
\{1\}[x] &\mapsto x; \\
R_{pK} &\stackrel{\cong}{\to} R(1-m) &
[x]\{1\} &\mapsto x(1-m),
\end{align*}
where the second and third isomorphisms hold for any knot $K$ and the
first isomorphism holds as long as $K$ is not the unknot.
We use $\phi$ to denote all of these isomorphisms; then it is
furthermore the case that $\phi$ sends
all multiplications $R_{KK} \otimes R_{KK} \to R_{KK}$, $R_{KK} \otimes
R_{Kp} \to R_{Kp}$, $R_{pK} \otimes R_{KK} \to R_{pK}$, as well as the
product $R_{Kp} \otimes R_{pK} \to R_{KK}$, to
multiplication in $R$.
\end{proposition}

\begin{proof}
This follows the proof of \cite[Proposition~2.20]{CELN}.
To see that $\phi$ is well-defined, extend the definition of $\phi$ to
all generators of $S$ (ignoring Proposition~\ref{prop:generated} for
the moment) by:
\begin{align*}
\{\alpha_1\}[x_1]\{\alpha_2\} \cdots \{\alpha_{k-1}\}[x_k]\{\alpha_k\}
&\mapsto \alpha_1(1-m)x_1\alpha_2(1-m)\cdots\alpha_{k-1}(1-m)x_k\alpha_k \\
\{\alpha_1\}[x_1]\{\alpha_2\} \cdots \{\alpha_{k-1}\}[x_k]
&\mapsto \alpha_1x_1\alpha_2(1-m)\cdots\alpha_{k-1}(1-m)x_k\\
[x_1]\{\alpha_1\}[x_2]\cdots [x_k]\{\alpha_k\}
&\mapsto x_1\alpha_1(1-m)x_2\cdots x_k\alpha_k(1-m) \\
[x_1]\{\alpha_1\}[x_2]\cdots [x_k]
&\mapsto x_1\alpha_1(1-m)x_2\cdots x_k;
\end{align*}
that is, replace each term $\{\alpha\}$ in curly braces by
$\alpha(1-m)$ unless $\{\alpha\}$ is at the end of a word, in which
case replace it by $\alpha$. It is easily checked that this map
preserves the string relations on $S$, and so it gives a well-defined
map $\phi :\thinspace S \to R$. Restricted to generators of the form
$\{\alpha\}$, $\{1\}[x]\{1\}$, $\{1\}[x]$, $[x]\{1\}$, $\phi$ is as
given in the statement of the proposition; note now by
Proposition~\ref{prop:generated} that these suffice to determine $\phi$.

We next check bijectivity.
The maps $\phi$ are clearly surjective. It is proved in
\cite[Proposition~2.20]{CELN} that $\phi$ on $R_{KK}$ is
injective as long as $K$ is knotted. The
fact that $\phi$ is injective on $R_{Kp}$ is trivial:
by Proposition~\ref{prop:generated}, any element of $R_{Kp}$ can be
written as $\{1\}[x]$ for some $x\in R$, and then $\phi(\{1\}[x])=0$
implies $x=0$. To prove that $\phi$ is injective on
$R_{pK}$, note that if $\phi([x]\{1\})=0$, then $x(1-m)=0$ in
$R$; then since knot groups are left orderable, $R=\Z\pi$ has no zero
divisors \cite{Higman}, and so $x=0$.

Finally, the fact that $\phi$ respects multiplication and $\mu$
follows readily from the definition of $\phi$: for example,
\begin{align*}
\phi(\mu(\{\alpha_0\} \cdots \{\alpha_1\}[x_1],
[x_2] \{\alpha_2\} \cdots \{\alpha_3\}))
&=
\phi(\{\alpha_0\} \cdots \{\alpha_1\}[x_1x_2] \{\alpha_2\} \cdots
  \{\alpha_3\}) \\
&= \alpha_0(1-m)\cdots \alpha_1(1-m)x_1x_2\alpha_2(1-m)\cdots \alpha_3 \\
&= \phi(\{\alpha_0\} \cdots \{\alpha_1\}[x_1])
\cdot \phi([x_2] \{\alpha_2\} \cdots \{\alpha_3\}).
\end{align*}
\end{proof}

\begin{remark}\label{r:geometriciso}
In Section \ref{ssec:direct} we give an interpretation of the maps in Proposition \ref{prop:triple-isom} in terms of moduli spaces of holomorphic disks inducing maps from partially wrapped Floer cohomology into chains on spaces of paths and loops in $\R^{3}\setminus K$. 	
\end{remark}

Recall that the product $\mu$ gives a ring structure on $R_{pp} =
R_{pK} \otimes R_{Kp}$. Similarly to
Proposition~\ref{prop:triple-isom}, we then have the following.

\begin{proposition}
The isomorphisms $\phi :\thinspace R_{pK} \stackrel{\cong}{\to} R(1-m)$ and
$R_{Kp} \stackrel{\cong}{\to} R$ induce a ring isomorphism
\label{prop:twosided}
\[
\phi :\thinspace R_{pp} \stackrel{\cong}{\to} R(1-m)R
\]
where $R(1-m)R$ denotes the two-sided ideal of $R$ generated by $1-m$.
\end{proposition}

\begin{proof}
Same as the proof of Proposition~\ref{prop:triple-isom}, but now use
the fact that $\phi$ is defined on $R_{pp}$ by
\[
[x_1]\{\alpha_1\}[x_2]\cdots \{\alpha_{k-1}\}[x_k]
\mapsto x_1\alpha_1(1-m)x_2 \cdots \alpha_{k-1}(1-m)x_k
\]
and that any element of $R_{pp}$ can be written as $[x]$ for some
$x\in R$.
\end{proof}

Although $R(1-m)R$ is not all of $R$, we observe the following.

\begin{proposition}
The submodules $\Z = \Z \cdot 1$ and $R(1-m)R$ of $R$ are
complementary:
\label{prop:complementary}
$R \cong
\Z \oplus R(1-m)R$ as $\Z$-modules.
\end{proposition}

\begin{proof}
The homomorphism from $\pi$ to the trivial group induces a map $R =
\Z[\pi] \to \Z$, which restricts to the identity on $\Z$ and to the
zero map on $R(1-m)R$; thus $\Z \cap R(1-m)R = 0$.
On the other hand,
$\pi = \pi_1(\R^3\setminus K)$ is generated by a finite collection of
meridians, each of which is an element of $\pi$ of the form $\gamma m
\gamma^{-1}$ for some $\gamma \in \pi$, and $\gamma m \gamma^{-1} =
1-\gamma(1-m)\gamma^{-1} \in \Z + R(1-m)R$; thus $\Z + R(1-m)R = R$.
\end{proof}

We can restate the combination of Propositions~\ref{prop:twosided}
and~\ref{prop:complementary} as follows. Consider the direct sum
$\Z \oplus R_{pp}$, and give this a ring structure by defining the
generator $1$ of $\Z$ to be the multiplicative identity and setting
multiplication on the factor $R_{pp}$ to be as usual. That is:
\[
(n_1,r_1) \cdot (n_2,r_2) = (n_1n_2,n_1r_2+n_2r_1+r_1r_2).
\]
Then:

\begin{proposition}
We have a ring isomorphism
\label{prop:ring-isom}
$\phi :\thinspace \Z \oplus R_{pp}
\stackrel{\cong}{\to} R$ defined by
\[
\phi(n,[x_1]\{\alpha_1\}[x_2]) = n+x_1\alpha_1(1-m)x_2.
\]
\end{proposition}

\begin{remark}
The ring $\Z \oplus R_{pp}$ may seem like an odd candidate to be isomorphic to $R = \Z[\pi_1(\R^3\setminus K)]$. In fact, it is the correct object to consider from at least two perspectives. One is through the cord algebra and modules: we have
$R_{pp} \cong \Cord_{pK} \otimes \Cord_{Kp}$, and we can use the skein relation for cords (Definition~\ref{def:cordalg} \eqref{it:cord3}) to rewrite the product of a $pK$ cord with a $Kp$ cord as the difference of two $pp$ cords. This gives a map
\[
\Cord_{pK} \otimes \Cord_{Kp} \to \Cord_{pp}.
\]
This map is not surjective, since it maps to the ideal generated by differences of cords, but it becomes an isomorphism if we add the $\Z$-module generated by the trivial $pp$ cord to the left hand side to obtain $\Z \oplus R_{pp}$. On the other hand, any $pp$ cord is a loop in $\R^3 \setminus K$, and this induces an isomorphism between $\Cord_{pp}$ and $\Z[\pi_1(\R^3\setminus K)]$.

The other perspective is through partially wrapped Floer cohomology; see Section~\ref{ssec:wrapped-WK}.
\end{remark}


\section{Proof of Theorem~\ref{thm:main}}
\label{sec:mainpf}

Here we present the proof of our main result. Briefly, by Proposition~\ref{prop:ring-isom},  we can recover $R = \Z[\pi_1(\R^3\setminus K)]$ from
the ring structure on $R_{pp}$, which itself is determined by the
Legendrian contact homology of $\Lambda_K \cup \Lambda_p$ along with
the product $\mu$. This allows us to recover the knot $K$ itself. The details of the proof are broken into two subsections corresponding to statements \eqref{it:main1} and \eqref{it:main2} from Theorem~\ref{thm:main}, and at the end of this section we make some remarks about the difference between the two.

\subsection{Proof of Theorem~\ref{thm:main} (\ref{it:main2})}
\label{ssec:mainpf1}

Suppose that $K_0,K_1$ are knots such that there are isomorphisms
between the KCH-triples of $K_0$ and $K_1$; these
isomorphisms are compatible with the
products $\mu :\thinspace R_{K_ip} \otimes R_{pK_i} \to R_{K_iK_i}$;
and the isomorphism $R_{K_0K_0} \stackrel{\cong}{\to} R_{K_1K_1}$
sends the meridian and longitude $m_0,l_0$ of $K_0$ to the meridian
and longitude $m_1,l_1$ of $K_1$, respectively.
Under these assumptions, we want to conclude that $K_0,K_1$ are
isotopic in $\R^3$ as oriented knots.
Since $R_{KK}$ detects the unknot (see
e.g. \cite[Corollary~1.5]{CELN}), we will assume that $K_0,K_1$ are
both knotted.

For $i=0,1$, write $R_i = \Z[\pi_1(\R^3\setminus K_i)]$ and $\hat R_i
= \Z[m_i^{\pm 1},l_i^{\pm 1}] \subset R_i$.
By Proposition~\ref{prop:triple-isom}, we can write the isomorphisms
between the KCH-triples $(R_{K_0K_0},R_{K_0p},R_{pK_0})$ and
$(R_{K_1K_1},R_{K_1p},R_{pK_1})$ as a triple of maps
\[
(\psi_{KK},\psi_{Kp},\psi_{pK}) :\thinspace
(\hat R_0+R_0(1-m_0),R_0,R_0(1-m_0)) \stackrel{\cong}{\to}
(\hat R_1+R_1(1-m_1),R_1,R_1(1-m_1)).
\]
These maps are compatible with multiplication and the product $\mu$ in
the way described in Proposition~\ref{prop:triple-isom}.

Write $R_{pp}^i = R_{pK_i} \otimes_{R_{K_iK_i}} R_{K_ip}$; then the
map $\psi_{pK} \otimes \psi_{Kp}$ is an isomorphism from $R_{pp}^0$ to
$R_{pp}^1$, and it preserves the ring structure on $R_{pp}^i$
determined by $\mu$. By Proposition~\ref{prop:ring-isom}, we can then
view
\[
\psi_{pp} := \text{id} \oplus (\psi_{pK}
\otimes \psi_{Kp}) :\thinspace \Z \oplus R_{pp}^0 \stackrel{\cong}{\to} \Z \oplus
R_{pp}^1
\]
as a ring isomorphism $\psi_{pp} :\thinspace R_0 \stackrel{\cong}{\to} R_1$.

Now $R_i = \Z[\pi_1(\R^3\setminus K_i)]$ and knot groups are
left-orderable, and so any isomorphism between $R_0$ and $R_1$ must
come from a group isomorphism between the knot groups \cite{Higman}.
More precisely, the set of units in the group ring $\Z[G]$ of a
left-orderable group $G$ are exactly the
elements $\pm g$ for $g\in G$, and so there exists an isomorphism
\[
\psi :\thinspace \pi_1(\R^3\setminus K_0) \stackrel{\cong}{\to}
\pi_1(\R^3\setminus K_1)
\]
such that $\psi_{pp}(\gamma) = \pm \psi(\gamma)$ for all $\gamma\in
\pi_1(\R^3\setminus K_0)$.

Now that we know that the knot groups of $K_0$ and $K_1$ are
isomorphic, it remains to show that the isomorphism preserves the
peripheral structure. We will show that there is some $\gamma \in
\pi_1(\R^3\setminus K_1)$ such that $\psi(m_0) = \gamma^{-1} m_1
\gamma$ and $\psi(l_0) = \gamma^{-1} l_1 \gamma$, whence the
composition of $\psi$ and conjugation by $\gamma$ gives an isomorphism
$\pi_1(\R^3\setminus K_0)
\stackrel{\cong}{\to} \pi_1(\R^3\setminus K_1)$ sending $m_0,l_0$ to
$m_1,l_1$ as desired.

For this, we use the assumption that
$\psi_{KK}(m_0) = m_1$ and $\psi_{KK}(l_0) = l_1$. The elements
$1 \in R_0 \cong R_{K_0p}$ and $1-m_0 \in R_0(1-m_0) \cong R_{pK_0}$
have images under $\psi_{Kp}$ and $\psi_{pK}$
\begin{align*}
\psi_{Kp}(1) &= x & \psi_{pK}(1-m_0) &= x'(1-m_1)
\end{align*}
for some $x,x' \in R_1$. In $R_1$, we have
\begin{align*}
1-m_1 &= \psi_{KK}(1-m_0) = \psi_{KK}\mu(1,1-m_0)\\
&=\mu(\psi_{Kp}(1),\psi_{pK}(1-m_0)) =\mu(x,x'(1-m_1))
= xx'(1-m_1).
\end{align*}
Now the group ring of a left-orderable group has no zero divisors, and
so it follows that $xx'=1\in R_1$ and hence there exists some $\gamma
\in \pi_1(\R^3\setminus K_1)$ such that $x=\pm \gamma$, $x'=\pm
\gamma^{-1}$.

Next, for $\alpha\in\Z$, view $l_0^\alpha
(1-m_0) = (1-m_0) \cdot l_0^\alpha \cdot 1$ as an element of
$R_{pK_0} \otimes_{R_{K_0K_0}} R_{K_0K_0} \otimes_{R_{K_0K_0}}
R_{K_0p} = R_{pp}^0$. Then we have
\[
\psi_{pp}(l_0^\alpha
(1-m_0)) = \psi_{pK}(1-m_0) \cdot \psi_{KK}(l_0^\alpha) \cdot
\psi_{Kp}(1) = \gamma^{-1}(1-m_1) \cdot l_1^\alpha \cdot
\gamma.
\]
Now since $\psi_{pp} = \pm \psi$ on elements of $\pi_1(\R^3\setminus
K_0)$, both of $\psi_{pp}(l_0^\alpha(1-m_0))$ and $\gamma^{-1}(1-m_1) \cdot l_1^\alpha \cdot
\gamma$ are
binomials (i.e., sums of the form $n_1\gamma_1+n_2\gamma_2$ with $n_i \in \Z$ and $\gamma_i \in \pi_1(\R^3\setminus K_1)$), and equating terms gives
\[
\{\psi(l_0^\alpha),\psi(m_0 l_0^\alpha)\}
= \{\gamma^{-1} l_1^\alpha\gamma,\gamma^{-1} m_1 l_1^\alpha
\gamma\}.
\]
Plugging in $\alpha=0$ and $\alpha=1$ in succession gives $\psi(m_0) =
\gamma^{-1} m_1 \gamma$ and $\psi(l_0) = \gamma^{-1} l_1 \gamma$, as desired.

\subsection{Proof of Theorem~\ref{thm:main} (\ref{it:main1})}
\label{ssec:mainpf2}

We now consider the case where there is an isomorphism
between the KCH-triples of $K_0$ and $K_1$ preserving the product
$\mu$, but without the assumption that longitude and
meridian classes are mapped to themselves. We will prove that there is
an isomorphism $\pi_1(\R^3\setminus K_0) \cong \pi_1(\R^3\setminus
K_1)$ sending $m_0$ to $m_1^{\pm 1}$ and $l_0$ to $l_1^{\pm 1}$. It
then follows from Waldhausen \cite{Wal} that $K_0$ is smoothly
isotopic to either
$K_1$ or the mirror of $K_1$, as unoriented knots. (Note that our
argument identifies not only the peripheral subgroups but also the
meridians in each, and so we do not need to appeal to Gordon and
Luecke \cite{GL}.)
As in the previous proof, we can assume that $K_0$ and $K_1$ are both knotted.

The setup is
as in Section~\ref{ssec:mainpf1}, except that
$\psi_{KK}(m_0) = m_1^{n_1}l_1^{n_2}$ and $\psi_{KK}(l_0) =
m_1^{n_3}l_1^{n_4}$ for some $\left( \begin{smallmatrix} n_1 & n_2 \\
    n_3 & n_4 \end{smallmatrix} \right) \in GL_2(\Z)$ not necessarily
the identity matrix.
As in the previous proof, $\psi_{pp}$
is induced by a group
isomorphism $\psi :\thinspace \pi_1(\R^3\setminus K_0) \to
\pi_1(\R^3\setminus K_1)$.
Now however when we write $\psi_{Kp}(1)=x$, $\psi_{pK}(1-m_0) =
x'(1-m_1)$ for $x,x'\in R_1$, we have
\begin{equation}
xx'(1-m_1) = 1-m_1^{n_1}l_1^{n_2}.
\label{eq:left}
\end{equation}
We can then appeal to the following algebraic result.

\begin{lemma}
Suppose $G$ is left orderable and $m,g\in G$, $z \in \Z[G]$ satisfy
$z(1-m) = 1-g$ in $\Z[G]$.
\label{lma:left}
Then $g=m^n$ for some $n\in\Z$.
\end{lemma}

\begin{proof}
Let $<$ be the left-invariant ordering on $G$. Without loss of
generality, we may assume $m>1$; the lemma is trivial if $m=1$, and we
can replace $m$ by $m^{-1}$ if $m<1$.
Write $z = \sum_{i=1}^k a_i g_i$ for $a_i \in \Z$ and $g_i \in G$; we
may assume that $a_i \neq 0$ for all $i$ and $g_1 < g_2 < \cdots <
g_k$. Then $g_i < g_i m$ for all $i$, and so among the subset $S =
\{g_1,\ldots,g_k,g_1m,\ldots,g_km\}$ of $G$ (which may
contain repeated elements), $g_1$ is strictly
lowest. Furthermore, let $g_jm$ be the largest among the $k$ distinct
elements $g_1m,\ldots,g_km$; then $g_jm$ is strictly largest among the
elements of $S$. It follows that the expansion of $z(1-m) =
\sum_{i=1}^k (a_i g_i - a_i g_i m)$ involves $g_1$ and $g_j m$ at
least. Since $1-g$ is a binomial, it follows that $\{g_1,g_jm\} =
\{1,g\}$, and furthermore that all terms in the expansion not
involving  $g_1$ or $g_j m$ must cancel. Thus $g_1m$ must be canceled
by $g_{i_1}$ for some $i_1$, whence $g_{i_1}m$ must be canceled by
$g_{i_2}$ for some $i_2$, and so forth. We conclude that there is a
sequence $i_0=1,i_1,\ldots,i_\ell=j$ such that $g_{i_r}=g_{i_{r-1}}m$
for all $r=1,\ldots,\ell$, and so $g_j = g_1 m^{r-1}$. The lemma follows.
\end{proof}

We now continue with the proof of Theorem~\ref{thm:main} (\ref{it:main1}).
By Equation~\eqref{eq:left} and Lemma~\ref{lma:left},
$m_1^{n_1}l_1^{n_2} = m_1^n$ for some $n$. By the Loop Theorem,
since $K_2$ is knotted, $l_1$ is not a power of $m_1$, and so
$n_2=0$. Since $\left( \begin{smallmatrix} n_1 & n_2 \\
    n_3 & n_4 \end{smallmatrix} \right)$ is invertible, it follows
that $n_1 = \pm 1$ and $n_4 = \pm 1$.

We treat the cases $n_1=1$ and $n_1=-1$ separately. If $n_1=1$, then
$\psi_{KK}(m_0) = m_1$ and $\psi_{KK}(l_0) = m_1^{n_3}l_1^{\pm 1}$.
As in Section~\ref{ssec:mainpf1}, we have $xx'(1-m_1) = 1-m_1$ and
so $x = \pm \gamma$, $x' = \pm \gamma^{-1}$ for some $\gamma \in
\pi_1(\R^3\setminus K_1)$. Now for any $\alpha\in\Z$, we compute
\[
\psi_{pp}(l_0^\alpha(1-m_0)) =
\psi_{pK}(1-m_0) \cdot \psi_{KK}(l_0^\alpha)\cdot
\psi_{Kp}(1)
= \gamma^{-1}m_1^{\alpha n_3} l_1^{\pm \alpha} (1-m_1) \gamma.
\]
Identifying terms as in Section~\ref{ssec:mainpf1} gives
\[
\{\psi(l_0)^\alpha,\psi(m_0)\psi(l_0)^\alpha\} =
\{\gamma^{-1} m_1^{\alpha n_3}l_1^{\pm \alpha} \gamma,
\gamma^{-1} m_1^{\alpha n_3+1}l_1^{\pm \alpha} \gamma\}.
\]
Plugging in $\alpha=0$ and $\alpha=1$ in succession gives
$\psi(m_0) = \gamma^{-1} m_1 \gamma$ and $\psi(l_0) = \gamma^{-1}
m_1^{n_3}l_1^{\pm 1}\gamma$. Conjugating $\psi$ by $\gamma$ gives a
group isomorphism $\pi_1(\R^3\setminus K_1) \stackrel{\cong}{\to}
\pi_1(\R^3\setminus K_2)$ sending $m_0$ to $m_1$ and $l_0$ to
$m_1^{n_3}l_1^{\pm 1}$. Since the longitude is the identity in (and
the meridian generates) the
abelianization of $\pi_1$, we must have $n_3=0$: thus $m_0$ is sent to
$m_1$ and $l_0$ to $l_1^{\pm 1}$.

If instead $n_1=-1$, we can run the same argument to conclude $x =
\pm \gamma$, $x' = \mp \gamma^{-1}m_1^{-1}$;
$\psi_{pp}(l_0^\alpha(1-m_0)) =
-\gamma^{-1}m_1^{-1}(1-m_1)m_1^{\alpha n_3}l_1^{\pm\alpha}\gamma$;
$\psi(m_0) = \gamma^{-1} m_1^{-1}\gamma$ and
$\psi(l_0) = \gamma^{-1} m_1^{n_3}l_1^{\pm 1} \gamma$;
and finally there is an isomorphism
$\pi_1(\R^3\setminus K_0) \stackrel{\cong}{\to}
\pi_1(\R^3\setminus K_1)$ sending $m_0$ to $m_1^{-1}$ and $l_0$ to
$l_1^{\pm 1}$.

This completes the proof of Theorem~\ref{thm:main} (\ref{it:main1}).

\subsection{A note on different types of Legendrian isotopy}

The difference between statements \eqref{it:main1} and \eqref{it:main2} in Theorem~\ref{thm:main} is in the strength of the assumption about the Legendrian isotopy relating two conormal tori $\Lambda_{K_0}$ and $\Lambda_{K_1}$. One might ask if the weaker assumption---an unparametrized Legendrian isotopy between $\Lambda_{K_0}$ and $\Lambda_{K_1}$---might still imply the stronger result---a smooth isotopy between $K_0$ and $K_1$ as oriented knots. This appears to be possible, but our invariants do not show this.

The issue is a symmetry of $ST^*\R^3$: the diffeomorphism of $\R^3$ given by $(x,y,z) \mapsto (x,y,-z)$ induces a coorientation-preserving contactomorphism of $ST^*\R^3$ that preserves cotangent fibers and sends the conormal torus of $K$ to the
conormal torus of the mirror $m(K)$ of $K$. It follows that there is an isomorphism
\[
\LCH_*(\Lambda_p \cup \Lambda_K) \cong \LCH_*(\Lambda_p \cup
\Lambda_{m(K)}),
\]
and indeed between the DGAs for $\Lambda_p \cup \Lambda_K$ and $\Lambda_p \cup \Lambda_{m(K)}$. Further, this isomorphism preserves the product $\mu$. (On the homology of the conormal torus, the mirroring map preserves $l$ but sends $m$ to $m^{-1}$, and so this symmetry does not contradict Theorem~\ref{thm:main} \eqref{it:main2}.)

There is a result that is somewhere in between statements \eqref{it:cor1} and \eqref{it:cor2} from
Theorem~\ref{thm:complete}: if $\Lambda_{K_0},\Lambda_{K_1}$ are
Legendrian isotopic in an orientation-preserving manner, then
$K_0,K_1$ are smoothly isotopic as unoriented knots. Indeed, from the proof of Theorem~\ref{thm:main} \eqref{it:main1} in Section~\ref{ssec:mainpf2}, the Legendrian isotopy must send $(m_0,l_0)$ to $(m_1^{\pm 1},l_1^{\pm 1})$, and the orientation-preserving condition implies that the two signs agree. It follows that there is an isomorphism between $\pi_1(\R^3\setminus K_0)$ and $\pi_1(\R^3\setminus K_1)$ that preserves the peripheral subgroup, possibly after changing the orientation of $K_1$.


\section{Lagrangian Skeleta, Wrapped Floer Cohomology, and Microlocal Sheaves}
\label{sec:wrapped}

In this section we describe a geometric setup where the enhanced knot contact homology appears as the partially wrapped Floer cohomology of certain Lagrangian disks. This Floer cohomology is closely related to the microlocal sheaf approach to these questions, see e.g. \cite{Shende}. We calculate the partially wrapped Floer cohomology via a version of the Legendrian surgery isomorphism \cite{BEE}.

Let $K\subset\R^{3}$ be a knot. We will associate a Weinstein manifold $W_K$ to $K$ with Lagrangian skeleton $\R^3\cup L_K$. Here a Lagrangian skeleton of a Weinstein domain $X$ is a Lagrangian subvariety $L\subset X$ (possibly singular but looking like the symplectization of a Legendrian $\Lambda$ at infinity) such that $X$ is a regular neighborhood of $L$.
The space $W_K$ will be topologically identical to $T^*\R^3$, but the Liouville vector field on $W_K$ is somewhat different than the one on $T^*\R^3$. We use this Liouville field to define partially wrapped Floer cohomology adapted to the Legendrian $\Lambda$.

Before we describe $W_K$, we first discuss our version of partially wrapped Floer cohomology in more generality. We then explain the analogue of the construction of $W_K$ in one dimension down ($\R^2$ rather than $\R^3$) before proceeding to $W_K$ itself and calculating partially wrapped Floer cohomology in this case.
We caution that the discussion in this section lacks full details at times; rather than giving a detailed and rigorous treatment, our goal is to use wrapped Floer cohomology to provide motivation for the constructions in the previous sections and connect our argument to the sheaf-theoretic argument from \cite{Shende}.

\subsection{Partially wrapped Floer cohomology}
\label{ssec:partially-wrapped}

\begin{figure}
\labellist
\small\hair 2pt
\pinlabel $W$ at 257 34
\pinlabel ${[0,\infty) \times V}$ at 350 188
\pinlabel $V$ at 286 116
\pinlabel ${DT^*([0,\infty) \times \Lambda)}$ at 265 284
\pinlabel $V'$ at 71 242
\pinlabel $V_c'$ at 242 178
\pinlabel ${ST^*([0,\infty)\times \Lambda)}$ at 259 260
\pinlabel ${\color{red} \Lambda}$ at 163 145
\pinlabel ${\color{red} [0,\infty) \times \Lambda}$ at 68 276
\endlabellist
\centering
\includegraphics[width=0.5\textwidth]{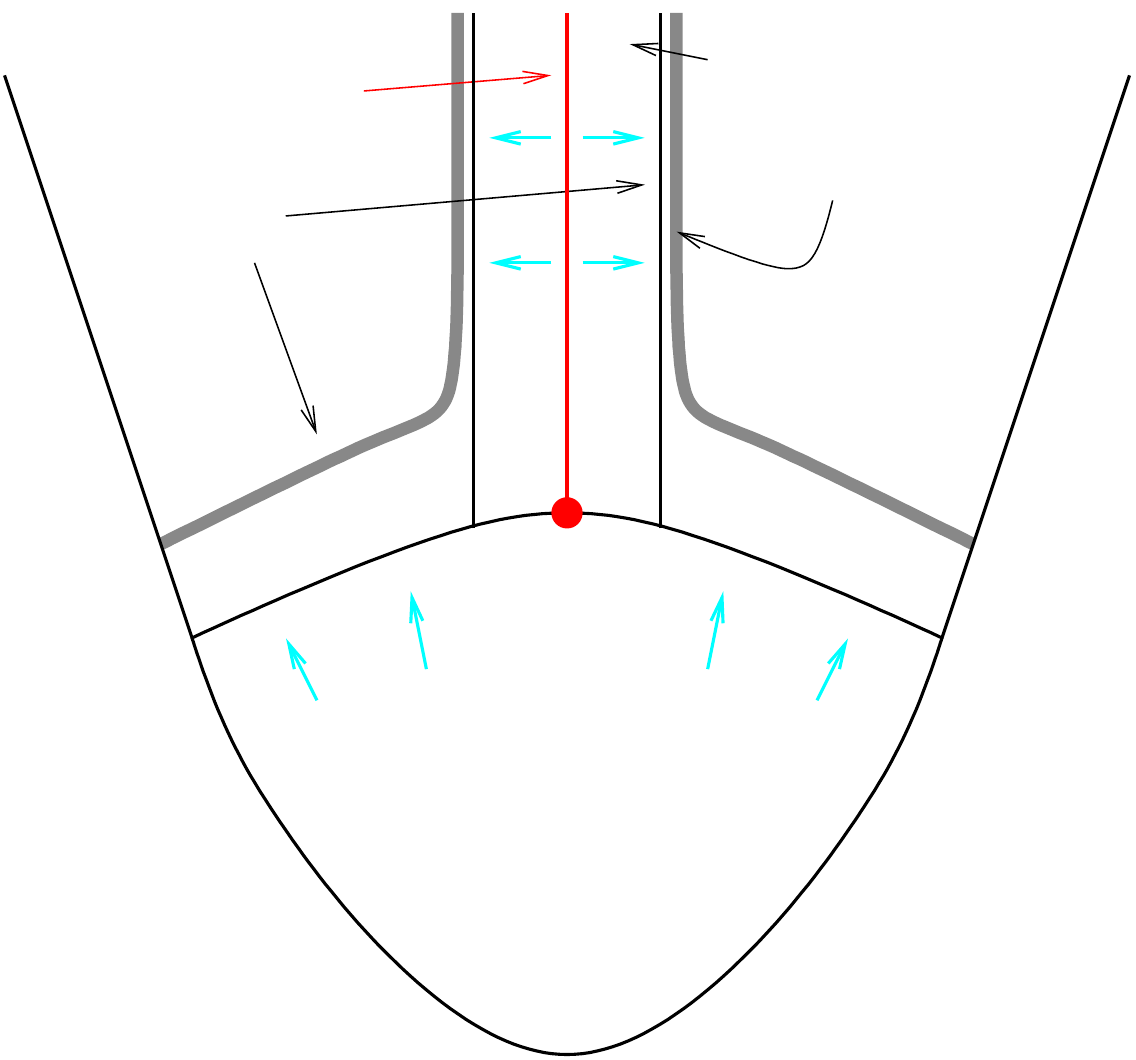}
\caption{The setup for partially wrapped Floer cohomology with wrapping stopped by $\Lambda$. The small arrows indicate the Liouville field $Z$.
}
\label{fig:wrapped1}
\end{figure}
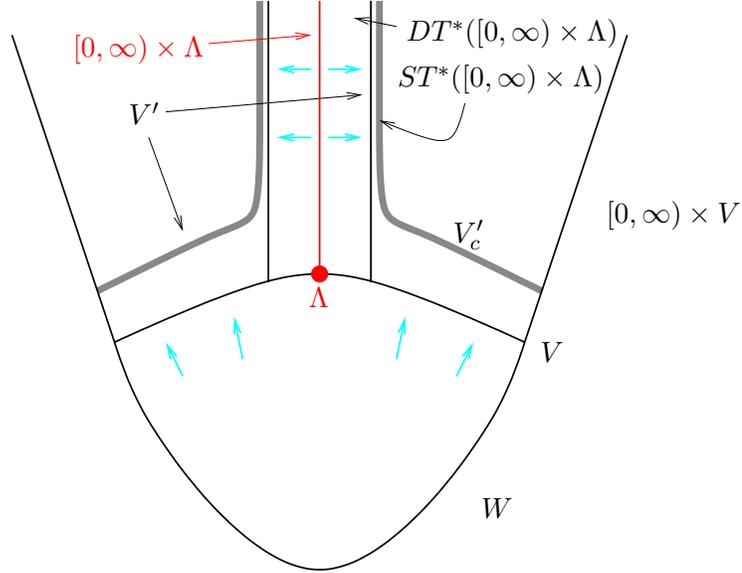

First we give an informal description of what we call partially wrapped Floer cohomology adapted to a Legendrian; see Figure~\ref{fig:wrapped1} for an illustration.
Let $W$ be a Weinstein manifold and let $\Lambda$ be a Legendrian submanifold in the ideal contact boundary $V$ of $W$. Consider the end $\approx [0,\infty)\times V$ of $W$ and the Lagrangian cylinder  $[0,\infty)\times\Lambda\subset [0,\infty)\times V$. We will consider a Liouville vector field $Z$ on $W$ (the Liouville condition says that $\mathcal{L}_{Z}\omega=\omega$, where $\mathcal{L}$ is the Lie derivative and $\omega$ the symplectic form) adapted to $L$. The condition of being adapted to $L$ means that $Z$ agrees with the standard Liouville vector field $p\cdot\partial_{p}$ along fibers in a cotangent neighborhood of $[0,\infty)\times\Lambda$. The contact boundary $V'$ of $W$ equipped with such a Liouville vector field $Z$ is then non-compact but in a controlled way as follows. It consists of a compact piece $V'_{\mathrm{c}}$ which is a manifold with boundary and the boundary $\partial V'_{\mathrm{c}}$ equals $ST^{\ast} ([0,\infty)\times\Lambda|_{\{0\}\times\Lambda})$. The non-compact part equals
\[
ST^{\ast} ([0,\infty)\times\Lambda)\approx \partial V_{\mathrm{c}}'\times[0,\infty).
\]
The flow of the Liouville vector field $Z$ starting at $V'$ then gives an end of $W$ that looks like $[1,\infty)\times V'$. As in the definition of ordinary symplectic homology or wrapped Floer (co)homology we consider a Hamiltonian on $W$ that is approximately equal to $0$ outside the end and is a function $h(r)$ of only the first coordinate in the end $[1,\infty)\times V'$, where the first two derivatives of $h$ are non-negative and $h(r)=ar+b$ for constants $a$ and $b$ for $r>2$. We next consider Lagrangian submanifolds in $W$ that agree with cylinders over compact Legendrian submanifolds in $V'$ in the end. For such Lagrangians we define a Floer (co)homology just as wrapped Floer (co)homology is defined but using the Hamiltonians described above. We call the resulting Floer cohomology \emph{partially wrapped Floer cohomology adapted to $\Lambda$} and write $HW(L_1,L_2)$ for the partially wrapped Floer cohomology of $L_1$ and $L_2$.

\begin{remark}
One can rephrase the above construction starting from a compact Weinstein domain $\overline{W}$ with a compact Lagrangian $\overline{L}$ with Legendrian boundary $\Lambda$. Then attach a neighborhood of the zero-section in $T^{\ast}\Lambda\times[0,\infty)$ to $\partial \overline{W}$ along $\Lambda$ and use a standard interpolation between the Liouville vector fields $\overline{Z}$ pointing outwards along $\partial\overline{W}$ and standard Liouville vector field $p\cdot\partial_{p}$ along fibers in the cotangent bundle attached. This results in a non-compact Weinstein domain with non-compact contact boundary $V'$ as above. Our Weinstein manifold is then obtained by adding the positive end of the symplectization of $V'$. Note that this is very similar to Lagrangian handle attachment. We will use this in calculations below. 	
\end{remark}

\begin{remark}
In our setting the wrapping is ``stopped'' by the Legendrian $\Lambda$, in the sense that the Hamiltonian flow does not cross the cylinder over $\Lambda$. The term ``partially wrapped'' is due to Auroux \cite{Auroux1,Auroux2}, but in the setting considered there, the wrapping is stopped by a codimension-$1$ submanifold of $V$ rather than a Legendrian. The constructions there and here are of a similar spirit but we do not intend to study the exact relation here. Our notion of partially wrapped Floer cohomology aligns more closely perhaps with \cite{nadler2014fukaya} and certainly with \cite{Syl}.
\end{remark}

\begin{remark}
For a prototypical example of partial wrapping consider $\R^{2n}\approx T^{\ast}\R^{n}$ with coordinates $(q,p)\in\R^{n}\times\R^{n}$.
\label{rmk:Rn}
We equip it with the Liouville vector field $p\cdot\partial_{p}$ and Hamiltonian $H(p,q)=\frac12 p^{2}$. We use this Hamiltonian to define partially wrapped Floer cohomology adapted to the boundary sphere in the 0-section; that is, in this case $W = D^{2n}$ and $\Lambda \subset S^{2n-1}$ is the Legendrian sphere $\{p=0\}$. Unlike 
the usual Fukaya category for $D^{2n}$, which is trivial, the partially wrapped Fukaya category is nontrivial: for instance, if $F$ denotes the fiber then $HW(F,F)$ has rank 1.
\end{remark}

Below we will give a gluing operation similar to Lagrangian handle attachment/Legendrian surgery that allows us to construct Liouville manifolds with wrapping stopped by a Legendrian 
and to compute the resulting partially wrapped Floer cohomology in the spirit of the surgery formula of \cite{BEE}. Although this approach works rather generally we will restrict to the case of $W_K$ mentioned above. Here
the relevant Legendrian surgery operation that we perform can be thought of as an $S^{1}$-family of punctured disk attachments. We start by describing an individual punctured disk attachment.

\subsection{Attaching a punctured handle in dimension 2}
\label{ssec:babysurgery}
Before proceeding to $W_K$, we first consider a model case in $1$ dimension down. Our goal here is to describe partially wrapped Floer cohomology for a Lagrangian fiber disk as well as its pair-of-pants product. See Figure~\ref{fig:wrapped-R2} for a picture of the construction.

\begin{figure}
\labellist
\small\hair 2pt
\pinlabel $DT^*\R^2$ at 418 240
\pinlabel $W_0(\epsilon)$ at 130 310
\pinlabel $0$ at 252 198
\pinlabel $\xi$ at 127 198
\pinlabel $\eta$ at 270 320
\pinlabel $c^+$ at 302 224
\pinlabel $c^-$ at 211 140
\pinlabel $c_{\xi 0}$ at 190 226
\pinlabel $c_{0 \xi}$ at 180 160
\pinlabel ${\color{green} L_\xi}$ at 127 230
\pinlabel ${\color{green} \Lambda_\xi}$ at 86 236
\pinlabel ${\color{blue} \R^2}$ at 378 184
\pinlabel ${\color{red} L_0}$ at 259 143
\pinlabel ${\color{red} \Lambda_0}$ at 297 135
\pinlabel ${\color{red} D'}$ at 336 37
\pinlabel ${\color{red} D_\epsilon T^*D'}$ at 252 381
\pinlabel ${\color{magenta} \C_\eta}$ at 166 350
\endlabellist
\centering
\includegraphics[height=0.5\textwidth]{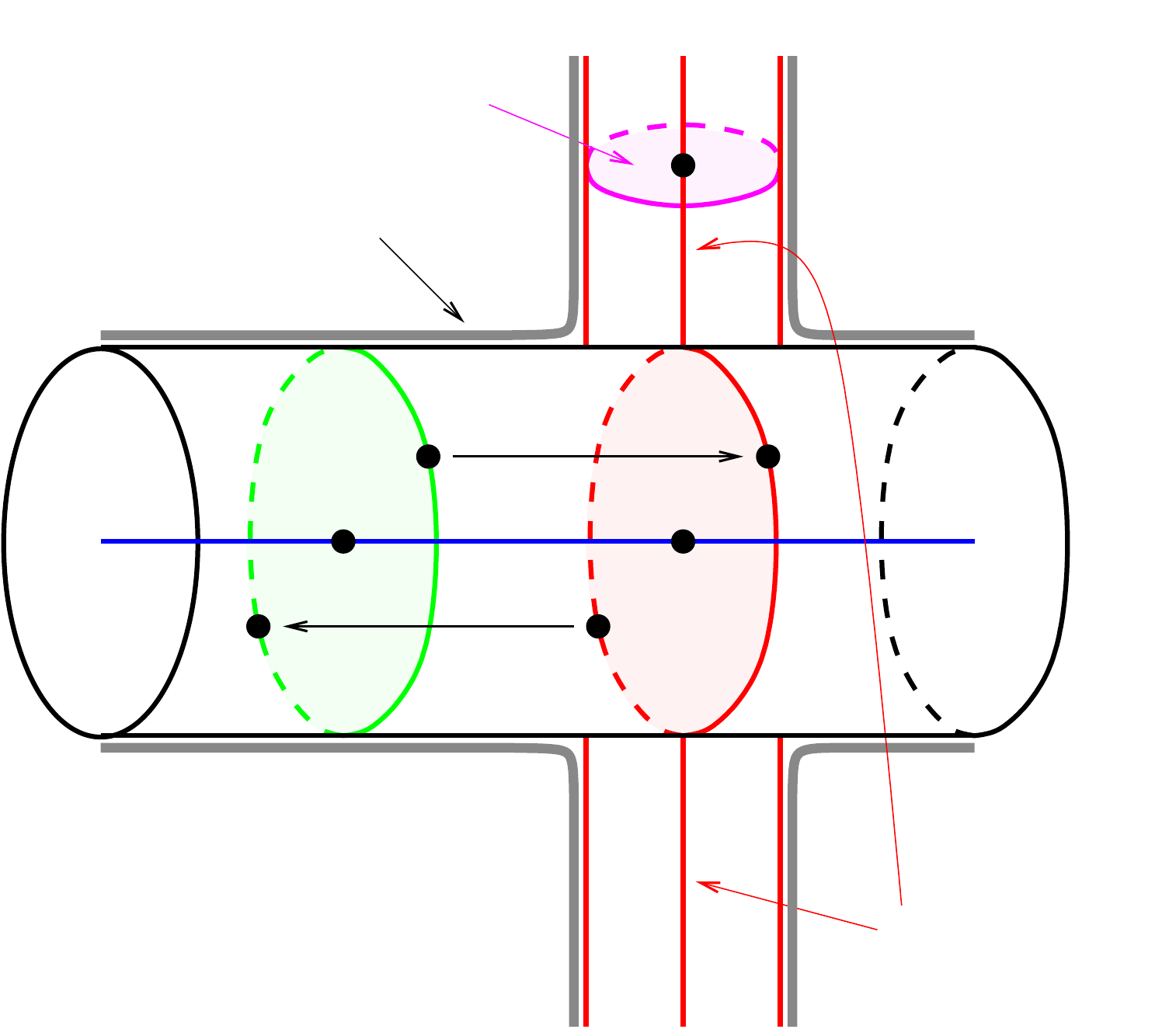}
\caption{
The construction of $W_0(\epsilon)$ from $T^*\R^2$.
}
\label{fig:wrapped-R2}
\end{figure}

Consider $T^{\ast}\R^{2}$ as above and fix the cotangent fiber $L_0$ thought of as the conormal of the point at the origin. We will construct a Weinstein manifold $W_0(\epsilon)$ with Lagrangian skeleton $\R^{2}\cup L_0$ by attaching a punctured Lagrangian handle to the Legendrian ideal boundary $\Lambda_0$ of $L_0$. The manifold $W_0(\epsilon)$ with its Liouville vector field is then the ambient space for partially wrapped Floer cohomology adapted to $\partial L_0$ and the boundary circle of the 0-section $\R^{2}$.  Let $D'$ denote a punctured disk equipped with the complete hyperbolic metric. To get a concrete model, consider the upper half plane with the hyperbolic metric:
\[
H=\{(x,y)\colon y>0\},\quad ds^{2}=\frac{dx^{2}+dy^{2}}{y}.
\]
Let $H_{y\ge 1}=H\cap\{y\ge 1\}$ and let $\Z$ act on $H_{\ge 1}$ by $n(x,y)=(x+n,y)$. Then
\[
D'=H_{\ge 1}/\Z.
\]
The metric induces a Hamiltonian $\frac12 p^{2}$ on $T^{\ast} D'$ for which the Liouville vector field $p\cdot\partial_{p}$ is gradient like. We now perform Lagrangian handle attachment of the $\epsilon$-disk cotangent bundle $D_{\epsilon}T^{\ast}D'$, gluing $\partial D'$ to $\Lambda$. This gives a Weinstein domain $W_0(\epsilon)$ of the desired type.

We next consider Reeb chords of the Legendrian boundaries of basic Lagrangian disks in $W_0(\epsilon)$. Let $\xi$ denote a point in $\R^{2}$, $\xi\ne 0$, and let $\eta$ denote a point in $D'$.
Let $L_\xi$ and $\C_\eta$ denote the Lagrangian disk fibers at $\xi$ and $\eta$. Let $\mathcal{C}_{\xi\xi}$ denote the set of Reeb chords from $\partial L_\xi$ to itself in $\partial W_0(\epsilon)$, $\mathcal{C}_{\xi\eta}$ denote the chords from $\partial L_\xi$ to $\partial \C_\eta$, and interpret $\mathcal{C}_{\eta\xi}$ and $\mathcal{C}_{\eta\eta}$ similarly. In $ST^{\ast}\R^{2}$, there is a unique Reeb chord $c_{\xi 0}$ connecting $\Lambda_\xi$ to $\Lambda_0$ (following the straight line segment from $\xi$ to $0$) and a unique Reeb chord $c_{0 \xi}$ connecting $\Lambda_0$ to $\Lambda_\xi$. Let $c^{+},c^{-}\in\Lambda_0$ denote the endpoint of $c_{\xi 0}$ and the beginning point of $c_{0 \xi}$, respectively.
\begin{lemma}
For $\epsilon>0$ sufficiently small, the Reeb chords in $\partial W_0(\epsilon)$ are as follows:
\label{lem:ReebR2}
\begin{itemize}
\item The set $\mathcal{C}_{\eta\eta}$ is in natural 1-1 correspondence with homotopy classes of loops in $D'$ connecting $\eta$ to $\eta$.
\item The set $\mathcal{C}_{\eta\xi}$ is in natural 1-1 correspondence with homotopy classes of paths in $D'$ connecting $p$ to $c^{-}$.
\item The set $\mathcal{C}_{\xi\eta}$ is in natural 1-1 correspondence with homotopy classes of paths in $D'$ connecting $c^{+}$ to $\eta$.
\item The set $\mathcal{C}_{\xi\eta}$ is in natural 1-1 correspondence with homotopy classes of loops in $D'$ connecting $c^{+}$ to $c^{-}$.
\end{itemize}	
\end{lemma}

\begin{proof}
We first note that there are no Reeb chords connecting $\Lambda_\xi$ to itself before the handle attachment, and as stated above there are unique Reeb chords between $\Lambda_{\xi}$ to $\Lambda_{0}$. Furthermore, since the Reeb flow on the boundary $S_{\epsilon}T^{\ast} D'$ of $D_{\epsilon}T^{\ast}D'$ is geodesic flow with respect to the hyperbolic metric on $D'$, there is a unique Reeb chord in $S_{\epsilon}T^{\ast} D'$ connecting $\eta$ to $c^-$ in any homotopy class of paths, and similarly there is a unique Reeb chord connecting $c^+$ to $\eta$ in any homotopy class.
A Reeb chord in $\partial W_0(\epsilon)$ converges as $\epsilon\to 0$ to an alternating word of Reeb chords inside and outside the handle. A straightforward fixed point argument as in \cite{BEE} shows that any such word corresponds to a chord after the surgery.	
\end{proof}

We can now compute the partially wrapped Floer cohomology $HW(L_{\xi},L_{\xi})$.
Fixing a path connecting $c^{+}$ to $c^{-}$, we have as in Lemma~\ref{lem:ReebR2} a 1-1 correspondence between Reeb chords in $ST^{\ast}D'$ connecting the fibers of these points and homotopy classes of loops in $D'$. We write $m$ for the generator of $\pi_{1}(D')$ and $c_{\xi 0}m^{k}c_{0\xi}$ for the Reeb chord in $\partial V_{0}$ corresponding to this word of Reeb chords.
To compute $HW(L_\xi,L_{\xi})$ we shift $L_\xi$ in the positive Reeb direction to obtain a shifted Lagrangian $L_{\xi}'$ and then consider chords from $L_{\xi}'$ to $L_{\xi}$. Generators for the model of $HW(L_{\xi}',L_{\xi})$ are then the following:
\begin{itemize}
\item an intersection point between $L_\xi$ and $L_{\xi}'$ in the middle of the disk $L_\xi$;
\item Reeb chords of the form $c_{\xi0}m^{k} c_{0\xi}$.	
\end{itemize}
As in \cite{BEE}, there is a chain map from $HW(L_{\xi}',L_{\xi})$ to this complex with the differential induced by the DGA differential before surgery, which is a chain isomorphism by an action filtration argument. Since the DGA differential is trivial we find that the wrapped Floer cohomology is generated by the generators of the model complex.

Finally, we describe the pair-of-pants product on $HW(L_{\xi},L_{\xi})$ using the above isomorphism. To this end, we must first describe that product on the model for wrapped Floer cohomology without Hamiltonian. The description is a consequence of arguments used in \cite{EHK} and \cite{EO} as follows. Let $S$ denote a horizontal strip in $\mathbb{C}$ with a horizontal slit pointing rightwards. We can view Floer holomorphic disks with two positive punctures as maps whose domain is $S$. For $\tau\in\R$, consider maps $u\colon (S,\partial S) \to (W_0(\epsilon),L_\xi)$ that solve the equation
\[
(du+\beta_{\tau}\otimes X_{H})^{0,1} =0,
\]
where $X_H$ is the Hamiltonian vector field and where $\beta$ is a non-positive 1-form on $S$ that in standard coordinates $s+it$ on $S$ interpolates in a region of width $1$ around $s=\tau$ between $dt$ to the left of the region and $0$ to the right.
We consider the solution space of dimension $0$ as $\tau$ ranges over $\R$. For $\tau=+\infty$ we find the isomorphism map from linearized contact homology to wrapped Floer cohomology followed by the standard pair-of-pants product. For $\tau=-\infty$ we find the product on linearized contact homology that counts disks with two positive and one negative puncture followed by the isomorphism. Split curves at other values of $\tau$ count rigid curves with interpolations with a curve contributing to boundary attached either above or below. It follows that these configurations do not contribute on the level of homology. We conclude that the pair-of-pants product on wrapped Floer cohomology corresponds to a count of disks with two positive punctures. (These disks may also involve punctures at Lagrangian intersection points.)

With this established we use a similar argument to transport the pair-of-pants product to the description of the wrapped Floer cohomology before the surgery: compose the product with the isomorphism map and study splittings. Here the other end of moduli space corresponds to two isomorphism disks joined by a disk with two mixed positive punctures (and several negative punctures at pure Reeb chords). We find that on words the product counts disks with two positive mixed punctures. Drawing the Lagrangian projection of $\Lambda_{\xi}$ and $\Lambda_{0}$ before the surgery, one can check that there are two such disks and they lie in distinct homotopy classes $1$ and $m$. Thus if $\cdot$ denotes the product we have that the empty word $1$ in the DGA acts as the identity and
\[
(c_{\xi0}m^{\alpha_1}c_{0\xi}) \cdot (c_{\xi 0}m^{\alpha_2}c_{0\xi}) =
c_{\xi0}m^{\alpha_1+\alpha_2} c_{0\xi} -  c_{\xi0}m^{\alpha_1+\alpha_2+1} c_{0\xi}.
\]

\begin{remark}
The above calculations show that the wrapped Floer cohomology of $HW(L_{\xi},L_{\xi})$ with its pair-of-pants product is ring isomorphic to
$\Z[m^{\pm 1}]$ via the map
\[
1\mapsto 1,\quad c_{\xi 0}m^{k} c_{0\xi} \mapsto m^{k}(1-m).
\]
Not coincidentally, $\Z[m^{\pm 1}]$ is the group ring of $\pi_1(\R^2\setminus\{0\})$, cf.\ Remark~\ref{rmk:wrapped} below.
\end{remark}

\subsection{The construction of $W_{K}$}\label{ssec:W_K}
We next turn to the construction of $W_K$. Consider $T^{\ast}\R^{3}$ and let $L_K$ denote the conormal Lagrangian as usual. In analogy with the construction of $W_0$ and the surgery calculation above, we want to construct a Weinstein domain $W_K$ so that if we want to compute partially wrapped Floer cohomology in $T^{\ast}\R^3$ adapted to $\partial \R^{3} \cup \Lambda_K$, we can write the partially wrapped Floer cohomology of fibers over the top-dimensional strata of the Lagrangian skeleton in terms of the contact homology DGA of the Legendrian attaching locus.

\begin{figure}
\labellist
\small\hair 2pt
\pinlabel $DT^*\R^3$ at 418 240
\pinlabel $W_K$ at 130 310
\pinlabel $K$ at 252 198
\pinlabel $p$ at 127 167
\pinlabel ${\color{green} L_p=\B}$ at 127 200
\pinlabel ${\color{green} \Lambda_p}$ at 86 236
\pinlabel ${\color{blue} \R^3}$ at 378 184
\pinlabel ${\color{red} L_K}$ at 259 143
\pinlabel ${\color{red} \Lambda_K}$ at 300 135
\pinlabel ${\color{red} S^1 \times D'}$ at 360 37
\pinlabel ${\color{red} D_\epsilon T^*(S^1\times D')}$ at 252 381
\pinlabel ${\color{magenta} \C}$ at 166 350
\pinlabel $DT^*\R^3$ at 598 107
\pinlabel $W_K$ at 598 48
\pinlabel ${\color{green} \B}$ at 542 161
\pinlabel ${\color{green} \Lambda_p}$ at 530 229
\pinlabel ${\color{red} L_K}$ at 598 173
\pinlabel ${\color{red} \Lambda_K}$ at 598 235
\pinlabel ${\color{red} S^1 \times D'}$ at 688 255
\pinlabel ${\color{magenta} \C}$ at 498 291
\endlabellist
\centering
\includegraphics[height=0.5\textwidth]{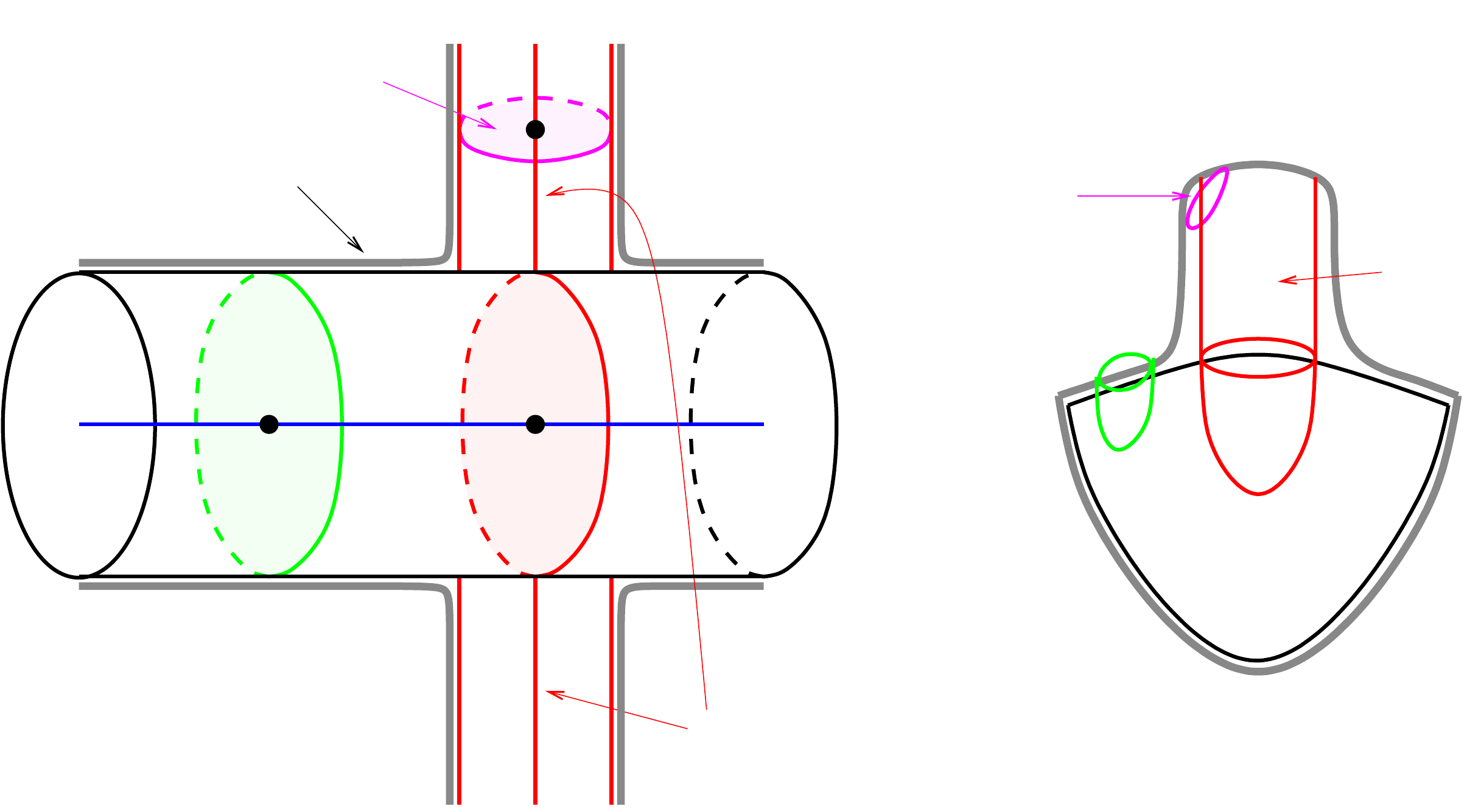}
\caption{The construction of $W_K$ from $T^*\R^3$ (left), and an alternate schematic picture (right).
}
\label{fig:wrapped-R3}
\end{figure}
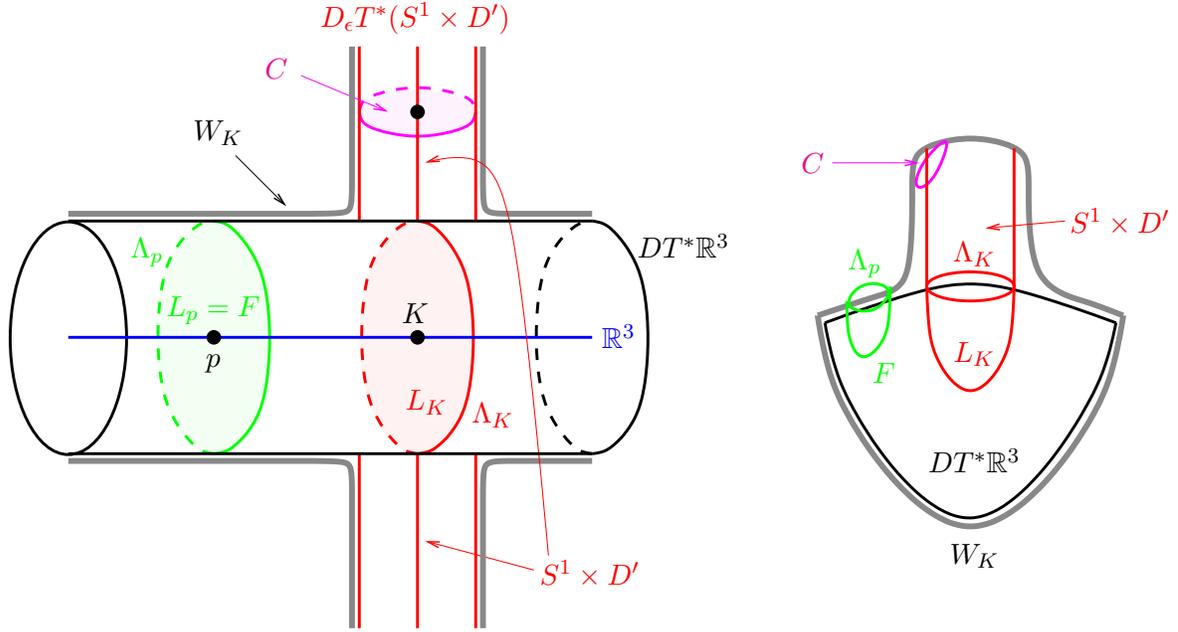

We will construct $W_K$ by attaching the cotangent bundle of $[0,\infty)\times T^{2}$ to $T^{\ast}\R^{3}$ as follows; see Figure~\ref{fig:wrapped-R3}. Note first that the conormal torus of the knot $\Lambda_{K}$ can be viewed as an $S^{1}$-family of conormal circles over points in $K$; now consider $S^{1}\times D'$ with the product metric (the metric on $D'$ is the complete hyperbolic metric as in Section \ref{ssec:babysurgery}) and its cotangent bundle $T^{\ast}(S^{1}\times D')$. We perform Lagrangian attachment of the $\epsilon$-disk cotangent bundle $D_{\epsilon}T^{\ast}(S^{1}\times D')$ to the disk cotangent bundle $DT^{\ast}\R^3$ by gluing $S^{1}\times \partial D'$ to $\Lambda_K$. This is our desired Weinstein domain $W_{K}(\epsilon)$, which depends on the size $\epsilon$ of the disk bundle attached. In analogy with the above we calculate the wrapped Floer cohomology for two basic Lagrangian disks in $W_K$, the \emph{outer} disk $\B$ which is the fiber of a point in $\R^{3}$ not on $K$ and the \emph{inner} disk $\C$, the fiber over a point in $S^{1}\times D'$.

We begin by a description of the Reeb chords after surgery in terms of data before. We write $\mathcal{C}(\Lambda_1,\Lambda_2)$ for the set of Reeb chords in $\partial W_K$ starting on $\Lambda_1$ and ending on $\Lambda_2$. As in Section~\ref{ssec:LCHlink}, write $\mathcal{R}^{KK}$ for Reeb chords of $\Lambda_{K}$ in $ST^{\ast}\R^{3}$, and $\mathcal{R}^{Kp}$ and $\mathcal{R}^{pK}$ for Reeb chords in $ST^{\ast}\R^{3}$ to $\Lambda_K$ from $\Lambda_p = \partial B$, and to $\Lambda_p$ from $\Lambda_K$, respectively.
Also, write $\pi_{1}(S^{1}\times D')=\Z^2=\langle l,m\rangle$.

\begin{lemma}
For $\epsilon>0$ sufficiently small we have the following.
\label{lem:Reeb-chords}

\begin{itemize}
\item There is a natural 1-1 correspondence between the elements of $\mathcal{C}(\partial \C,\partial \C)$ and words of the form
\[
l^{\alpha_1}m^{\beta_1}\, c_1\, l^{\alpha_2}m^{\beta_2}\, c_2\, \cdots c_r\,l^{\alpha_{r+1}}m^{\beta_{r+1}},
\]
where $c_j\in\mathcal{R}^{KK}$ and either $r \geq 1$ or $r=0$ and $(\alpha_1,\beta_1) \neq (0,0)$.
\item  There is a natural 1-1 correspondence between the elements of $\mathcal{C}(\partial \C,\partial \B)$ and words of the form
\[
a l^{\alpha_1}m^{\beta_1}\, c_1\, l^{\alpha_2}m^{\beta_2}\, c_2 \cdots c_r\,l^{\alpha_{r+1}}m^{\beta_{r+1}},
\]
where $r \geq 0$, $c_j\in\mathcal{R}^{KK}$, and $a\in\mathcal{R}^{pK}$.
\item There is a natural 1-1 correspondence between the elements of $\mathcal{C}(\partial \B,\partial \C)$ and words of the form
\[
l^{\alpha_1}m^{\beta_1}\, c_1 \,l^{\alpha_2}m^{\beta_2} \,c_2\, \cdots c_r\,l^{\alpha_{r+1}}m^{\beta_{r+1}}\,b,
\]
where $r \geq 0$, $c_j\in\mathcal{R}^{KK}$, and $b\in \mathcal{R}^{Kp}$.
\item  There is a natural 1-1 correspondence between the elements of $\mathcal{C}(\partial \B,\partial \B)$ and words of the form
\[
a\, l^{\alpha_1}m^{\beta_1} \, c_1 \,l^{\alpha_2}m^{\beta_2}\, c_2 \cdots c_r\,l^{\alpha_{r+1}}m^{\beta_{r+1}}\,b,
\]
where $r \geq 0$, $c_j\in\mathcal{R}^{KK}$, $a\in\mathcal{R}_{pK}$, and $b\in \mathcal{R}_{Kp}$.
\end{itemize}	
\end{lemma}

\begin{proof}
As $\epsilon\to 0$ Reeb chords in $W_K$, traversed from bottom to top, have limits as the words described, where each $c_i,a,b$ is also traversed backwards. (We follow the Reeb chords in $W_K$ backwards in order to obtain composable words; note that composable words in the spirit of Section~\ref{ssec:LCHlink} consist of a concatenation of Reeb chords whose consecutive endpoints line up \textit{if} the chords are traversed backwards.)
As in \cite{BEE} a fixed point argument shows that every word glues to a unique Reeb chord for small $\epsilon>0$.
\end{proof}

\begin{remark}
Note that in the construction of $W_K$ above we could have attached any model of $T^{\ast} S^{1}\times (D^{2}\setminus \{0\})$ with Reeb flow agreeing with that of a complete metric at infinity. The reason for the particular choice we made was to get a model suitable for direct calculations. This can be compared to the concrete handle model used in \cite{BEE} which is not the only possible choice but well suited for calculations.
\end{remark}

\subsection{Partially wrapped Floer cohomology in $W_K$}
\label{ssec:wrapped-WK}

We now turn to partially wrapped Floer cohomology in $T^*\R^3$ via $W_K$. In particular, we study $HW(L_1,L_2)$ where each of $L_1,L_2$ is one of the Lagrangians $\B,\C$ in $W_K$.
The complexes $CW(\C,\B)$ and $CW(\B,\C)$ are generated by elements of $\mathcal{C}(\partial \C,\partial \B)$ and $\mathcal{C}(\partial \B,\partial \C)$, respectively, while $CW(\C,\C)$ and $CW(\B,\B)$ are generated by elements of $\mathcal{C}(\partial \C,\partial \C)$ and $\mathcal{C}(\partial \B,\partial \B)$, along with an additional generator corresponding to an intersection point between the Lagrangian and its pushoff.

Now note that the Reeb chords in Lemma~\ref{lem:Reeb-chords} can be viewed as composable elements in the DGA $\A_{\Lambda_K \cup \Lambda_p}$. For $CW(\C,\B)$ and $CW(\B,\C)$, these are precisely the composable words with exactly $1$ mixed Reeb chord. For $CW(\C,\C)$, the Reeb chords in $\mathcal{C}(\partial \C,\partial \C)$ are all possible words involving only pure Reeb chords of $\Lambda_K$, with the exception of $l^0 m^0 = 1$; but $CW(\C,\C)$ has an extra generator corresponding to $1$. That is, we have:
\begin{align*}
CW(\C,\C) &= \A_{\Lambda_K,\Lambda_K}^{(0)} \\
CW(\C,\B) &= \A_{\Lambda_p,\Lambda_K}^{(1)} \\
CW(\B,\C) &= \A_{\Lambda_K,\Lambda_p}^{(1)}.
\end{align*}
For $CW(\B,\B)$, the Reeb chords in $\mathcal{C}(\partial \B,\partial \B)$ are the composable words beginning and ending on $\Lambda_p$ with exactly $2$ mixed Reeb chords, and $CW(\B,\B)$ is generated by these along with the extra generator. Thus the generating set of $CW(\B,\B)$ is naturally the same as the generating set of $\F^0\A_{\Lambda_p,\Lambda_p}/ \F^4\A_{\Lambda_p,\Lambda_p}$ (the extra generator plays the role of $1$), and we have:
\[
CW(\B,\B) = \Z \oplus \A_{\Lambda_p,\Lambda_p}^{(2)}.
\]

\begin{figure}
\labellist
\small\hair 2pt
\pinlabel ${\color{magenta} \C}$ at 17 137
\pinlabel ${\color{magenta} \C}$ at 83 137
\pinlabel ${\color{magenta} \C}$ at 209 137
\pinlabel ${\color{magenta} \C}$ at 267 137
\pinlabel ${\color{green} \B}$ at  141 137
\pinlabel ${\color{green} \B}$ at  337 137
\pinlabel ${\color{green} \B}$ at  393 137
\pinlabel ${\color{green} \B}$ at  463 137
\pinlabel ${\color{red} L_K}$ at 18 64
\pinlabel ${\color{red} L_K}$ at 83 64
\pinlabel ${\color{red} L_K}$ at 208 64
\pinlabel ${\color{red} L_K}$ at 269 64
\pinlabel ${\color{red} L_K}$ at 48 48
\pinlabel ${\color{red} L_K}$ at 174 48
\pinlabel ${\color{red} L_K}$ at 300 48
\pinlabel ${\color{red} L_K}$ at 427 48
\endlabellist
\centering
\includegraphics[width=0.8\textwidth]{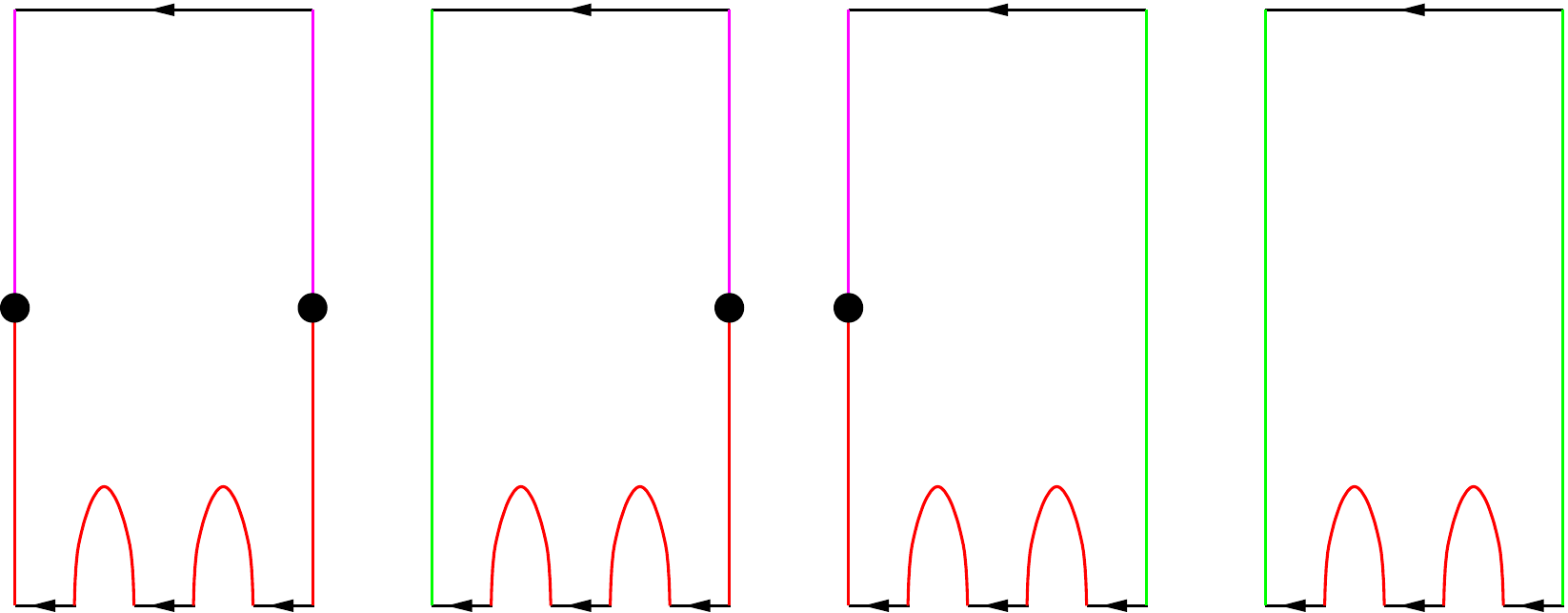}
\caption{
The holomorphic disks defining the maps $\Phi$ on, from left to right:
$CW(\C,\C)$; $CW(\C,\B)$; $CW(\B,\C)$; $CW(\B,\B)$. In each case, a Reeb chord with endpoints on $\partial \B \cup \partial \C$ is the input at the top, and the output at the bottom is a product of Reeb chords and homology classes for $\Lambda_K \cup \Lambda_p$. For the $L_K$ boundary conditions, we extend $L_K$ into $W_K$ by taking its union with $S^1\times D'$, see Figure~\ref{fig:wrapped-R3}. Arrows represent orientations of Reeb chords.}
\label{fig:wrapped-isom}
\end{figure}
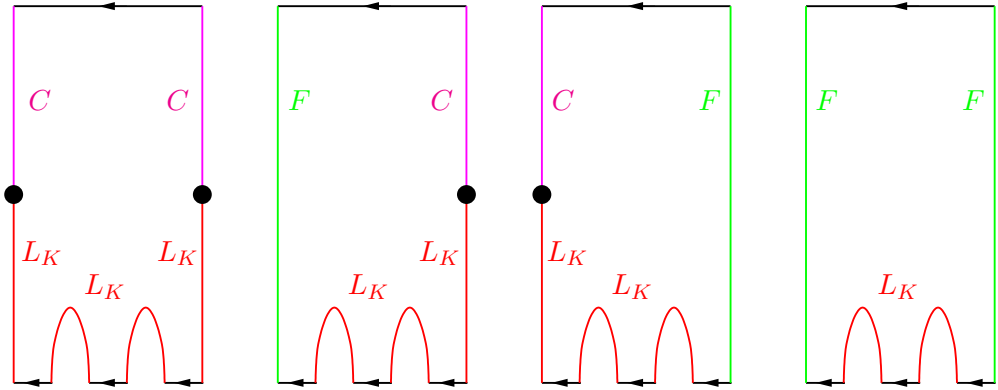

We next claim that the differentials on the Floer complexes agree with the differentials from the DGA. Indeed, there are chain maps $\Phi$ from $CW$ to the corresponding filtered quotients of $\A_{\Lambda_K \cup \Lambda_p}$ that count the following holomorphic disks:
\begin{itemize}
\item For $CW(\C,\C)$, disks with one positive puncture, two Lagrangian intersection punctures at $\C\cap (S^{1}\times D')$, and negative punctures at chords in $\mathcal{R}^{KK}$.
\item For $CW(\C,\B)$, disks with one positive puncture, a first negative puncture at a chord in $\mathcal{R}^{pK}$, several negative punctures at chords in $\mathcal{R}^{KK}$, and a Lagrangian intersection puncture at $\C\cap (S^{1}\times D')$.
\item For $CW(\B,\C)$, disks with one positive puncture, a Lagrangian intersection at $\C\cap (S^{1}\times D')$, several negative punctures at chords in $\mathcal{R}^{KK}$, and a last negative puncture at a chord in $\mathcal{R}^{Kp}$.
\item For $CW(\B,\B)$, disks with one positive puncture, a first negative puncture at a chord in $\mathcal{R}^{pK}$, several negative punctures at chords in $\mathcal{R}^{KK}$, and a last negative puncture at a chord in $\mathcal{R}^{pK}$.
\end{itemize}
\noindent See Figure~\ref{fig:wrapped-isom}.

\begin{lemma}
The chain maps $\Phi$ are chain isomorphisms.
\label{lma:wrapped-isom}
The degrees of the chain maps on $CW(\B,\B)$ and $CW(\C,\C)$ are 1 and 0, respectively, and there is a choice of reference path connecting $\partial \C$ to $\partial \B$ so that the grading shifts on $CW(\C,\B)$ and $CW(\B,\C)$ are $1$ and $0$, respectively.
\end{lemma}

\begin{proof}
This follows from constructions of curves over almost trivial strips and an energy filtration argument as in \cite{BEE}. We will omit a full justification of the degree statement since this involves a careful consideration of the grading on the wrapped Floer complex, which we have not discussed here. However, note that two Lagrangian corners of codimension 3 gives a contribution to the dimension by $-3+2=-1$ which means that the grading of the Reeb chord at the positive puncture is one above the sums of the gradings at the negative end. (Compare the grading shift by $(n-2)$ in the wrapped homology surgery isomorphism from \cite{BEE}.)
\end{proof}

Lemma~\ref{lma:wrapped-isom} allows us to relate the wrapped Floer picture to the KCH-triple. The maps $\Phi$ descend in homology to isomorphisms
\begin{align*}
HW^0(\C,\C) & \stackrel{\cong}{\to} H_0(\A_{\Lambda_K,\Lambda_K}^{(0)}) \cong R_{KK} \\
HW^0(\C,\B) & \stackrel{\cong}{\to} H_1(\A_{\Lambda_p,\Lambda_K}^{(1)}) \cong R_{pK} \\
HW^0(\B,\C) & \stackrel{\cong}{\to} H_0(\A_{\Lambda_K,\Lambda_p}^{(1)}) \cong R_{Kp}.
\end{align*}
That is, the KCH-triple is precisely given by partially wrapped Floer cohomology in lowest degree.
In addition, $\Phi$ gives an isomorphism
\[
HW^0(\B,\B) \stackrel{\cong}{\to} \Z \oplus H_1(\A_{\Lambda_p,\Lambda_p}^{(2)}) \cong \Z \oplus R_{pp},
\]
where the second isomorphism comes from Proposition~\ref{prop:Rpp}.

We examine this last isomorphism $HW^0(\B,\B) \cong \Z\oplus R_{pp}$ more carefully. The right hand side has a product given by $\mu$. The left hand side also has a product: on $CW(\B,\B)$, there is a pair-of-pants product
\[
\Pi\colon CW(\B,\B)\otimes CW(\B,\B)\to CW(\B,\B)
\]
which here has degree $0$, cf.\ Remark~\ref{r:gendim}.
We claim that this agrees in homology with $\mu$ under the isomorphism $\Phi$.

\begin{lemma}\label{l:pairofpants}
There is a degree $0$ pairing $P\colon CW(\B,\B)\otimes CW(\B,\B)\to CW(\B,\B)$ such that
\[
\Phi\circ\Pi - \mu\circ(\Phi\otimes\Phi) + P\circ(1\otimes d+d\otimes 1) - \partial\circ P = 0
\]
where $d$ is the differential on $CW(\B,\B)$. In particular, on homology $\Phi$ sends $\Pi$ to $\mu$.
\end{lemma}

\begin{proof}
Consider the $1$-dimensional moduli space of disks in the cobordism with two positive punctures at chords connecting $\partial \B$ to $\partial \B$ and with negative punctures at a mixed chord $\partial \B\to\Lambda_K$ followed by an alternating word of chords and homotopy classes on $\Lambda_K$ and then a mixed chord $\Lambda_K\to \partial \B$. If $P$ counts the corresponding rigid curves in the cobordism then terms in the left hand side of the equation count ends of a 1-dimensional compact oriented moduli space. The lemma follows.
\end{proof}

Thus the isomorphism between $HW^0(\B,\B)$ and $\Z \oplus R_{pp}$ is a ring isomorphism.
By Proposition~\ref{prop:ring-isom}, we conclude that:
\[
HW^0(\B,\B) \cong \Z[\pi_1(\R^3 \setminus K)].
\]
This is as expected by the general theory, see Remark~\ref{rmk:wrapped} below.

\begin{remark}
We can use the isomorphism $\Phi$ to describe other pair-of-pants products of the form
$HW(L_2,L_3) \otimes HW(L_1,L_2) \to HW(L_1,L_3)$, as follows:
\begin{itemize}
\item Exactly as for $HW(\B,\B)$ above, the product
\[
HW^0(\B,\C)\otimes HW^0(\C,\B) \to HW_{1}(\C,\C)
\]	
corresponds to the product $\mu$ counting rigid disks with two positive mixed punctures and several negative pure punctures.
\item Arguing as in \cite[Theorem 5.8]{BEE}, we find that the products
\begin{align*}
&HW^0(\C,\C)\otimes HW^0(\C,\C) \to HW^0(\C,\C),\\
&HW^0(\C,\B)\otimes HW^0(\C,\C) \to HW^0(\C,\B),\\
&HW^0(\C,\C)\otimes HW^0(\B,\C) \to HW^0(\B,\C),\\
&HW^0(\C,\B)\otimes HW^0(\B,\C) \to HW^0(\B,\B)
\end{align*}
correspond to concatenation in the DGA's.
\end{itemize}	
\end{remark}

\begin{remark}\label{rmk:wrapped}
We have seen that $HW^0(\C,\C)$ is knot contact homology in degree $0$, while $HW^0(\B,\B)$ is the group ring of $\pi_1(\R^3\setminus K)$. Both of these are consistent with general considerations in wrapped Floer cohomology. The picture for the partially wrapped Floer cohomology of $\C$ is analogous to the calculation of the wrapped Floer cohomology of a cocore disk in Legendrian handle attachment, which via the surgery isomorphism \cite{BEE} is isomorphic to the Legendrian contact homology of the attaching sphere. In our setting, $\C$ plays the role of the cocore disk, while $\Lambda_K$ plays the role of the Legendrian attaching sphere. Compared to the situation in \cite{BEE}, we are attaching a different handle, and the Reeb flow inside the handle is more complicated but easy to control from the point of view of holomorphic disks since all Reeb chords inside the punctured handle have index $0$.

Our calculation of the partially wrapped Floer cohomology of $\B$ relates it to the DGA of $\Lambda_{K}\cup \partial \B=\Lambda_K\cup \Lambda_p$ again as in \cite{BEE}, but this time we compute the wrapped Floer cohomology of $\B$ in the manifold with the punctured solid torus handle attached in terms of data before the attachment. Via string homology of broken strings on the Lagrangian skeleton we find that in the lowest degree the partially wrapped Floer cohomology of $\B$ is given by $\Z[\pi_{1}(\R^{3}\setminus K)]$. This should be compared to the result from \cite{AS} (cf.\ \cite{Ab}) that the wrapped Floer cohomology of a cotangent fiber in $T^*Q$ is isomorphic to the homology of the based loop space, with the pair-of-pants product in wrapped Floer cohomology mapping to the Pontryagin product on chains of loops, see Section \ref{ssec:direct} below for further discussion of this relation.
\end{remark}

\subsection{Wrapped Floer cohomology and the knot complement}\label{ssec:direct}

Here we describe a direct connection between partially wrapped Floer cohomology of $\B$ and $\C$ in $W_K$ and the topology of the loop space of the knot complement, which explains the geometry underlying the isomorphisms in Proposition \ref{prop:triple-isom} and the ring isomorphism $HW^0(\B,\B) \cong \Z[\pi_1(\R^3 \setminus K)]$.

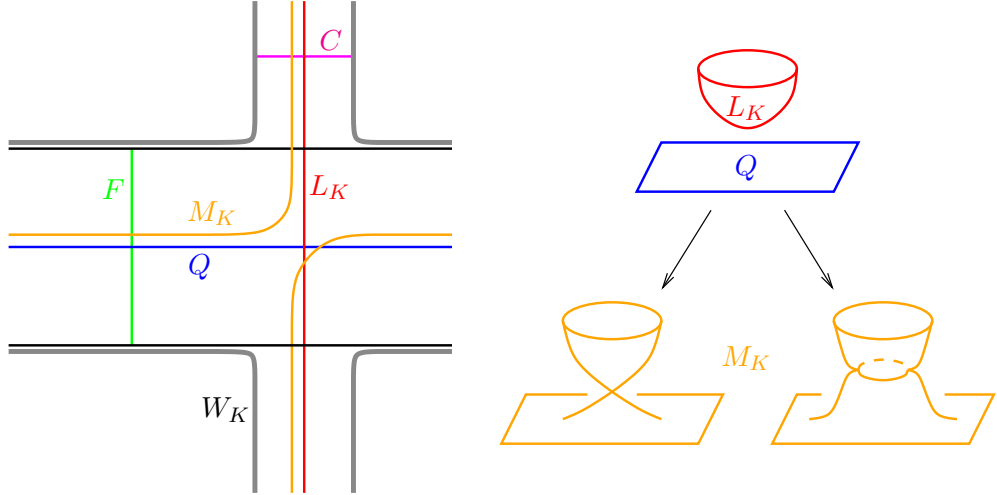
\begin{figure}
\labellist
\small\hair 2pt
\pinlabel ${\color{magenta} \C}$ at 239 335
\pinlabel ${\color{green} \B}$ at  80 226
\pinlabel ${\color{red} L_K}$ at 238 227
\pinlabel ${\color{blue} Q}$ at 143 169
\pinlabel ${\color{orange} M_K}$ at 152 208
\pinlabel $W_K$ at 162 64
\pinlabel ${\color{red} L_K}$ at 543 285
\pinlabel ${\color{blue} Q}$ at 543 241
\pinlabel ${\color{orange} M_K}$ at 543 102
\endlabellist
\centering
\includegraphics[width=0.8\textwidth]{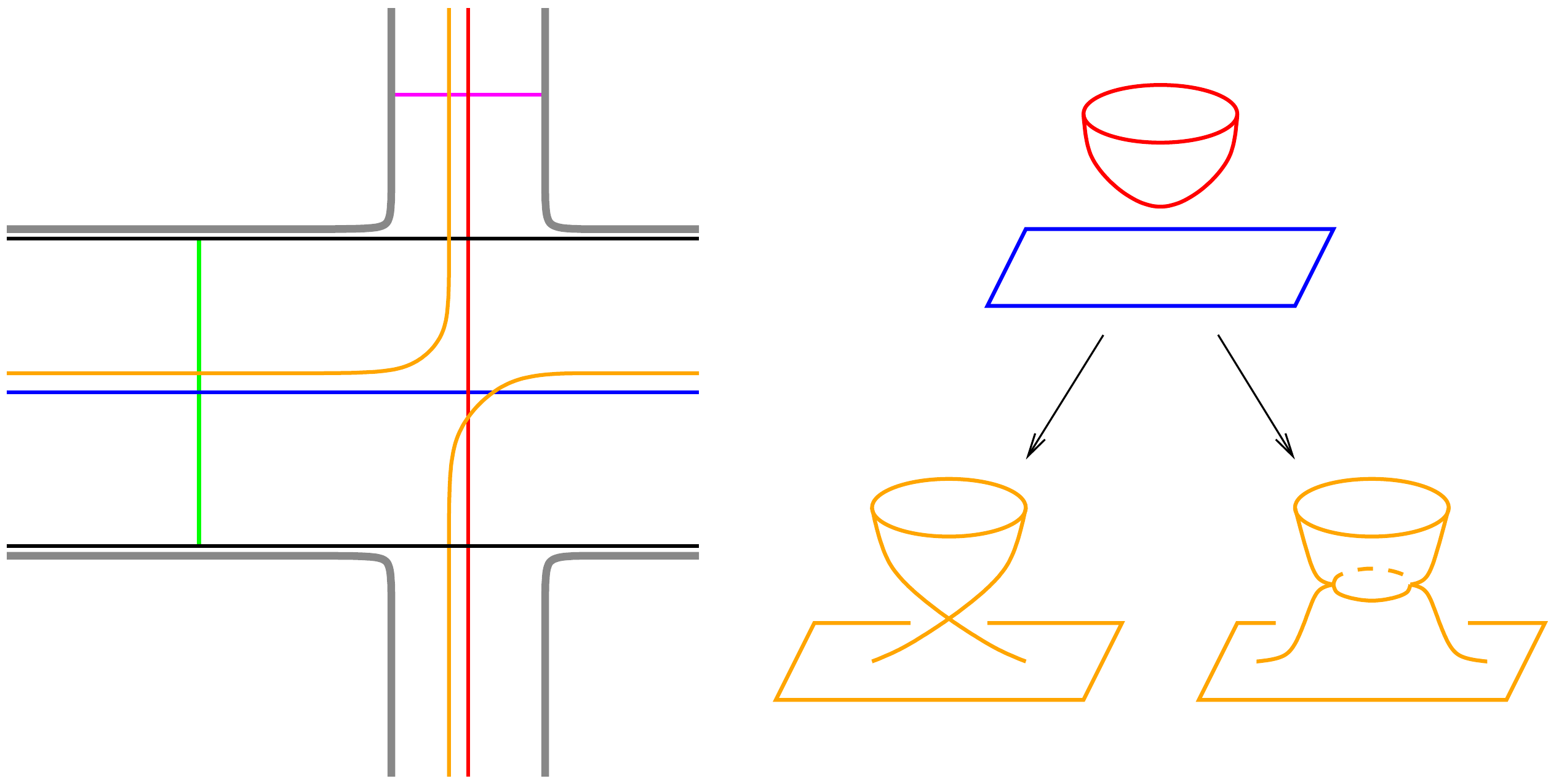}
\caption{
Constructing the Lagrangian $M_K$ from the Lagrangians $Q$ and $L_K$: on the left, depicting $M_K$ within $W_K$; on the right, the two possible Lagrange surgeries.}
\label{fig:wrapped-MK}
\end{figure}

As a first step, following \cite{AENV}, one can construct a Lagrangian submanifold $M_K\subset W_K$ that is diffeomorphic to the knot complement $M_{K}\approx \R^{3}\setminus K$, see Figure~\ref{fig:wrapped-MK}. This is the result of applying Lagrange surgery to $Q \cup L_K$ along the clean intersection $Q \cap L_K = K$.
In fact there are two distinct Lagrange surgeries: in the fiber over a point in $K$ we see two transversely meeting Lagrangian $2$-disks and we can smooth their intersection in two ways, as shown in the right diagram of Figure~\ref{fig:wrapped-MK}. (If the two disks are a local model for the double point of the Whitney sphere in $\R^{4}$, then one smoothing gives the Clifford torus and the other the Chekanov torus.)
Note that $M_K$ intersects each of $\B$ and $\C$ transversely in one point.

In Section~\ref{ssec:wrapped-WK}, we described chain isomorphisms $\Phi$ from the partially wrapped Floer homologies of $\C$ and $\B$ to the DGA $\A_{\Lambda_K \cup \Lambda_p}$. Here we outline geometric counterparts of $\Phi$, mapping into chains of paths on $M_K$ rather than to $\A_{\Lambda_K \cup \Lambda_p}$.

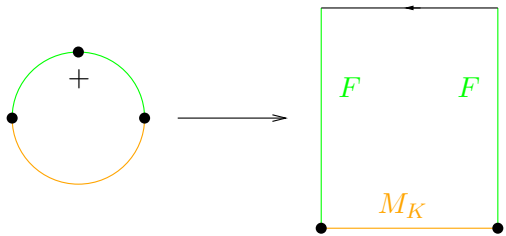
\begin{figure}
\labellist
\small\hair 2pt
\pinlabel $+$ at 59 127
\pinlabel ${\color{green} \B}$ at  280 120
\pinlabel ${\color{green} \B}$ at  377 120
\pinlabel ${\color{orange} M_K}$ at 323 23
\endlabellist
\centering
\includegraphics[width=0.4\textwidth]{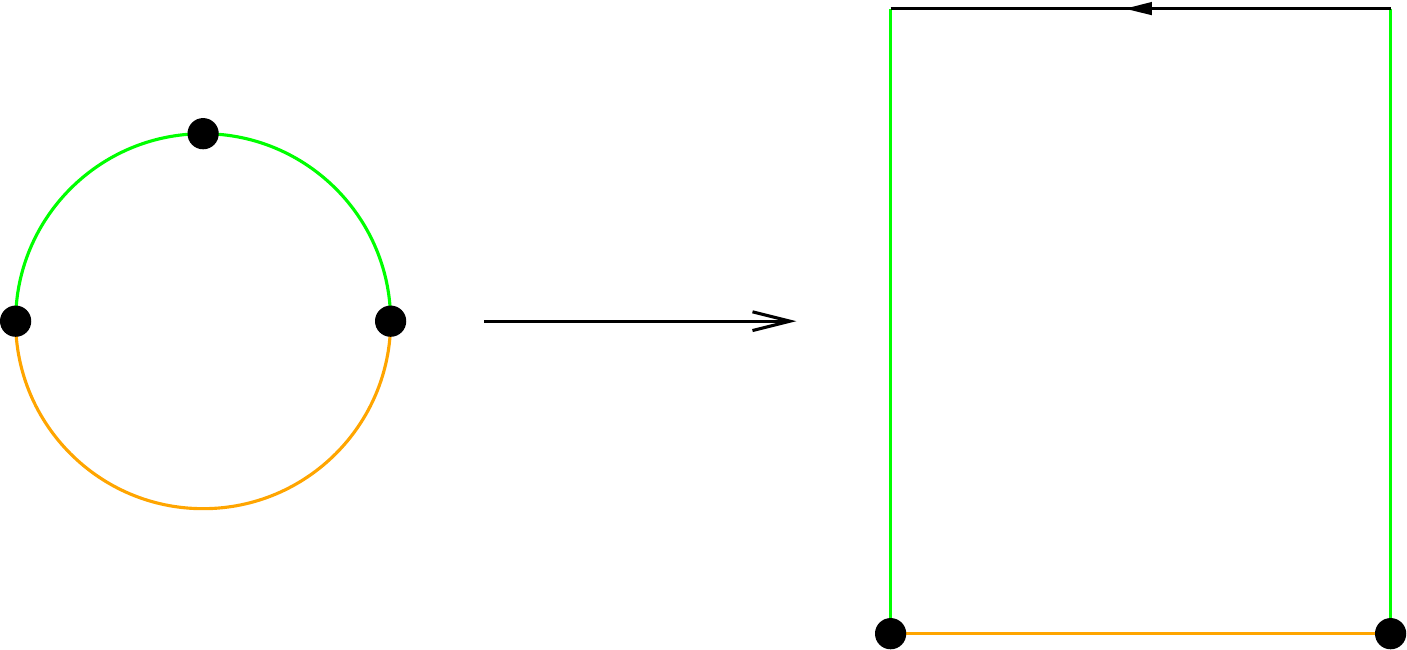}
\caption{
A holomorphic disk with boundary on $\B \cup M_K$, used to define a map from Reeb chords of $\partial \B$ to chains in the loop space of $M_K$.}
\label{fig:wrapped-isom-MK}
\end{figure}

We can construct a map from $CW(\B,\B)$ into chains of based loops in $M_K$ by associating to a Reeb chord of $\B$ the chain of based loops in $M_K$ (based at the point $\B \cap M_K$) carried by the moduli space of holomorphic disks with one positive puncture at the Reeb chord, two Lagrangian intersection punctures at $\B\cap M_K$, and boundary on $\B\cup M_K$, see Figure~\ref{fig:wrapped-isom-MK}. Similarly we have maps from $CW(\C,\C)$, $CW(\C,\B)$, and $CW(\B,\C)$ to chains of loops in $M_K$ (for the first one) or paths in $M_K$ (for the last two) by replacing the $\B$ boundary conditions in Figure~\ref{fig:wrapped-isom-MK} by $\C$ as needed. All four maps are chain maps and induce maps on homology.

Now Theorem~6.15 in \cite{AENV} relates holomorphic disks with positive puncture at a Reeb chord of $\Lambda_K$ and switching boundary conditions on $Q \cup L_K$, via a smoothing procedure, to disks with the same positive puncture but now with boundary on $M_K$. We can combine this with the surgery isomorphism relating Reeb chords of $\partial \C \cup \partial \B$ (the generators of the partially wrapped complexes) to words of Reeb chords of $\Lambda_K \cup \Lambda_p$.
The geometric maps from $CW(\C,\C)$, $CW(\C,\B)$, and $CW(\B,\C)$ to chains in loop/path spaces then descend in homology to maps from $R_{KK} = HW^0(\C,\C)$, $R_{Kp} = HW^0(\C,\B)$, and $R_{pK} = HW^0(\B,\C)$ to $\Z[\pi_1(M_K)]$. These are precisely the isomorphisms in Proposition~\ref{prop:triple-isom} from the KCH-triple to subrings of $\Z[\pi_1(\R^3\setminus K)] = \Z[\pi_1(M_K)]$; furthermore, the map on $CW(\B,\B)$ induces the isomorphism from $\Z\oplus R_{pp} \cong HW^0(\B,\B)$ to the entire group ring from Proposition~\ref{prop:ring-isom}.

Indeed, the algebraic maps constructed in the proof of Proposition~\ref{prop:triple-isom} correspond to the maps from chains of broken strings to chains of paths on $M_K$ determined by \cite[Theorem~6.5]{AENV}. More precisely, the algebraic maps correspond to the disks for which the length functional on the image chain is almost equal to the action of the Reeb chord at the positive puncture. Using the action/length filtration as in \cite{CELN} we then find that the algebraic maps are homotopic to the geometric maps.

The factors of $1-m$ in Proposition~\ref{prop:triple-isom} come from the fact that a holomorphic disk with switching boundary at $Q \cap L_K$ can be smoothed to have boundary on $M_K$ in two different ways, differing in homotopy by a meridian in $M_K$. In this context, we note that the choice of where to place the $1-m$ factors in the isomorphisms from Proposition~\ref{prop:triple-isom}, see Remark~\ref{r:choiceofsmoothing}, corresponds to the choice of Lagrange surgery (Clifford or Chekanov) along $K$.

The geometric chain maps that we have described here have direct analogues for the wrapped Floer cohomology of a cotangent fiber in a closed manifold \cite{AS,Ab}. In particular, the identification of the pair-of-pants product in wrapped Floer cohomology of the fiber in a closed manifold with the Pontryagin product on chains of loops is directly analogous to our identification of the products in $HW^0(\B,\B)$ and $\Z[\pi_{1}(\R^{3}\setminus K)]$.

\subsection{Relation to sheaves and cosheaves}
\label{ssec:sheaves}
This paper has presented a holomorphic-curve approach to the result
that the conormal torus is a complete knot invariant.  This result
was previously established by the methods of microlocal sheaf
theory in \cite{Shende}.  We now explain how these two approaches are
related.

Consider a closed manifold $M$.  Abouzaid has shown on the one hand that
the wrapped Floer cohomology of a cotangent fiber in $T^*M$ is naturally
identified with chains on the based loop space of $M$ \cite{Ab}, and on the other
hand that this cotangent fiber generates the wrapped Fukaya category \cite{abouzaid2011cotangent}.
One way to express these facts simultaneously is to say that the wrapped
Fukaya category of $T^*M$ localizes to a cosheaf of categories over $M$.

To generalize this notion to the cotangent bundle of an open manifold, and
in particular to make sense of what the sections of the above cosheaf should be
over a smaller open set in symplectic-geometric terms, requires considerations
of ambient symplectic spaces more general than standard Liouville domains because of the following.
For a manifold $M$ with non-empty boundary $\partial M$,
the standard Liouville structure on the cotangent bundle (with Liouville vector field $p\cdot\partial_{p}$ pointing radially outwards along fibers)
does not have contact type boundary.  This problem can be resolved in more or less
equivalent ways by either allowing such ``Liouville manifolds with boundary'' as ambient symplectic manifolds (essentially the approach we have taken here), or by taking the more
conventional Liouville structure on $T^*M$ (obtained by rounding the corners over $T^{\ast}_{\partial M}M$) but remembering $\partial M$ as a Legendrian submanifold in the contact boundary and construct
the wrapped Floer theory in such a way that this Legendrian serves to stop the wrapping. We point out that it does not suffice to take the conventional Liouville structure on $T^* M$ alone and forget the Legendrian: with this structure $T^{\ast}M$ is subcritical, the wrapped Fukaya category is trivial, and the Floer cohomology of a cotangent fiber vanishes and can in particular not be identified with chains on the based loop space as in the closed case, cf.\ Remark~\ref{rmk:Rn}.

The setup for cotangent bundles of manifolds with boundary applies in more general situations. Consider a Weinstein manifold $W$ with ideal contact boundary $V$ that contains a possibly singular Legendrian $\Lambda\subset V$. We write $\mathbb{L}\subset W$ for the Lagrangian skeleton, by which in this setting we mean the complement
of all points in $W$ which escape to $V \setminus \Lambda$ under the Liouville flow.
In this context, it has long been expected \cite{kontsevich2009symplectic}
that the Fukaya category defined by stopping wrapping at $\Lambda$ will localize on $\mathbb{L}$.
In the special case that $\mathbb{L}$ is locally a conormal variety
(results of Nadler \cite{Nad} suggest that this can always be achieved after perturbing the
Liouville structure), it is further expected that the resulting cosheaf
is dual to the Kashiwara--Schapira stack which governs microlocal sheaves
\cite{nadler2014fukaya}.  Results along these lines will
soon appear \cite{GPS1, GPS2}, where it will also be shown that in good
circumstances, the wrapped Fukaya category is again
generated by fibers, one in each connected component
of the smooth locus of the skeleton.

We now return to the specific case studied here. In \cite{Shende}, the conormal
torus $\Lambda_K$ of a knot $K\subset \R^{3}$ was studied via the category of sheaves
microsupported in the union of the zero section and the cone over $\Lambda_K$.
This sheaf of categories is known by \cite{kashiwara2013sheaves} to
be the global sections of a sheaf of categories
over the skeleton formed by the union of the zero section and the conormal.
These in turn are modules over the global sections of a corresponding cosheaf that is locally constant away from the singular locus.
It then follows formally that the global sections are generated by
the stalk in each connected component of the smooth locus and the stalk
at the singular locus. The stalk at the smooth locus is just the category of
perfect complexes and a calculation at a singular point of the skeleton
shows that in fact the stalks at the smooth points suffice.

By the above discussion, the sheaf category studied in \cite{Shende} is the category of modules over the wrapped Fukaya category of $T^{\ast}\R^{3}$ with wrapping stopped by $\partial \R^3 \cup \Lambda_K$,
which is generated by two cotangent fibers, one at a point on the
zero section and one at a point on the of the Lagrangian conormal of $K$. These generators are the Lagrangian disks
$\B$ and $\C$ discussed above.

The surgery formula calculates the partially wrapped Floer cohomology (along with the pair-of-pants product) for these Lagrangians $\B$ and $\C$, which generate the Fukaya category of the neighborhood of the Lagrangian skeleton, from holomorphic-curve considerations on the Legendrian attaching locus. In light of the above discussion, the constructions of the previous sections then have a natural interpretation as calculating morphisms in a category equivalent to that studied in \cite{Shende}. Note however that
this connection is not explicitly used in the proof presented here that these morphism spaces determine a complete knot invariant. Rather, in the proof of our main theorem we map the holomorphic curve theories to the string topology of the Lagrangian skeleton. Hence, in a sense, the holomorphic curve theory of the Legendrian conormal allows us to see rather concretely how the Fukaya category of $W_K$ localizes to its Lagrangian skeleton in the case of conormals of knots.

\bibliographystyle{alpha}
\bibliography{biblio}

\end{document}